\documentclass[preprint]{elsarticle}

\usepackage[utf8x]{inputenc}
\usepackage[english]{babel}

\usepackage{fleqn}
\usepackage{geometry}
\usepackage{tikz}
\usepackage{tikz-cd}
\usetikzlibrary{matrix}
\usepackage{mathrsfs}
\usepackage{natbib}
\usepackage{ifpdf}
\usepackage{amsmath}
\usepackage{amsthm, amssymb}
\usepackage{hyperref}
\usepackage{float}
\usepackage{color}
\usepackage[all,cmtip]{xy}

\hypersetup{
	unicode=true,
	colorlinks=true,
	citecolor=black,
	linkcolor=black,
	anchorcolor=black
}
\usepackage{upgreek}

\usepackage{aliascnt}
\numberwithin{equation}{section}

\newtheorem{mainthm}{Theorem}

\newtheorem{theo}{Theorem}[subsection]	
\newtheorem*{theo*}{Theorem}

\newaliascnt{lem}{theo}
\newtheorem{lem}[lem]{Lemma}
\aliascntresetthe{lem}

\newaliascnt{propo}{theo}
\newtheorem{propo}[propo]{Proposition}
\aliascntresetthe{propo}

\newaliascnt{corol}{theo}
\newtheorem{corol}[corol]{Corollary}
\aliascntresetthe{corol}

\addto\extrasenglish{

}

\theoremstyle{remark}
\newaliascnt{rem}{theo}
\newtheorem{rem}[rem]{Remark}
\aliascntresetthe{rem}

\newaliascnt{exam}{theo}
\newtheorem{exam}[exam]{Example}
\aliascntresetthe{exam}

\theoremstyle{definition}
\newaliascnt{defi}{theo}
\newtheorem{defi}[defi]{Definition}
\aliascntresetthe{defi}

\newaliascnt{nota}{theo}
\newtheorem{nota}[nota]{Notation}
\aliascntresetthe{nota}

\newcommand{\diri}{\mfr{D}}
\newcommand{\nil}{\Upsilon}
\newcommand{\spring}[1]{\kk\langle#1\rangle}
\newcommand{\cma}{\Xi}
\newcommand{\invv}{{\flat}}
\newcommand{\inv}{\star}
\newcommand{\A}{S}
\newcommand{\B}{T}

\newcommand{\BB}{{\mbf{J}}}
\newcommand{\bkk}{{\kk^\alg}}
\newcommand{\alg}{{\mrm{alg}}}
\newcommand{\pd}[1]{\widehat{#1}}
\newcommand{\GGamma}{\mbf{\Gamma}}
\newcommand{\ggamma}{\upgamma}
\newcommand{\Lie}{\mrm{Lie}}

\newcommand{\mfr}[1]{\mathfrak{#1}}
\newcommand{\mbf}[1]{\mathbf{#1}}
\newcommand{\msf}[1]{\mathsf{#1}}
\newcommand{\mrm}[1]{\mathrm{#1}}
\newcommand{\mcal}[1]{\mathcal{#1}}

\newcommand{\set}[1]{\left\{{#1}\right\}}

\newcommand{\abs}[1]{\left|#1\right|}

\newcommand{\dbN}{\mathbb N}
\newcommand{\dbZ}{\mathbb Z}
\newcommand{\dbF}{\mathbb F}

\newcommand{\dbC}{\mathbb C}
\newcommand{\dbA}{\mathbb A}

\newcommand{\sym}{\mathrm{Sym}}

\newcommand{\GL}{\mrm{GL}}
\newcommand{\matr}{\mrm{M}}
\newcommand{\SL}{\mrm{SL}}
\newcommand{\UU}{\mrm{U}}

\newcommand{\SO}{\mrm{SO}}
\newcommand{\Sp}{\mrm{Sp}}
\newcommand{\gl}{\mfr{gl}}

\newcommand{\HH}{\mbf{H}}
\newcommand{\GG}{\mbf{G}}
\newcommand{\g}{{\mathbf{g}}}

\newcommand{\pp}{\mfr{p}}
\newcommand{\oo}{\mfr{o}}
\newcommand{\PP}{\mfr{P}}
\newcommand{\OO}{\mfr{O}}
\newcommand{\kk}{\msf{k}}

\newcommand{\Gal}{\mathrm{Gal}}
\newcommand{\TT}{\mbf{T}}
\newcommand{\CC}{\mbf{C}}

\newcommand{\reg}{\mathrm{reg}}

\newcommand{\mat}[1]{\left(\begin{matrix}#1\end{matrix}\right)}

\newcommand{\smat}[1]{\left(\begin{smallmatrix}#1\end{smallmatrix}\right)}

\newcommand{\ZZ}{\mathbf Z}

\renewcommand{\ker}{\mathrm{Ker}}

\newcommand{\irr}{\mathrm{Irr}}
\renewcommand{\hom}{\mathrm{Hom}}
\newcommand{\End}{\mathrm{End}}
\newcommand{\aut}{\mathrm{Aut}}

\newcommand{\Span}{\textstyle\mathop{Span}}
\newcommand{\spec}{\mathrm{Spec\:}}

\newcommand{\im}{\mathrm{Im}}
\newcommand{\Tr}{\mathrm{Tr}}
\newcommand{\Nr}{\mathrm{Nr}}

\newcommand{\cay}{\mathrm{cay}}
\newcommand{\diag}{\mathrm{diag}}

\newcommand{\ttau}{\boldsymbol{\tau}}

\newcommand{\llog}{\mathrm{log}}

\newcommand{\ad}{\mathrm{ad}}
\newcommand{\Ad}{\mathrm{Ad}}
\newcommand{\Char}{{\mathrm{char}}}

\newcommand{\rk}{\mathrm{rk}}

\renewcommand{\setminus}{\smallsetminus}

\newcommand{\id}{\mathbf{1}}

\begin{document}
\title{Regular characters of classical groups over complete discrete valuation rings}

\author{Shai Shechter}
\ead{shais@post.bgu.ac.il}
\address{Department of Mathematics\\ Ben Gurion University of the Negev\\ Beer-Sheva 84105\\Israel}

\begin{abstract}

Let $\oo$ be a complete discrete valuation ring with finite residue field $\kk$ of odd characteristic, and let $\GG$ be a symplectic or special orthogonal group scheme over $\oo$. For any $\ell\in\dbN$ let $G^\ell$ denote the $\ell$-th principal congruence subgroup of $\GG(\oo)$. An irreducible character of the group $\GG(\oo)$ is said to be \textbf{regular }if it is trivial on a subgroup $G^{\ell+1}$ for some $\ell$, and if its restriction to $G^\ell/G^{\ell+1}\simeq \Lie(\GG)(\kk)$ consists of characters of minimal $\GG(\bkk)$-stabilizer dimension
. In the present paper we consider the regular characters of such classical groups over $\oo$, and construct and enumerate all regular characters of $\GG(\oo)$, when the characteristic of $\kk$ is greater than two. As a result, we compute the regular part of their representation zeta function.
\end{abstract}
\maketitle 
\section{Introduction}\label{section:introduction}

Let $K$ be a non-archimedean local field, and let $\oo$ be its valuation ring, with maximal ideal $\pp$ and finite residue field $\kk$ of odd characteristic. Let $q$ and $p$ denote the cardinality and characteristic of $\kk$, respectively. Fix $\pi$ to be a uniformizer of $\oo$. Let $\GG\subseteq\SL_N$ be a symplectic or a special orthogonal group scheme over $\oo$, i.e.~the group of automorphisms of determinant $1$, preserving a fixed non-degenerate anti-symmetric or symmetric $\oo$-defined bilinear form. In this article we study the set of irreducible regular characters of the group of $\oo$-points $G=\GG(\oo)$, the definition of which we now present.

\subsection{The basic definitions}\label{subsection:definitions}

Let $\irr(G)$ denote the set of irreducible complex characters of $G$ which afford a continuous representation with respect to the profinite topology. The \textbf{level} of a character $\chi\in \irr(G)$ is the minimal number $\ell\in\dbN_0=\dbN\cup\set{0}$ such that the restriction of any representation associated to $\chi$ to the principal congruence subgroup $G^{\ell+1}=\ker\left(G\to\GG(\oo/\pp^{\ell+1})\right)$ is trivial. For example, the set of characters of level~$0$ is naturally identified with the set of irreducible complex characters of $\GG(\kk)$. 

\subsubsection{The residual orbit of a character}\label{subsubsection:res-orbit}
Let $\g=\Lie(\GG)\subseteq\mfr{gl}_N$ denote the Lie algebra scheme of $\GG$. The smoothness of $\GG$ implies the equality $G/G^{\ell+1}=\GG(\oo/\pp^{\ell+1})$, and moreover, the existence of an isomorphism of abelian groups between $\g(\kk)$ and the quotient group $G^\ell/G^{\ell+1}$, for any $\ell\ge 1$ (see \cite[II,~\S 4,~no.~3]{DemazureGabriel}). In the notation of \cite{DemazureGabriel}, this isomorphism is denoted $x\mapsto e^{\pi^{\ell}x}$. The action of $G$ by conjugation on the quotient $G^\ell/G^{\ell+1}$ factors through its quotient $\GG(\kk)$, making the isomorphism above $\GG(\kk)$-equivariant, with respect to the action given by $\Ad\circ \alpha_\ell$, where $\Ad$ denotes the adjoint action of $\GG(\kk)$ on $\g(\kk)$, and $\alpha_\ell:\GG(\kk)\to\GG(\kk)$ is a bijective endomorphism of $\GG(\kk)$, determined by a field automorphism of $\kk$ (see, e.g.~\cite[Lemma~3]{McNinch-Faithful} and the reference therein). Additionally, by the assumption $\Char(\kk)\ne 2$ and \cite[I,\:Lemma~5.3]{SteinbergSpringer}, the underlying additive group of $\g(\kk)$ can be naturally identified with its Pontryagin dual in a $\GG(\kk)$-equivariant manner. Consequently, there exists an isomorphism of $\GG(\kk)$-spaces 
\begin{equation}
\label{equation:identification}\g(\kk)\xrightarrow{\sim}\irr\left(G^\ell/G^{\ell+1}\right).
\end{equation}

Let $\chi\in\irr(G)$ have level $\ell>0$. Consider the restriction $\chi_{G^\ell}$ of $\chi$ to $G^\ell$. By Clifford's Theorem and the definition of level, the restricted character $\chi_{G^\ell}$ is equal to a multiple of the sum over a single $\GG(\kk)$-orbit of characters of $G^{\ell}/G^{\ell+1}$. Using \eqref{equation:identification}, this orbit corresponds to a single $\GG(\kk)$-orbit in $\g(\kk)$, which we call the\textbf{ residual orbit of $\chi$}, and denote $\Omega_1(\chi)\in~\Ad\circ\alpha_\ell(\GG(\kk))\backslash\g(\kk)=\Ad\left(\GG(\kk)\right)\backslash \g(\kk)$.

\subsubsection{Regular characters}\label{subsubsection:reg-chars}
Let $\bkk$ be a fixed algebraic closure of $\kk$. An element of $\g(\bkk)$ is said to be \textbf{regular} if its centralizer in $\GG(\bkk)$ has minimal dimension among such centralizers (cf. \cite[\S~3.5]{SteinbergConjugacy}). By extension, an element of $\g(\kk)$ is said to be regular if its image under the natural inclusion of $\g(\kk)$ into $\g(\bkk)$ is regular. 

\begin{defi}[Regular Characters]\label{defi:regularity} A character $\chi\in\irr(G)$ of positive level is said to be \textbf{regular} if its residual orbit $\Omega_1(\chi)$ consists of regular elements of $\g(\kk)$.
\end{defi}

For a general overview of regular elements in reductive algebraic groups over algebraically closed fields, we refer to \cite[Ch.~III]{SteinbergSpringer}. The definition of regular characters  goes back to Shintani \cite{Shintani} and Hill \cite{Hill}. An overview of the history of regular characters of $\GL_N(\oo)$ can be found in \cite{StasinskiOverview}. Also- see \cite{KOS,StasinskiStevens} and \cite{TakaseGLn} for the analysis of regular characters of isotropic groups of type $\msf{A}_n$, as well as \cite{Shechter}, for a partial treatment of anisotropic groups of type $\msf{A}_n$.

\subsection{Regular elements and regular characters}

Following \cite{Hill}, we begin our investigation of regular characters with the study of regular elements in the finite Lie rings $\g(\oo_r)$, where $\oo_r=\oo/\pp^r$ (see \autoref{defi:regular-elements}). 

A central feature of the analysis undertaken in \cite{Hill} is the introduction and application of geometric methods to the study of regular characters. Given $x\in\matr_N(\oo)$ and $r\in\dbN$, let $x_r$ denote the image of $x$ in $\matr_N(\oo_r)$ under the reduction map. The condition of commuting with $x_r$ defines a closed $\oo_r$-group subscheme of the fiber product\footnote{The notation $\GG\times\oo_r$ is shorthand for the fiber product $\GG\times_{\spec{\oo}}\spec{\oo_r}$. Similar notation is used whenever the base change being performed is between spectra of rings, and the base ring of the given schemes is understood from context.} $\GL_N\times\oo_r$, which, upon application of the Greenberg functor, defines a $\kk$-group scheme \cite[\S~4,\: Main Theorem.(5)]{GreenbergI}. The element $x_r$ is said to be \textbf{regular} if the group scheme thus obtained is of minimal dimension among such group schemes (see \cite[Definition~3.2]{Hill}). In \cite[Theorem~3.6]{Hill}, Hill proved that $x_r\in\matr_n(\oo_r)$ is regular if and only if its image $x_1\in\matr_n(\kk)$ is regular. Additionally, regularity of $x_r$ was shown to be equivalent to the cyclicity of the module $\oo_r^N$ over the ring $\oo_r[x_r]\subseteq\matr_N(\oo_r)$. We note that Hill's definition of regularity is equivalent to Shintani's definition of \textit{quasi-regularity} \cite[\S~2]{Shintani}.

The equivalence of regularity over the ring $\oo_r$ and over $\kk$ was recently extended to general semisimple groups of type $\msf{A}_n$ in \cite{KOS}. In \autoref{subsection:reg-characters} we further extend this equivalence of to the generality of classical groups of type $\msf{B}_n,\:\msf{C}_n$ and $\msf{D}_n$ in odd characteristic. However, the equivalence of regularity of an element $x_r\in \g(\oo_r)$ with the cyclicity of the module $\oo_r^N$ over $\oo_r[x_r]$, while true in $\GL_N$ and generically true in $\GG$ (see \autoref{corol:reg-g1-reg-glN-nonsing}), is not a general phenomenon and in fact fails in certain cases (see \autoref{lem:nilpotent-reg-glN-not-reg-soN}). Nevertheless, in the present setting, it is possible to prove a supplementary result (\autoref{theo:inverse-lim}), which specializes to the above equivalence in the case of $\GG=\GL_N$,  and provides us with the information needed in order to describe the inertia subgroup of such a character and enumerate the characters of $G$ lying above a given regular orbit. Consequently, we deduce the first main result of this article.
\begin{mainthm}\label{mainthm:enumeration+dimension} Let $\oo$ be a discrete valuation ring with finite residue field of odd characteristic, and let $\GG$ be a symplectic or a special orthogonal group over $\oo$ with generic fiber of absolute rank $n$. Let $\Omega\subseteq\g(\kk)$ be an $\Ad(\GG(\kk))$-orbit consisting of regular elements and let $\ell\in\dbN$.
\begin{enumerate}
\item The number of regular characters $\chi\in\irr(G)$ of level $\ell$ whose residual orbit is equal to $\Omega$ is $\frac{\abs{\GG(\kk)}}{\abs{\Omega}}\cdot q^{(\ell-1)n}$.
\item Any such character has degree $\abs{\Omega}\cdot q^{(\ell-1)\alpha},$ where $\alpha=\frac{\dim\GG-n}{2}$.
\end{enumerate}
\end{mainthm}

\subsection{Regular representation zeta functions}\label{subsection:reg-zeta} Taking the perspective of representation growth, given a group $H$, one is often interested in understanding the asymptotic behaviour of the sequence $\set{r_m(H)}_{m=1}^\infty$, where $r_m(H)\in\dbN\cup \set{0,\infty}$ denotes the number of elements of $\irr(H)$ of degree $m$. In the case where the sequence $r_m(H)$ is bounded above by a polynomial in $m$, the representation zeta function of $H$ is defined to be the Dirichlet generating function
\begin{equation}
\label{equation:rep-zeta}\zeta_H(s)=\sum_{m=1}^\infty r_m(H)m^{-s},\quad(s\in\dbC).
\end{equation}

In the specific case $H=G=\GG(\oo)$, one may initially restrict to a description of the regular representation zeta function, i.e. the Dirichlet function counting only regular characters of $G$. In this respect, \autoref{mainthm:enumeration+dimension} implies that the rate of growth of regular characters of $G$ is polynomial of degree $\frac{2n}{\dim\GG-n}$. Furthermore, 
we obtain the following.
\begin{corol}\label{corol:reg-zeta-function} Let $X\subseteq \Ad(\GG(\kk))\backslash \g(\kk)$ denote the set of orbits consisting of regular elements, and let \begin{equation}\label{equation:orbit-dirichlet-poly}
\diri_{\g(\oo)}(s)=\sum_{\Omega\in X}
{\abs{\GG(\kk)}}\cdot \abs{\Omega}^{-(s+1)},\quad(s\in\dbC).
\end{equation}

The regular zeta function of $G=\GG(\oo)$ is of the form
\begin{equation*}
\label{equation:reg-rep-zeta}\zeta^\reg_G(s)=\frac{\diri_{\g(\oo)}(s)}{1-q^{n-\alpha s}}
\end{equation*}
where $n$ and $\alpha$ are as in \autoref{mainthm:enumeration+dimension}. \end{corol}

\subsection{Classification of regular orbits in $\g(\kk)$}
The second goal of this article is to compute the regular representation zeta function of the symplectic and special orthogonal groups over~$\oo$. In view of \autoref{corol:reg-zeta-function}, to do so, one must classify and enumerate the regular orbits in $\g(\kk)$, under the adjoint action of $\GG(\kk)$. This classification is undertaken in \autoref{section:classical-groups}, and its consequences are summarized in \autoref{theo:orbits-sp2nso2n+1} and \autoref{theo:orbits-so2n}. The classification of regular adjoint classes in the $\g(\kk)$ is closely related to the question of classifying conjugacy classes in classical groups over a finite field, a question which was solved in complete generality by Wall in \cite[\S~2.6]{WallIsometries}. Taking a enumerative prespective, the regular semisimple conjugacy classes in finite classical groups were enumerated in \cite{FulmanGuralnick}, using generating functions. Another closely related question is that of enumerating \textit{cyclic} conjugacy classes in finite classical groups. This question is addressed in \cite{FNP-Generating, NeumannPraeger-Classical}; see \autoref{subsection:summary-classical} for further discussion. The enumeration carried out in this paper yields uniform formulae for the function $\diri_{\g(\oo)}$ (and, consequently, for the regular representation zeta function) of each of the classical groups in question, which are independent of the cardinality of $\kk$.

Given $n\in\dbN$ let $\mcal{X}_n$ denote the set of triplets $\ttau=(r,\A,\B)$, in which $r\in\dbN_0$ and $\A=~(\A_{d,e})$ and $\B=~(\B_{d,e})$ are $n\times n$ matrices with non-negative integer entires, satisfying the condition \begin{equation}\label{equation:t-condition}
r+\sum_{d,e=1}^n de\left(\A_{d,e}+\B_{d,e}\right)=n.\end{equation}

Given $\ttau=(r,S,T)\in\mcal{X}_n$, define the following polynomial in $\dbZ[t]$
\begin{align}
c^{\ttau}(t)&=t^n\prod_{d,e}(1+t^{-d})^{S_{d,e}}(1-t^{-d})^{T_{d,e}}\label{equation:v-u-tau}
\end{align}
and let $u_1(q)=\abs{\Sp_{2n}(\kk)}=\abs{\SO_{2n+1}(\kk)}$. Note that the value $u_1(q)$ is given by evaluation at $t=q$ of a  polynomial $u_1(t)\in\dbZ[t]$, which is independent of $q$ (see, e.g., \cite[\S~3.5 and \S~3.2.7]{WilsonFinite}). Additionally, for any $\ttau\in \mcal{X}_n$, let $M_{\ttau}(q)$ denote the number of polynomials of type $\ttau$ over a field of $q$ elements; see \autoref{defi:poly-type}. An explicit formula for $M_{\ttau}(q)$ is computed in \autoref{subsubsection:enumerative}. We remark that the value of $M_{\ttau}(q)$ is given by evaluation at $t=q$ of a uniform polynomial formula which is independent of $q$ as well; see~\eqref{equation:M_tau}.

\begin{mainthm}\label{mainthm:dirichlet-polnomials-sp-odd-orth} Let $\oo$ be a complete discrete valuation ring of odd residual characteristic. Let $n\in\dbN$ and $\GG$ be one of the $\oo$-defined algebraic group schemes $\Sp_{2n}$ or $\SO_{2n+1}$, with $\g=\Lie(\GG)$.  

Given $\ttau=(r,S,T)\in\mcal{X}_n$ let \[\nu(\ttau)=\nu_{\GG}(\ttau)=\begin{cases}1&\text{if }\GG=\Sp_{2n}\text{ and }r>0,\\0&\text{otherwise}.\end{cases}\]

The Dirichlet polynomial $\diri_{\g(\oo)}(s)$ (see \eqref{equation:orbit-dirichlet-poly}) is given by

\begin{equation}\label{equation:diri-sp2n}
\diri_{\g(\oo)}(s)=\sum_{\ttau\in \mcal{X}_n} 4^{\nu(\ttau)} M_{\ttau}(q)\cdot  c^{\ttau}(q)\cdot \left(\frac{ u_1(q)}{2^{\nu(\ttau)}c^{\ttau}(q)}\right)^{-s}.
\end{equation}
\end{mainthm}

Recall that a symmetric bilinear form over a finite field of odd characteristic is determined by the \textit{Witt index} of the form, i.e. the dimension of a maximal totally isotropic subspace with respect to the form. Following standard notation, we write $\SO_{2n}^+$ and $\SO_{2n}^-$ to the group schemes whose group of $\kk$-points are associated with a symmetric bilinear form of Witt index $n$ and $n-1$ respectively. Also, for convenience, we often use the notation $\SO_{2n}^{\pm 1}$ for $\SO_{2n}^\pm$. 

Given $\epsilon\in\set{\pm 1}$, let $u_2^\epsilon(q)=\abs{\SO_{2n}^\epsilon(\kk)}$. As in the previous case, note that the value $u_2^{\epsilon}(q)$ is given by evaluation at $t=q$ of a polynomial $u_2^\epsilon(t)\in\dbZ[t]$, which is independent of $q$ (see \cite[\S~3.2.7]{WilsonFinite})

\begin{mainthm}\label{mainthm:dirichlet-polnomials-even-orth} Let $\oo$ be a complete discrete valuation ring of odd residual characteristic and whose residue field has more than $3$ elements. Let $n\in\dbN$ and $\epsilon \in\set{\pm 1}$. Let $\GG^\epsilon=\SO_{2n}^\epsilon$ be the $\oo$-defined special orthogonal group scheme, as described above, and let $\g^\epsilon=\Lie(\GG^\epsilon)$. 

Let $\mcal{X}_n^{0}$ denote the set of triplets $\ttau=(r,\A,\B)\in\mcal{X}_n$ with $r=0$, and let $\mcal{X}_n^{0,+1}$ denote the subset of $\mcal{X}_n^0$ consisting of elements $(0,\A,\B)$ such that $\sum_{d,e}e\A_{d,e}$ is even and $\mcal{X}_n^{0,-1}=\mcal{X}_n^0\setminus\mcal{X}_n^{0,+1}$. 

The Dirichlet polynomial $\diri_{\g(\oo)}$ (see~\eqref{equation:orbit-dirichlet-poly}) is given by
\begin{equation}
\diri_{\g^\epsilon(\oo)}(s)=\sum_{\ttau\in \mcal{X}_n^{0,\epsilon}} M_{\ttau}(q) \cdot c^{\ttau}(q)\cdot \left(\frac{u_2^\epsilon(q)}{c^\tau(q)}\right)^{-s}+\sum_{\ttau\in\mcal{X}_n\setminus\mcal{X}_n^{0}} 4\cdot M_{\ttau}(q) \cdot c^{\ttau}(q)\cdot \left(\frac{u_2^\epsilon(q)}{2\cdot c^\tau(q)}\right)^{-s}.
\end{equation}
\end{mainthm}

\subsection{Organization}
\autoref{section:notation} gathers necessary preliminary results and sets up notation. \autoref{section:reg-elements-and-chars} contains basic structural results regarding the regular orbits of $\g(\oo)$ and regular characters of $\GG(\oo)$, and the proof of \autoref{mainthm:enumeration+dimension}. Finally, in  \autoref{section:classical-groups} we classify the regular adjoint orbits of $\g(\kk)$ and compute the regular representation zeta function of $\GG(\oo)$.

\section{Notation, preliminaries and basic definitions} \label{section:notation}

\subsection{The symplectic and orthogonal groups}\label{subsection:group-notation}

Fix $N\in\dbN$ and a matrix $\BB\in\GL_N(\oo)$ such that $\BB^t=\epsilon\BB$, with $\epsilon=-1$ in the symplectic case and $\epsilon=1$ in the special orthogonal case. The group scheme $\GG$ is defined by 
\begin{equation}\label{equation:G-definition}\GG(R)=\set{\mbf{x}\in\matr_N(R)\mid \mbf{x}^t\BB\mbf{x}=\BB\:\text{and }\det(\mbf{x})=1},
\end{equation}
where $R$ is a commutative $\oo$-algebra and the notation $\mbf{x}^t$ stands for the transpose matrix of $\mbf{x}$. A standard computation (see, e.g.~\cite[\S~12.3]{waterhouse}) shows that the Lie-algebra scheme $\g=\mrm{Lie}(\GG)$ is given by 
\begin{equation}\label{equation:Lie(G)-definition}
\g(R)=\set{\mbf{x}\in\matr_N(R)\mid \mbf{x}^t\BB+\BB\mbf{x}=0}.
\end{equation}
Let $n$ and $d$ denote the dimension and the absolute rank of the generic fiber of $\GG$. Note that the absolute rank and dimension of the generic fiber of $\GG$ are equal to those of its special fiber, by flatness of $\GG$ and of its maximal tori (see \cite[$\rm VI_B$,\: Corollary~4.3]{SGA3}).

\subsubsection{Adjoint operators}\label{subsubsection:adjoint-operators}
Let $R^N$ denote the $N$-th cartesian power of $R$, identified with the space $\matr_{N\times 1}(R)$ of column vectors, and define a non-degenerate bilinear form on $R^N$ by $B_R(u,v)=u^t\BB v$. One defines an $R$-anti-involution on $\matr_N(R)=\End_R(R^N)$ by \begin{equation}\label{equation:circ-definition}A^\inv=\BB^{-1}A^t\BB \quad(A\in\matr_N(R)),\end{equation} or equivalently, by letting $A^\inv$ be the unique matrix satisfying $B_R(A^\inv u,v)=B_R(u,Av)$, for all $u,v\in R^N$. In this notation, we have that $A\in \GG(R)$ if and only if $\det(A)=1$ and $A^\inv A= 1$, and that $A\in\g(R)$ if and only if $A^\inv+A=0$.

\subsubsection{Maximal tori and centralizers over algebraically closed fields}\label{subsubsection:tori}
Let $\TT$ be a maximal torus of $\GG$ and let $\mbf{t}\subseteq \g$ be its Lie-algebra. Given an algebraically closed field $L$, which is an $\oo$-algebra, we may assume that $\TT(L)$ is the group of $N\times N$ diagonal matrices. Moreover, upto possibly replacing $\BB$ with a congruent matrix, which amounts to conjugation of the given embedding $\GG\subseteq\GL_N$ by a fixed matrix over $\oo$, we may assume that $\TT(L)$ is mapped onto the subgroup of diagonal matrices $\diag(\nu_1,\ldots,\nu_{N})$, satisfying $\nu_{2i}=\nu_{2i-1}^{-1}$ for all $i=1,\ldots,\lfloor N/2\rfloor$, and with $\nu_N=1$ if $N$ is odd. In particular, the absolute rank of the generic fiber of $\GG$ is $n=\dim( \TT\times_{\spec \oo} \spec L)=\lfloor N/2\rfloor$, for $L=K^\alg$ the algebraic closure of $K$. 

Under this embedding, the Lie-algebra $\mbf{t}(L)$ consists of diagonal matrices of the form $\diag(\nu_1,\ldots,\nu_N)$, with $\nu_{2i}=-\nu_{2i-1}$ for all $i=1,\ldots,n$ and $\nu_N=0$ if $N$ is odd. We require the following well-known result.

\begin{propo}\label{propo:centralizer-semisimple-element}
Let $s\in \g(L)$ be a semisimple element. 
 The centralizer of $s$ under the adjoint action of $\GG(L)$ is of the form
\[\CC_{\GG(L)}(s)\simeq\prod_{j=1}^t\GL_{m_j}\left(L\right)\times \Delta(L),\]
where $\Delta$ is the $L$-algebraic group of isometries of the restriction of $B_L$ to (a 
non-degenerate bilinear form on) $\ker(s)$, the eigenspace associated with the eigenvalue $0$, and the values $m_1,\ldots,m_t$ are the algebraic multiplicities of all non-zero eigenvalues of $s$ such that for any such eigenvalue $\lambda$, there exists a unique $j=1,\ldots,t$ such that $m_j$ is the algebraic multiplicity of $\lambda$ and $-\lambda$.
\end{propo}
\begin{proof} 
Let $V=L^{N}$ the be the fixed $L$-vector space on which $\GG(L)\subseteq\GL_N(L)$ acts. The element $s$ is thus considered as an endomorphism of $V$. The decomposition of $V$ into eigenspaces of $s$ gives rise to a direct decomposition into isotypic $\CC_{\GL_N(L)}(s)$-modules, $V=\oplus_{\lambda\in L} W_\lambda$, where $W_\lambda=\ker(s-\lambda\id)$. For any non-zero $\lambda\in L$, put $W_{[\lambda]}=W_\lambda\oplus W_{-\lambda}$. A simple computation reveals that the spaces $W_0$ and $W_{[\lambda]}$ are non-degenerate with respect to the ambient symmetric or anti-symmetric bilinear form. Since $\CC_{\GG(L)}(s)=\CC_{\GL_N(L)}(s)\cap \GG(L)$, it holds that $x\in \CC_{\GG(L)}(s)$ if and only if $x\in \CC_{\GL_N(L)}(s)$ and $x$ acts as an isometry with respect to the restriction of $B$ to the spaces $W_0$ and $W_{[\lambda]}$ (for $\lambda\ne 0$). 

Arguing as in \cite[III, \S~2.4]{BCM}, one verifies that for any $\lambda\ne 0$ the decomposition $W_{[\lambda]}=W_{\lambda}\oplus W_{-\lambda}$ is into maximal isotropic subspaces, and in particular $\dim_L W_{\lambda}=\dim_{L} W_{-\lambda}=m_j$, for some $j=1,\ldots,t$. Invoking Witt's Theorem \cite[\S~1.2]{WallIsometries}, and the $\CC_{\GL_N(L)}(s)$-isotipicity of the decomposition, we obtain that any automorphism of $W_{\lambda}$ extends uniquely to an isometric automorphism of $W_{[\lambda]}$ which commutes with the action of $s$, and that the action of $\CC_{\GG(L)}(s)$ on $W_{[\lambda]}$ is determined in this manner. Furthermore, it holds that any automorphism of $W_0=\ker(s)$ which preserves the restriction of $B$ to $W_0$ necessarily commutes with $s$, and that the action of $\CC_{\GG(L)}(s)$ on this subspace is by such automorphisms. The proposition follows. 
\end{proof}

\subsection{Artinian local principal ideal rings}\label{subsubsection:artinian-rings}
Let $K^\alg$ be a fixed algebraic closure of $K$ and let $K^\mrm{unr}$ be the maximal unramified extension of $K$ in $K^\alg$. Let $\OO$ be the valuation ring of $K^\mrm{unr}$, and $\PP=\pi\OO$ its maximal ideal. The residue field of $\OO$ is identified with the algebraic closure $\bkk$ of $\kk$. Given $r\in\dbN$ we put $\oo_r:=\oo/\pp^r$ and $\OO_r:=\OO/\PP^r$ and write $\eta_r:\OO\to\OO_r$ and $\eta_{r,m}:\OO_r\to \OO_m$ for the reduction maps, for any $1\le m\le r$. The notation $\eta_r$ and $\eta_{r,m}$ is also used to denote the coordinatewise reduction map on $\matr_N(\OO)$ and $\matr_N(\OO_r)$, respectively. 

The map $\eta_1$ admits a canonical splitting map $s:\bkk\to\OO$, which restricts to a homomorphic embedding of $(\bkk)^\times$ into $\OO^\times$, and satisfies $s(0)=0$ ; see \cite[Ch.~II,\:\S~4,\:Proposition~8]{SerreLocal}. 

Let $\sigma:K^{\mrm{unr}}\to K^{\mrm{unr}}$ be the local Frobenius map whose fixed field is $K$. Then $\sigma$ restricts to a ring automorphism of $\OO$, with fixed subring $\OO^\sigma=\oo$, and induces a map $\OO_r\to \OO_r$ for any $r\ge 1$ whose fixed subring is $\oo_r$. In the special case $r=1$, the map $\sigma:\bkk\to\bkk$ is given by the $q$-power map $x\mapsto x^q$, where $q=\abs{\kk}$.
\subsection{The Greenberg functor}\label{subsection:greenberg}

The Greenberg functor was introduced in \cite{GreenbergI} and \cite{GreenbergII}, as a generalization of Shimura's reduction mod $\pp$ functor to higher powers of $\pp$. Given an artinian local principal ideal ring $R$ (or more generally, an artinian local ring) with a perfect residue field $\mfr{k}$, the Greenberg functor $\mcal{F}_R$ associates to any $R$-scheme $\mbf{Y}$ locally of finite type a scheme $\mcal{F}_R(\mbf{Y})$ locally of finite type over the residue field $\mfr{k}$. Given another such ring $R'$ with residue field $\mfr{k}$ and a ring homomorphism $R\to R'$, the functors $\mcal{F}_R$ and $\mcal{F}_{R'}$ are related via \textit{connecting morphisms}, on which we expand further below.

A defining property of the functor is the existence of a canonical bijection\begin{equation}
\label{equation:greenberg-main}
\mcal{F}_{R}(\mbf{Y})(\mfr{k})=\mbf{Y}(R).
\end{equation}
More generally, if $A$ is a perfect commutative unital $\mfr{k}$-algebra, then either $\mcal{F}_R(\mbf{Y})(A)=\mbf{Y}(R\otimes_{\mfr{k}} A)$, in the case where $R$ is a $\mfr{k}$-algebra, or otherwise
\begin{equation}\label{equation:greenberg-general}
\mcal{F}_{R}(\mbf{Y})(A)=\mbf{Y}(R\otimes_{W(\mfr{k})}W(A))
\end{equation}
where $W(\cdot)$ denotes the ring of $p$-typical Witt vectors~\cite[Ch.~II,\S~6]{SerreLocal}. For further introduction we refer to \cite[p.~276]{BLR-Neron}.

Our application of the Greenberg functor is focused on the artininan principal ideal rings $\OO_r$. For any $r$, we let $\GG_{\OO_r}=\GG\times \OO_r$ and $\g_{\OO_r}=\g\times\OO_r$ denote the base change of the group and Lie-algebra schemes $\GG$ and $\g$. Put $\GGamma_r=\mcal{F}_{\OO_r}(\GG_{\OO_r})$ and $\ggamma_r=\mcal{F}_{\OO_r}(\g_{\OO_r})$. Given $m\le r$, we write $\eta_{r,m}^*$ to denote the connecting maps $\GGamma_r\to \GGamma_m$ and $\ggamma_r\to\ggamma_m$, and put $\GGamma^m_r=(\eta_{r,m}^*)^{-1}(1)=\spec(\kappa(1))\times_{\GGamma_m}\GGamma_r$ (the scheme-theoretic group kernel) and $\ggamma_{r}^m=(\eta_{r,m}^*)^{-1}(0)=\spec(\kappa(0))\times_{\ggamma_m}\ggamma_r$ (the scheme-theoretic Lie-algebra kernel). Here, the notation $\kappa(\cdot)$ stands for the residue field at a rational point of a scheme.


Note that, a priori, the connecting morphism between a scheme and its base change is dependent on the scheme in question as well. The apparent abuse of notation in writing $\eta_{r,m}^*$ for the connecting morphisms of different schemes is permissible by \cite[\S~5,\:Corollary~4]{GreenbergI}, applied for $g$ the inclusion morphism (see also Assertion~2 of the Main Theorem of \textit{loc. cit.}).


The main properties which we require are summarized in the following lemma. 

\begin{lem}\label{lem:greenberg-props}For $r\in\dbN$ fixed, we have
\begin{enumerate}
\item The rings $\OO_r$ are the $\bkk$-points of an $r$-dimensional algebraic ring scheme $\mbf{O}_r$ over $\bkk$. 

The canonical map $s:\bkk\to\OO_r$ defines a closed embedding $s^*:\dbA_\bkk^1\to\mbf{O}_r$ of the affine line over $\bkk$ into this ring variety. The restriction of $s^*$ to the multiplicative group $\mathbb{G}_m\subseteq\dbA_{\bkk}^1$ is a monomorphism of $\bkk$-linear algebraic groups, satisfying $\eta_{r,1}^*\circ s^*=\id_{\dbA_{\bkk}^1}$.
\item The group $\GGamma_r$ is a $d\cdot r$-dimensional linear  algebraic group over $\bkk$. 
\item  The Greenberg functor maps smooth closed sub-$\OO_r$-group schemes of $\GG_{\OO_r}$ to closed algebraic $\bkk$-subgroups of $\GGamma_r$.
\item The scheme $\ggamma_r$ is a $d\cdot r$-dimensional affine space over $\bkk$, which is naturally endowed with the structure of a Lie-algebra scheme over the ring scheme $\mbf{O}_r$. 
\item The connecting morphisms $\eta_{r,m}^*$, for $m\le r$, give rise to surjective $\bkk$-group scheme $\GGamma_r\to~\GGamma_m$ morphisms. Similarly, for $\ggamma_r\to \ggamma_m$, these are surjective Lie-ring morphisms.
\item The adjoint action of $\GG_{\OO_r}$ on the Lie-ring scheme $\g_{\OO_m}$ with $m\le~r$, induces an action of the algebraic group $\GGamma_r$ on $\ggamma_m$. The application of $\mcal{F}_{\OO_r}$ preserves centralizers of $\OO_m$-rational points of $\g_{\OO_m}$.
\end{enumerate}
\end{lem}
\begin{proof}
\newcounter{list}	
\begin{list}{\arabic{list}.}{\usecounter{list}\setlength{\leftmargin}{5pt}
   \setlength{\itemsep}{2pt} \setlength{\parsep}{0pt}\setlength{\labelwidth}{-5pt}}
\item See \cite[\S~1,\:Proposition~4]{GreenbergI}.
\item The group $\GGamma_r$ is a smooth affine group scheme of finite type over $\bkk$ (see \cite[\S~4,\:Theorem.(5)]{GreenbergI} and \cite[Corollary~1,\:p.~263]{GreenbergII}). Thus, by \cite[11.6]{waterhouse}, $\GGamma_r$ is a linear algebraic $\bkk$-group (see also \cite[\S~4]{Stasinski}). The dimension of $\GGamma_r$ may be computed by induction on $r$, using Greenberg's Structure Theorem \cite{GreenbergII} (see remark on p.~266 of \textit{loc. cit.}; also, see \cite[Lemma~4.1.1]{Begueri-Dualite} for an explicit argument in the case where $r$ is divisible by the absolute ramification index of $\oo$). 
\item Let $\Delta\subseteq\GG_{\OO_r}$ be a closed smooth sub-$\OO_r$-scheme. The argument of the previous assertion shows that $\mcal{F}_{\OO_r}(\Delta)$ is a linear algebraic group over $\bkk$. That $\mcal{F}_{\OO_r}(\Delta)$ is a closed subgroup of $\GGamma_r$ follows from Assertions (2) and (5) of the Main Theorem of \cite[p.~643]{GreenbergI}.


\item The Lie-algebra scheme $\g_{\OO_r}$ is isomorphic to the affine $d$-dimensional space $\dbA^d_{\OO_r}$ over $\OO_r$, and is endowed with $\OO_r$-regular maps, defining an $\OO_r$-module structure and an $\OO_r$-bilinear Lie-bracket on $\g_{\OO_r}$. It follows from the Main Theorem \cite[p.~643]{GreenbergI}, that $\ggamma_r=\mcal{F}_{\OO_r}(\g_{\OO_r})$ is isomorphisc to $\dbA_{\bkk}^{dr}$, the affine space of dimension $d\cdot r$ over $\bkk$. Multiplication by scalars from $\mbf{O}_r$, the Lie-bracket and addition on $\g_{\OO_r}$ are transported by $\mcal{F}_{\OO_r}$ to schematic morphisms $\mathbf{O}_r\times\ggamma_r\to \ggamma_r$ and ${\ggamma_r\times\ggamma_r}\to\ggamma_r$ by \cite[Corollary~3, p.641]{GreenbergI}. The Lie-axioms on $\mcal{F}_{\OO_r}(\ggamma_r)$ may be verified using compatibility of the Greenberg functor with preimages \cite[Corollary~3,\:p.~641]{GreenbergI}. For example, the Jacobi identity can be reformulated using the  equality $\g_{\OO_r}\times\g_{\OO_r}\times\g_{\OO_r}=J^{-1}(0)$, where $J:\g_{\OO_r}\times\g_{\OO_r}\times\g_{\OO_r}\to\g_{\OO_r}$ is the morphism satisfying $J(R)(x,y,z)=[[x,y],z]+[[y,z],x]+[[z,x],y]$ for any $\OO_r$-algebra $R$ and $x,y,z\in\g_{\OO_r}(R)$.

\item The connecting map is shown to be a group homomorphism in \cite[\S~5,\:Corollary~5]{GreenbergI}, and the preservation of the Lie-bracket follows similarly from \cite[\S~5,\:Corollary~2]{GreenbergI}. Its surjectivity follows from the smoothness of $\GG_{\OO_r}$ (resp. $\ggamma_{\OO_r}$), and \cite[Corollary 2, p. 262]{GreenbergII}. 
\item The action of $\GGamma_r$ on $\ggamma_m$ is given by $\mcal{F}_{\OO_m}(\alpha_{m})\circ (\eta_{r,m}^*\times \id_{\ggamma_m}):\GGamma_r\times\gamma_m\to\ggamma_m$, where $\alpha_{m}:\GG_{\OO_m}\times\g_{\OO_m}\to \g_{\OO_m}$ is the adjunction map; see \cite[\S~3]{Stasinski}. One notes easily that, since the group $\GGamma^m_r$ acts trivially on $\ggamma_m$, this action commutes pointwise with the bijection \eqref{equation:greenberg-general}. The preservation of centralizers follows from \cite[Proposition~3.6]{Stasinski}, by taking $\mbf{Y}$ and $\mbf{Z}$ to be the sub-schemes defined by the spectrum of the reside field of $\g_{\OO_m}$ at the given rational point. 

\end{list}
\end{proof}

\begin{rem} In the case where $\OO_r$ is a $\bkk$-algebra, \autoref{lem:greenberg-props}.(3) may be somewhat strengthened, as in this case $\gamma_r$ can be shown to coincide with the Lie-algebra of $\GGamma_r$. In the case of unequal characteristic, the equality $\ggamma_r=\Lie(\GGamma_r)$ is generally false. For example, in the case of $\GG=\mathbb{G}_a$, the additive group scheme, we have that $\ggamma_2(\bkk)=\Lie(\mathbb{G}_a)(W_2(\bkk))=W_2(\bkk)$ is a ring of characteristic $p^2$, while $\Lie(\GGamma_2)(\bkk)$ is a two-dimensional $\bkk$-Lie-algebra and, in particular, has $p$-torsion. 
\end{rem}

We also require the following lemma.
\begin{lem}\label{lem:verschiebung}For any $m,r\in\dbN$ with $m\le r$, there exists an injective homomorphism of the underlying additive  group schemes $v_{r,m}^*:\ggamma_{r-m}\to \ggamma_r$, such that 
\begin{enumerate}
\item for any $y\in\ggamma_{r-m}(\bkk)=\g(\OO_{r-m})$, it holds that $v_{r,m}^*(\bkk)(y)=\pi^{m}\tilde{y}$, where $\tilde{y}\in\g(\OO_r)$ is such that $\eta_{r,r-m}(\tilde{y})=y$;
\item the sequence $0\to \ggamma_{r-m}\xrightarrow{v_{r,m}^*}\ggamma_{r}\xrightarrow{\eta_{r,{m}}^*}\ggamma_{m}\to 0$ is exact;
\item for any $y\in \ggamma_{r-m}(\bkk)$ the square \eqref{equation:verschiebung} commutes
\begin{equation}
\label{equation:verschiebung}
\begin{tikzpicture}
 \matrix (m) [matrix of math nodes,row sep=2em,column sep=4.5em,minimum width=2em,nodes=
        {minimum height=2em, 
        text height=1.5ex,
        text depth=.25ex,
        anchor=center}]
{\ggamma_r&\ggamma_r\\
\ggamma_{r-m}&\ggamma_{r-m},\\
};
\path[->]{
(m-1-1)	edge[thin]	node[auto]	{$\ad(v_{r,m}^*(y))$}	(m-1-2)
(m-2-1)	edge[thin]	node[auto]	{$\ad(y)$}	(m-2-2)
(m-1-1)	edge[thin]	node[left]	{$\eta_{r,r-m}^*$}	(m-2-1)
(m-2-2)	edge[thin]	node[right]	{$v_{r,m}^*$}	(m-1-2)
}
;
\end{tikzpicture}\end{equation}
where $\ad(z):\ggamma_j\times\ggamma_j\to \ggamma_j$ (for $j\in\dbN$ and $z\in\ggamma_j(\bkk)$) is the map defined by $\ad(z)(A)(x)=[z,x]$ for any commutative unital $\bkk$-algebra $A$ and $x\in \ggamma_j(A)$;
\item The equality $\eta_{r,m+1}^*\circ v_{r,m}^*=v_{m+1,1}^*\circ \eta_{r-m,1}^*$ holds.
\end{enumerate} 
\end{lem}
\begin{proof}

The map $x\mapsto \pi^{m}x\colon\g(\oo_r)\to\g(\oo_r)$ gives rise to an injective $\oo_r$-module map $v_{r,m}:\g(\oo_{r-m})\to\g(\oo_r)$, which in turn extends to a map of $\OO_r$-modules, giving rise to the exact sequence
\[0\to\g_{\OO_{r-m}}(\OO_r)\xrightarrow{v_{r,m}}\g_{\OO_r}(\OO_r)\xrightarrow{\eta_{r,m}}\g_{\OO_{m}}(\OO_r)\to 0.\]
Applying \cite[\S~1,\:Proposition 3.(6)]{GreenbergI} to both maps of the above sequence, these define $\kk$-regular maps between associated module variety structures over $\bkk$ of the modules above, which, in turn, define an exact sequence of $\bkk$-schemes
\[0\to\ggamma_{r-m}\xrightarrow{v_{r,m}^*}\ggamma_r\xrightarrow{\eta_{r,m}^*}\ggamma_{m}\to 0,\]
where the right-most map coincides with $\eta_{r,m}^*$ by same proposition and by \cite[\S~5,\:Corollary~2]{GreenbergI}. The first and second assertions of the lemma follow.

As for the third assertion, in order to prove that the morphisms ${v_{r,m}^*\circ \ad(y)\circ \eta_{r,r-m}^*}$ and $\ad(v_{r,m}(y))$ coincide, it is enough to show that, upon passing to their associated comorphisms, they induce the same endomorphism of the coordinate ring of $\ggamma_r$. Invoking the isomorphism $\ggamma_r\simeq \dbA_{\bkk}^{dm}$ of \autoref{lem:greenberg-props}.(3), since an endomorphism of a polynomial algebra in $dm$ variables is determined by specifying the images of $t_1,\ldots,t_{dm}$ in $\bkk[t_1,\ldots,t_{dm}]$, by Nullstellensatz, it is enough to show that the two endomorphisms above coincide pointwise on $\ggamma_r(\bkk)$. This is immediate by the first two assertions and the linearity of $\ad(\cdot)$ over $\OO_r$. The fourth assertion may be proved in a similar vein as Assertion (3).
\end{proof}
\begin{rem}
In the case where $\OO$ is either a $\bkk$-algebra, or is absolutely unramified (i.e. $\PP=p\OO$), and thus isomorphic to the ring $W(\bkk)$, the map $v_{r,m}^*$ of the lemma may be described explicitly, by fixing a suitable coordinate system for $\ggamma_r$ over $\bkk$ and taking $v^*_{r,m}$ to be either a coordinate shift in the former case, or given by successive applications of the verschiebung and Frobenius maps coordinatewise (see \cite{SerreLocal}) in the latter.
\end{rem}

\subsection{The Cayley map}\label{subsubsection:cayley}
Let $D$ be the affine $\oo$-scheme $\spec(\oo[t_{1,1},\ldots,t_{N,N},(\det(\mbf{t}+1))^{-1}])$, where $t_{1,1},\ldots,t_{N,N}$ are indeterminates and $\mbf{t}+1$ is the $N\times N$ matrix whose $(i,j)$-th entry is $t_{i,j}+\delta_{i,j}$, with $\delta_{i,j}$ the Kronecker delta function. Note that for any commutative unital $\oo$-algebra $R$, the set of $R$-points of $D$ is naturally identified with the set
\begin{equation}\label{equation:D(R)-definition}\set{\mbf{x}\in\matr_N(R)\mid \det\left(1+\mbf{x}\right)\in R^\times}.\end{equation}

Let $\cay:D\to D$ be the $\oo$-scheme morphism with associated comorphism $\cay^\sharp$ given on generators of $\oo[t_{1,1},\ldots,t_{N,N},(\det(1+\mbf{t}))^{-1}]$ by mapping $t_{i,j}$ to the $(i,j)$-th entry of the matrix $(1-\mbf{t})(1+\mbf{t})^{-1}$. Note that $\cay^\sharp(\det(1+\mbf{t})^{-1})=2^{-N}\det(1+\mbf{t})$. A direct computation shows that, as $2$ is invertible in $\oo$, the map $\cay^\sharp$ is its own inverse and thus $\cay$ is an isomorphism of $D$ onto itself.

Under the identification \eqref{equation:D(R)-definition} for $R$ an $\oo$-algebra as above, the action of $\cay$ on the set of $R$-points of $D$ is given explicitly by
\begin{equation}
\cay(R)(\mbf{x})=(1-\mbf{x})(1+\mbf{x})^{-1}.\label{equation:cayley-definition}\end{equation}

In the specific case $R=\bkk$, the sets $(D\cap\g)(\bkk)$ and $(D\cap\GG)(\bkk)$ are principal open subsets of $\ggamma(\bkk)$ and $\GGamma(\bkk)$ respectively\footnote{Here $\cap$ denotes the scheme-theoretic intersection, $D\cap\g=D\times_{\spec(\oo[t_{1,1},\ldots,t_{N,N}])}\ggamma$. Note that, in the present setting, as $\g\subseteq\matr_N\times\oo=\spec(\oo[t_{1,1},\ldots,t_{N,N}])$, for any $\oo$-algebra $R$, $(D\cap \g)(R)$ is simply the set of matrices $\mbf{x}\in \g(R)$ such that the matrix $\id+\mbf{x}$ is invertible. Likewise for $D\cap \GG$, using the inclusions $\GG\subseteq\SL_N\times\oo\subseteq\matr_N\times\oo$.}. Using the description of $\g$ and $\GG$ given in \autoref{subsubsection:adjoint-operators}, one verifies that the restriction of $\cay({\bkk})$ to $(D\cap \g)(\bkk)$ defines an algebraic isomorphism of affine varieties onto $(D\cap \GG)(\bkk)$, and hence a birational map $\cay({\bkk}):\g(\bkk)\dashrightarrow \GG(\bkk)$. Additionally, being given by a rational function in $\mbf{x}$ on $(D\cap\g)(\bkk)$, the map $\cay(\bkk)$ is equivariant with respect to the conjugation action of $\GG(\bkk)$. The properties listed in this paragraph carry over to the associated $\bkk$-group schemes described in the previous section, as noted in \autoref{lem:props-of-cayley-map} below.

The Cayley map was introduced in \cite{Cayley}. Its generalization to groups arising as the set of unitary transformations with respect an anti-involution of an associative algebra is attributed to A. Weil \cite[\S~4]{WeilAlgebras}. See also \cite{LPR-Cayley} for a more generalized treatment of the Cayley map.

\subsubsection{Properties of the Cayley map}\label{subsection:cayley-map}\newcommand{\DDelta}{\boldsymbol{\Delta}}
Given $r\in\dbN$, put $D_r=D\times\OO_r$ and let $\cay_r=\cay\times \id_{\OO_r}$ be the base change of $\cay$. Let $\DDelta_r=\mcal{F}_{\OO_r}(D_r)$ and $\widehat{\cay}_r=\mcal{F}_{\OO_r}(\cay_r)$. Note that, by construction and by the Main Theorem of \cite{GreenbergI}, $\DDelta_{r}$ is an open affine subscheme of $\dbA_{\bkk}^{N^2m}$.

\begin{lem}\label{lem:props-of-cayley-map} 
Let $1\le m\le r$. The map $\widehat{\cay}_r$ has the following properties.
\begin{enumerate}\renewcommand{\theenumi}{Cay\arabic{enumi}}
\item\label{cayley-property-1} The map $\widehat{\cay}_r$ is a birational equivalence $\ggamma_r\dashrightarrow\GGamma_r$. Furthermore, its restriction to the kernel $\ggamma_r^m$ is an isomorphism of $\bkk$-varieties onto $\GGamma_r^m$, and is an isomorphism of abelian groups if $2m\ge r$.
\item\label{cayley-property-2} The map $\widehat{\cay}_r$ is $\GGamma_r$-equivariant with respect to the action given in \autoref{lem:greenberg-props}.(6) on $\ggamma_r$ and with respect to group conjugation on $\GGamma_r$.
\item\label{cayley-property-3} The diagram in \eqref{figure:cayley-square} commutes.
\begin{equation}
\begin{tikzpicture}
 \matrix (m) [matrix of math nodes,row sep=3em,column sep=4.5em,minimum width=2em,nodes=
        {minimum height=2em, 
        text height=1.5ex,
        text depth=.25ex,
        anchor=center}]
{\ggamma_r&\GGamma_r\\
\ggamma_m&\GGamma_m.\\
};
\path[->]{
(m-1-1)	edge[dashed]	node[auto]	{$\widehat{\cay}_r$}	(m-1-2)
(m-2-1)	edge[dashed]	node[auto]	{$\widehat{\cay}_m$}	(m-2-2)
(m-1-1)	edge[thin]	node[left]	{$\eta_{r,m}^*$}	(m-2-1)
(m-1-2)	edge[thin]	node[auto]	{$\eta_{r,m}^*$}	(m-2-2)
}
;
\end{tikzpicture}\label{figure:cayley-square}\end{equation}\end{enumerate}
\end{lem}
\begin{proof}
\begin{list}{\arabic{list}.}{\usecounter{list}\setlength{\leftmargin}{0pt}
   \setlength{\itemsep}{2pt} \setlength{\parsep}{0pt}\setlength{\labelwidth}{-5pt}}
\item The inclusion map $D_r\cap \g_{\OO_r}\subseteq \g_{\OO_r}$ is an open immersion, and thus by Assertion (2) and (3) of the Main Theorem of \cite{GreenbergI}, the $\bkk$-scheme $\DDelta_r\cap \ggamma_r$ is immersed as an open subscheme of $\ggamma_r$. Similarly for $\DDelta_r \cap \GGamma_r$. By functoriality, the morphism $\widehat{\cay}_r$ is an isomorphism of $\DDelta_r\cap \ggamma_r$ onto $\DDelta_r\cap \GGamma_r$, and hence $\ggamma_r$ and $\GGamma_r$ are birationally equivalent. 

To prove that $\widehat{\cay}_r$ restricts to an isomorphism of $\ggamma^m_r$ onto $\GGamma^m_r$, it would be enough that both are embedded as sub-schemes of $\DDelta_r$ under the given inclusions into $\dbA_{\bkk}^{N^2m}$. Note that by applying Greenberg's Structure Theorem \cite{GreenbergII} inductively, both $\ggamma^m_r$ and $\GGamma_r^m$ are reduced, and thus are $\bkk$-varieties. Thus, by Nullstellensatz, they are determined by their $\bkk$-points and it suffices to show they are included in the reduced subscheme $(\DDelta_r)_{\mrm{red}}\subseteq\DDelta_r$. This follows from the bijection~\eqref{equation:greenberg-main}, as $\ggamma^m_r(\bkk)=\g_{\OO_r}(\OO_r)\cap \eta_{r,m}^{-1}(0)$ is included in the nilradical of the matrix algebra~$\matr_N(\OO_r)$, and hence included in $D_r(\OO_r)$, and since $\GGamma_r^m(\bkk)=\GG_{\OO_r}(\OO_r)\cap \eta_{r,m}^{-1}(1)\subseteq 1+\pi\matr_N(\OO_r)\subseteq \GL_N(\OO_r)$, and thus (since $\Char( \bkk)\ne 2$) included in $D_r(\OO_r)$.

Lastly, to prove that $\widehat{\cay}_r$ is a group homomorphism whenever $2m\ge r$, it is equivalent to show that it preserves comultiplication in the Hopf-algebra structure of the coordinate ring of $\ggamma_r^m$ in this case. Arguing as in the proof \autoref{lem:verschiebung}, it sufficient to verify this on the $\bkk$-points of the variety. This follows from the definition of $\cay$ \eqref{equation:cayley-definition}, as in this case $\ggamma_r^m(\bkk)\subseteq\g_{\OO_r}(\OO_r)$ is included in an ideal of vanishing square in $\matr_N(\OO_r)$ and the map $\widehat{\cay}_r$ coincides with the map $\mbf{x}\mapsto 1-2\mbf{x}$.

%

\item Property \eqref{cayley-property-2} holds since $\mcal{F}_{\OO_r}$ maps the cartesian square \eqref{equation:Gm-equivariance-cayley}, which states the $\GG_{\OO_r}$-equivariance of $\cay_r$, to a corresponding cartesian square, stating the $\GGamma_r$-equivariance of $\widehat{\cay}_r$.
\begin{equation}
\label{equation:Gm-equivariance-cayley}\begin{tikzpicture}
 \matrix (m) [matrix of math nodes,row sep=3em,column sep=7em,minimum width=2em,
 nodes={minimum height=2em, 
        text height=1.5ex,
        text depth=.25ex,
        anchor=center}]
{\GG_{\OO_r}\times\g_{\OO_r}&\g_{\OO_r}
\\\GG_{\OO_r}\times\GG_{\OO_r}&\GG_{\OO_r}\\
};
\node (1,1) [right] {$\square$};
\path[->]{
(m-1-1)	edge[thin]	node[auto]	{$\alpha_{\GG_{\OO_r},\g_{\OO_r}}$}	(m-1-2)
(m-2-1)	edge[thin]	node[auto]	{$\alpha_{\GG_{\OO_r},\GG_{\OO_r}}$}	(m-2-2)
(m-1-1)	edge[dashed]	node[left]	{$\id_{\GG_{\OO_r}}\times \cay_r$}	(m-2-1)
(m-1-2)	edge[dashed]	node[auto]	{$\cay_r$}	(m-2-2)
}
;
\end{tikzpicture}\end{equation}
Here $\alpha_{\GG_{\OO_r},\mbf{X}}$ denotes the action map of $\GG_{\OO_r}$ on $\mbf{X}\in\set{\GG_{\OO_r},\g_{\OO_r}}$ by either conjugation or by the adjoint action. 

\item Finally, property \eqref{cayley-property-3} is simply an application of \cite[Corollary~4, p.645]{GreenbergI}, to the case $R=~\oo_r$, $R'=\oo_{m},\:\varphi=\eta_{r,m}$, $X_1=\ggamma_r\cap D_r$, $X_2=\GGamma_r\cap D_r$ and $g=\cay_r$.
\end{list}
\end{proof}


\subsection{Groups, Lie algebras and characters}\label{subsection:characters}

In general, given finite groups $\Delta\subseteq \Gamma$ and characters $\sigma\in\irr(\Delta)$ and $\chi\in\irr(\Gamma)$, we denote by $\chi_\Delta$ the restriction of $\chi$ to $\Delta$, and by $\sigma^\Gamma$ the character induced from $\sigma$ in $\Gamma$. 
Group commutators are denoted by $(x,y)=xyx^{-1}y^{-1}$. Lie-algebra commutators are denoted by $[x,y]=xy-yx$. The center of a group $\Gamma$ is denoted by $\ZZ(\Gamma)$.

The Pontryagin dual of a finite abelian group $\Delta$ is denoted by $\widehat{\Delta}=\hom(\Delta,\dbC^\times)$. If $\Delta$ is endowed with an additional structure (e.g. a ring or a Lie-algebra), then $\widehat{\Delta}$ refers to the Pontryagin dual of the abelian group underlying $\Delta$.

\section{Regular elements and regular characters}\label{section:reg-elements-and-chars}
\subsection{Regular elements}\label{subsection:regular-elements}

We begin our analysis of regular characters by inspecting the group $\GG(\OO)$. To do so, we first consider the regular orbits for the action of $\GG(\OO_r)$ on $\g(\OO_r)$, or, equivalently (see \cite[\S~3]{Stasinski}), of $\GGamma_r(\bkk)$ on $\ggamma_r(\bkk)$, via the action described in \autoref{lem:greenberg-props}.(6). The methods which we apply are influenced by \cite{Hill}. 

Recall that an element of a reductive algebraic group over an algebraically closed field is said to be \textbf{regular} if its centralizer is an algebraic group of minimal dimension among such centralizers \cite[\S3.5]{SteinbergConjugacy}. Following \cite{Hill}, this definition is extended to elements of $\ggamma_r$.

\begin{defi}\label{defi:regular-elements} Let $r\ge 1$. An element $x\in\g(\OO_r)$ is said to be \textbf{regular} if the group scheme $\mcal{F}_{\OO_r}\left(\CC_{\GG_{\OO_r}}(x)\right)=\CC_{\GGamma_r}(x)$, obtained by applying the Greenberg functor to the centralizer group scheme of $x$ in $\GG_{\OO_r}$, is of minimal dimension among such group schemes.
\end{defi}

The following theorem lists the main properties of regular elements of $\ggamma_r$, which are proved in this section.

\begin{theo}\label{theo:properties}
Let $\GG$ be a symplectic or a special orthogonal group scheme over a complete discrete valuation ring $\oo$ of odd residue field characteristic, with Lie-algebra $\g=\Lie(\GG)$. Fix $r\in\dbN$ and let $x\in \g(\OO_r)$. 
\begin{enumerate}
\item If $x_r$ is a regular element of $\g_{\OO_r}(\OO_r)$, then $\CC_{\GGamma_{r}}(x_r)$ is a $\bkk$-group scheme of dimension $r\cdot n$, where $n=~\rk(\GG\times~ K^{\alg})$. 
\item The element $x_r$ is regular if and only if $x_1=\eta_{r,1}(x_r)$ is a regular element of~$\g(\bkk)$.
\item Suppose $x_r\in\g(\OO_r)$ is regular. The restriction of the reduction map $\eta_{r,1}$ to $\CC_{\GG(\OO_r)}(x_r)$ is surjective onto $\CC_{\GG(\bkk)}(x_1)$.
\end{enumerate}
\end{theo}


\begin{rem}Assertions (1) and (3) of \autoref{theo:properties}, as well as Assertion (1) of \autoref{theo:inverse-lim} below, are formal consequences of the stronger statement that the centralizer group scheme $\CC_{\GG_{\OO_r}}(x)$ is smooth over $\OO_r$, whenever $x\in\g(\OO_r)$ is regular. This statement, while plausible, is not proved in this article.
\end{rem}

The proofs of Assertions (1),\:(2) and (3) of \autoref{theo:properties} are given, respectively, in sections \ref{subsubsection:general-props-Gamma_r}, \ref{subsubsection:reduction-map} and \ref{subsubsection:image-eta_r,1} below. Once the proof of \autoref{theo:properties} is complete, we return to analyze the case of regular elements of $\mfr{g}_r=\ggamma_r(\OO_r)^\sigma$.

\begin{propo}\label{theo:inverse-lim} Let $\GG$ be a symplectic or a special orthogonal group over $\oo$ with $\g=\Lie(\GG)$ and let $x\in \mfr{g}=\g(\oo)$. Assume $x_r=\eta_r(x)$ is regular for some $r\in\dbN$. Then
\begin{enumerate}
\item $\CC_{G}(x)=\varprojlim_r\CC_{G_r}(x_r),$
where $G=\GG(\oo)$ and $G_r=\GG(\oo_r)$
\item Furthermore, $x$ is a regular element of $\g(K^\alg)$.
\end{enumerate} 
\end{propo}

\autoref{theo:inverse-lim} has the following corollary.
\begin{corol}\label{corol:fin-index-abelian}In the notation of \autoref{theo:inverse-lim}, let $x\in\mfr{g}$ such that $x_r=\eta_r(x)$ is a regular element of $\mfr{g}_{r}$, for some $r\in\dbN$. Then $\CC_{G_r}(x_r)$ is abelian.
\end{corol}

\subsubsection{General properties of the groups $\GGamma_r$}\label{subsubsection:general-props-Gamma_r} 
 We begin by examining some basic properties of the group $\GGamma_r$ ($r\in\dbN$) and of centralizers of elements of $\ggamma_r$, when considered as algebraic group schemes over $\bkk$. The following lemma summarizes the necessary components for the proof of \autoref{theo:properties}.(1), and is mostly included in \cite{Stasinski}.

\begin{lem}\label{lem:GGamma-properties}\begin{enumerate}
\item The group scheme $\GGamma_r$ is a connected linear algebraic group over $\bkk$.
\item The unipotent radical of $\GGamma_r$ is $\GGamma_r^1$.
\item Let $\TT$ be a maximal torus of $\GG$, defined over $\OO$, and let $\TT_1=\TT\times\bkk\subseteq \GGamma_1$. The restriction of the map $s^*:~\dbA_{\bkk}^1\to~\mbf{O}_r$ of \autoref{lem:greenberg-props}.(1) to $\mathbb{G}_m$ extends to an embedding of $\TT_1$ as a maximal torus in $\GGamma_r$. 
\item The centralizer of $s^*(\TT_1)$ in $\GGamma_r$ is the Cartan subgroup $\mcal{F}_{\OO_r}(\TT\times\OO_r)$. Moreover, $\mcal{F}_r(\TT\times~\OO_r)$ is a linear algebraic $\bkk$-group of dimension $n\cdot r$.
\end{enumerate}
\end{lem}
\begin{proof}
\begin{list}{\arabic{list}.}{\usecounter{list}\setlength{\leftmargin}{0pt}
   \setlength{\itemsep}{2pt} \setlength{\parsep}{0pt}\setlength{\labelwidth}{-5pt}}
\item Connectedness is proved in \cite[Lemma~4.2]{Stasinski}. The fact that $\GGamma_r$ is linear algebraic over $\bkk$ is shown in \autoref{lem:greenberg-props}.(1).
\item See \cite[Proposition~4.3]{Stasinski}.
\item May be proved by following the argument of \cite[Proposition~2.2.(2)]{Hill}, practically verbatim, making use of the fact that $\GGamma_r$ and $\GGamma_1=\eta_{r,1}^*(\GGamma_r)$ are of the same rank by the previous assertion, and that $s^*(\TT_1)$ is a connected abelian subgroup of $\GGamma_r$ of dimension $n=\rk(\GGamma_1)$.
\item The inclusion $\TT(\OO_r)\subseteq\CC_{\GGamma_r(\bkk)}(s^*(\TT_1)(\bkk))$ is clear, since $\TT(\OO_r)$ is abelian and contains $s^*(\TT_1)(\bkk)$, by \autoref{lem:greenberg-props}.(1). The inclusion $\TT\times\OO_r\subseteq\CC_{\GGamma_r}(s^*(\TT_1))$ follows (see \cite[Proposition~3.2]{Stasinski}). By \cite[Theorem~4.5]{Stasinski}, $\mcal{F}_{\OO_r}(\TT\times\OO_r)$ is a Cartan subgroup of $\GGamma_r$ and hence is equal to the centralizer of $s^*(\TT_1)$. Finally, the statement regarding the dimension of $\mcal{F}_r(\TT\times\OO_r)$ follows from \autoref{lem:greenberg-props}.(1).
\end{list}
\end{proof}

\begin{proof}[Proof of \autoref{theo:properties}.(1)]
The alternative proof of \cite[Ch.~III, \S~3.5, Proposition~1]{SteinbergConjugacy} shows that the minimal centralizer dimension of an element of $\g(\OO_r)$ is equal to that of a Cartan subgroup of $\GGamma_r$, provided that the Cartan subgroups of $\GGamma_r$ are abelian and that their union forms a dense subset of $\GGamma_r$. The former of these conditions holds by \cite[Theorem~4.5]{Stasinski}, and the latter by \cite[IV~12.1]{Borel}.
\end{proof}

\subsubsection{Regularity and the reduction maps}\label{subsubsection:reduction-map} The first step towards the proof of the second assertion of \autoref{theo:properties} is an analogous result to \cite[Lemma~3.5]{Hill} in the Lie-algebra setting. Following this, we use the properties of the Cayley map in order to transfer the result to the group setting and to deduce the equivalence of regularity of an element of $\ggamma_r$ and of its image in $\ggamma_1$.
\begin{lem}\label{lem:reg-element-reduction} Let $x\in \g(\OO)$ be fixed, and for any $r\in \dbN$ put $x_r=\eta_{r}(x)\in\g(\OO_r)$. Let $\CC_{\ggamma_r}(x_r)$ denote the Lie-algebra centralizer of $x_r$, i.e. $\CC_{\ggamma_r}(x_r)(A)=\set{y\in\ggamma_r(A)\mid \ad(x_r)(A)(y)=0},$ for any commutative unital $\bkk$-algebra $A$. The image of $\CC_{\ggamma_r}(x_r)$ under the connecting morphism $\eta_{r,1}^*$ is a $\bkk$-group scheme of dimension greater or equal to $n$.
\end{lem}
\begin{proof}


Assume towards a contradiction that the statement of the lemma is false, and let $r$ be minimal such that $\dim\eta_{r,1}^*\left(\CC_{\ggamma_r}(x_r)\right)<n$. Note that, since $\eta_{r,1}^*\circ\eta_{m,r}^*=\eta_{m,1}^*$ for all $m>r$ (by \cite[Proposition~3,\:\S~5]{GreenbergI}) we also have that $\dim\eta_{m,1}^*\left(\CC_{\ggamma_m}(x_m)\right)<n$ for all $m\ge r$.

Fix $m\ge r$, and consider the sequence of immersions
\begin{equation}\CC_{\ggamma_m}(x_m)\supseteq\CC_{\ggamma_m^1}(x_m)\supseteq\ldots\supseteq\CC_{\ggamma_{m}^{m-1}}(x_m)\supseteq 0,\end{equation}
where $\CC_{\ggamma_m^i}(x_m)=\CC_{\ggamma_m}(x_m)\cap\ggamma^i_m$. Then
\begin{equation}\label{equation:dim-filtration-lie}
\dim\CC_{\ggamma_m}(x_m)=\sum_{i=0}^{m-1}\left(\dim\CC_{\ggamma^{i}_m}(x_m) -\dim \CC_{\ggamma^{i+1}_m}(x_m)\right),\end{equation}
where $\ggamma_m^0=\ggamma_m$ and $\ggamma_m^m=\spec(\kappa(0))$. 

For any $0\le i\le m-1$, the map $v_{i,m}^*:\ggamma_{m-i}\to\gamma_{m}$ of \autoref{lem:verschiebung} restricts, by Assertion (3) of the lemma, to an isomorphism of abelian $\bkk$-group schemes $\CC_{\ggamma_{m-i}}(x_{m-i})\simeq \CC_{\ggamma_{m}^i}(x_m)$, which restricts further, by Assertion (4) of the lemma, to an isomorphism $\CC_{\ggamma^{i+1}_m}(x_m)\simeq ~\CC_{\ggamma_{m-i}^1}(x_{m-i})$. Using these isomorphisms and the exact sequence
\[0\to \CC_{\ggamma_{m-i}^1}(x_{m-i})\to \CC_{\ggamma_{m-i}}(x_{m-i})\xrightarrow{\eta_{m-i,1}^*}\eta_{m-i,1}^*\left(\CC_{\ggamma_{m-i}}(x_{m-i})\right)\to 0,\]
we deduce 
\begin{align}
\dim\CC_{\ggamma_m}(x_m)&=\sum_{i=0}^{m-1}\left(\dim\CC_{\ggamma_{m-i}}(x_{m-i})-\dim\CC_{\ggamma_{m-i}^1}(x_{m-i})\right)=\sum_{i=0}^{m-1}\dim\eta_{m-i,1}^*\left(\CC_{\ggamma_{m-i}}(x_{m-i})\right)\notag\\
&=\sum_{i=1}^{r-1}\dim\eta_{i,1}^*\left(\CC_{\ggamma_i}(x_i)\right)+\sum_{i=r}^{m}\dim\eta_{i,1}^*\left(\CC_{\ggamma_{i}}(x_i)\right)\notag\\
&\le d\cdot (r-1)+(n-\alpha)\cdot (m-r),\label{equation:centralizer-dim-ineq}
\end{align}\nopagebreak
for some integer $\alpha\ge 1$, where $d=\dim\ggamma_1=\dim\GGamma_1$.

For any $m\in\dbN$, by Property~\eqref{cayley-property-2} of the Cayley map and the preservation of open immersions of the Greenberg functor, the Cayley map restricts to a birational equivalence of the Lie-centralizer~$\CC_{\ggamma_m}(x_m)$ and the group-centralizer $\CC_{\GGamma_m}(x_m)$ of $x_m$.  In particular, by \autoref{theo:properties}.(1), we have that~$\dim\CC_{\ggamma_m}(x_m)=\dim\CC_{\GGamma_m}(x_m)\ge m\cdot n$. Manipulating the inequality \eqref{equation:centralizer-dim-ineq}, we get that
\begin{equation}\label{equation:ineq-reduction-dim}\alpha\cdot m\le d\cdot (r-1)-r\cdot(n-\alpha)
\end{equation}
for all $m>r$. A contradiction, since $m$ can be chosen to be arbitrarily large while the right-hand side of \eqref{equation:ineq-reduction-dim} remains constant.
\end{proof}


Using \autoref{lem:props-of-cayley-map}, we now pass to the group setting.
\begin{propo}\label{propo:reg-element-reduction-groups} Let $x\in\ggamma$ and $x_r=\eta_{r}(x)$ for all $r\in\dbN$. The group scheme $\eta_{r,1}^*\left(\CC_{\GGamma_r}(x_r)\right)$ is a linear algebraic $\bkk$-group of dimension greater or equal to $n$.
\end{propo}
\begin{proof}
Properties \eqref{cayley-property-2} and \eqref{cayley-property-3} of the Cayley map imply the commutativity of the square \eqref{equation:square-cay-reductions}
\begin{equation}
\label{equation:square-cay-reductions}
\begin{tikzpicture}
 \matrix (m) [matrix of math nodes,row sep=2em,column sep=4.5em,minimum width=2em,nodes=
        {minimum height=2em, 
        text height=1.5ex,
        text depth=.25ex,
        anchor=center}]
{\CC_{\ggamma_r}(x_r)&\CC_{\GGamma_r}(x_r)\\
\eta_{r,1}^*\left(\CC_{\ggamma_r}(x_r)\right)&\eta_{r,1}^*\left(\CC_{\GGamma_r}(x_r)\right).\\
};
\path[->]{
(m-1-1)	edge[dashed]	node[auto]	{$\widehat{\cay}_r$}	(m-1-2)
(m-2-1)	edge[dashed]	node[auto]	{$\cay_\bkk$}	(m-2-2)
(m-1-1)	edge[thin]	node[left]	{$\eta_{r,1}^*$}	(m-2-1)
(m-1-2)	edge[thin]	node[auto]	{$\eta_{r,1}^*$}	(m-2-2)
}
;
\end{tikzpicture}\end{equation}
A short computation, using Property \eqref{cayley-property-3}, shows that this square is cartesian. Thus, by \eqref{cayley-property-1}, and the properties of the fiber product, it follows that the two terms of the bottom row are of the same dimension.
\end{proof}

\begin{proof}[Proof of \autoref{theo:properties}.(2)]The assertion is proved by induction on $r$, similarly to \cite[Theorem~3.6]{Hill}, the case $r=1$ being trivially true. Consider the following exact sequence
\begin{equation}\label{equation:ses1}
\xymatrix{1\ar[r]& \CC_{\GGamma^1_r}(x_r)\ar[r]&\CC_{\GGamma_r}(x_r)\ar[r]^{\eta_{r,1}^*}&\CC_{\GGamma_1}(x_1).}
\end{equation}

Properties \eqref{cayley-property-1} and \eqref{cayley-property-2} imply that the map $\widehat{\cay}_r$ is defined on $\CC_{\GGamma^1_r}(x_r)$ and is mapped onto $\CC_{\ggamma^1_{r}}(x_{r})$. Combined with \autoref{lem:verschiebung}, we get that $\dim\CC_{\GGamma_r^1}(x_r)=\dim \CC_{\ggamma_{r-1}}(x_{r-1})$. Moreover, since $\DDelta_{r-1}\cap \CC_{\ggamma_{r-1}}(x_{r-1})$ is a non-trivial open subscheme of $\CC_{\ggamma_{r-1}}(x_{r-1})$, and is mapped by $\widehat{\cay}_r$ to an open subscheme of $ \CC_{\GGamma_{r-1}}(x_{r-1})$, we deduce the equality
\begin{equation}
\label{equation:dim-equality}
\dim \CC_{\GGamma_{r}^1}(x_r)=\dim\CC_{\GGamma_{r-1}}(x_{r-1}).
\end{equation}

If $x_1$ is regular then by induction we have that $\dim\CC_{\GGamma_{r-1}}(x_{r-1})=n(r-1)$ and hence, by \eqref{equation:ses1} and \eqref{equation:dim-equality},
\[\dim\CC_{\GGamma_r}(x_r)\le\dim\CC_{\GGamma_{r-1}}(x_{r-1})+\dim\CC_{\GGamma_1}(x_1)=r\cdot n.\]

Conversely, if $x_1$ is not regular, then by induction $x_{r-1}$ is not regular, and the dimension of $\CC_{\GGamma_{r-1}}(x_{r-1})$ is strictly greater than $n(r-1)$. By \autoref{propo:reg-element-reduction-groups} and \eqref{equation:dim-equality}, have
\[\dim\CC_{\GGamma_r}(x_r)=\dim\CC_{\GGamma_{r-1}}(x_{r-1})+\dim\eta_1\left(\CC_{\GGamma_r}(x_r)\right)>n(r-1)+n=n\cdot r,\]
and $x_r$ is not regular.
\end{proof}

Before discussing the final assertion of \autoref{theo:properties}, let us observe a simple corollary of \autoref{lem:reg-element-reduction}, which is the Lie-algebra version of the assertion.

\begin{corol}\label{corol:reduction-surjective-lie-alg} Let $r\in\dbN$ and $x_r\in\ggamma_r(\bkk)$ be regular. The restriction of $\eta_{r,1}$ to $\CC_{\g(\OO_r)}(x_r)$ is onto $\CC_{\g(\bkk)}(x_1)$, where $x_1=\eta_{r,1}(x_r)\in\g(\bkk)$.
\end{corol}
\begin{proof}
\autoref{theo:properties}.(2) implies that $x_1$ is regular and hence $\CC_{\ggamma_1}(x_1)(\bkk)=\CC_{\g(\bkk)}(x_1)$ is a $\bkk$-vector space of dimension $n=\dim\CC_{\GGamma_1}(x_1)$. By \autoref{lem:reg-element-reduction}, the $\bkk$-points of the image of $\CC_{\ggamma_r}(x_r)$ under $\eta_{r,1}^*$ comprise a subspace of $\CC_{\g(\bkk)}(x_1)$ of the same dimension.
\end{proof}
\subsubsection{The image of $\eta_{r,1}$ on $\CC_{\GGamma_r}(x_r)$}\label{subsubsection:image-eta_r,1}

To complete the proof of the third assertion of \autoref{theo:properties} we require the following proposition, which is stated here in a slightly more general setting than necessary at the moment, and will also be applied later on in the proof of \autoref{corol:fin-index-abelian}.

\begin{propo}\label{propo:centralizer-of-reg-over-bkk} Let $L$ be either $\bkk$ or $K^\alg$, and let $\mbf{H}=\GG\times \spec(L)$ and $\mbf{h}=\Lie(\HH)$ its Lie-algebra. Put $H=\HH(L)$ and $\mfr{h}=\mbf{h}(L)$. Let $x\in\mbf{h}(L)$ be regular. Then 
\[\CC_{H}(x)=\CC_{\HH}(x)^\circ(L)\cdot \ZZ(H),\]
where $\CC_{\HH}(x)^\circ$ is the connected component of $1$. In particular, $\abs{\CC_{H}(x):\CC_{\HH}(x)^\circ(L)}\le 2$ and $\CC_{H}(x)$ is abelian.
\end{propo}
\begin{proof}

Let $x=s+h$ be the Jordan decomposition of $x$, with $s,h\in\mfr{h}$, $s$ semisimple, $h$ nilpotent and $[s,h]=0$. Note that, as an element of $H$ commutes with $x$ if and only if it commutes with both $s$ and $h$, we have that $\CC_{H}(x)=\CC_{\CC_H(s)}(h)$. From \autoref{propo:centralizer-semisimple-element}, it follows that 
\begin{equation}\label{equation:centralizer-decomposition-reg}
\CC_{H}(x)=\CC_{\CC_H(s)}(h)=\prod_{j=1}^t\CC_{\GL_{m_j}\left(L\right)}\left(h\mid_{W_{\lambda_j}}\right)\times \CC_{\Delta(L)}\left(h\mid_{\ker(s)}\right),
\end{equation}
where $\Delta$ is a classical linear algebraic group over $L$ of automorphisms preserving a non-degenerate bilinear form on a subspace of $L^N$, and $\pm\lambda_1,\ldots,\pm\lambda_t$ are the non-zero  eigenvalues of $s$, as described in \autoref{propo:centralizer-semisimple-element}, with respective multiplicities $m_1,\ldots,m_t$, and $W_{\lambda_j}=\ker(s-\lambda_j\id)$. Additionally, by \cite[3.5, Proposition~5]{SteinbergConjugacy}, the restricted operators $h\mid_{W(\lambda_j)}$ and $h\mid_{\ker(s)}$ are regular as elements of the Lie-algebras of $\GL_{m_j}$ and of $\Delta$ over $L$, respectively.

By \cite[III,~3.2.2]{SteinbergSpringer} it is known that all factors in \eqref{equation:centralizer-decomposition-reg}, apart from $\CC_{\Delta}(h\mid_{\ker(s)})$, are connected. Furthermore, by \cite[III,~1.14]{SteinbergSpringer} and the assumption $\Char(L)\ne 2$, we have \[\CC_{\Delta}\left(h\mid_{\ker(s)}\right)=\CC_{\Delta}\left(h\mid_{\ker(s)}\right)^\circ\cdot \ZZ(\Delta),\]
(see \cite[I,~4.3]{SteinbergSpringer}). Taking into account the fact that, as $\Char(L)\ne 2$, $\ZZ(\Delta(L))$ is the finite group $\set{\pm 1}$, one easily deduces from this the equality
\[\CC_{H}(x)=\CC_{\HH}(x)^\circ(L)\cdot \ZZ(H).\]
Lastly, $\CC_{H}(x)^\circ$ is abelian by \cite[Corollary~1.4]{SteinbergSpringer}, and $\abs{\CC_H(x):\CC_{H}(x)^\circ}\le\abs{\ZZ(H)}=2$. \end{proof}

\begin{proof}[Proof of \autoref{theo:properties}.(3)]
By \autoref{propo:reg-element-reduction-groups} and Chevalley's Theorem \cite[IV,\:1.8.4]{EGA}, the image of $\CC_{\GGamma_r}(x_r)$ under $\eta_{r,1}^*$ contains the connected component $\CC_{\GGamma_1}(x_1)^\circ$ of the identity in $\CC_{\GGamma_1}(x)$. Additionally, the center $\ZZ(\GGamma_r)$ of $\GGamma_r$ is clearly contained in $\CC_{\GGamma_r}(x_r)$ and is mapped by $\eta_{r,1}^*$ onto $\ZZ(\GGamma_1)$. This implies the inclusion
\[\CC_{\GGamma_1}(x_1)\supseteq \eta_{r,1}^*\left(\CC_{\Gamma_r}(x_r)\right)\supseteq\left(\CC_{\Gamma_1}(x_1)\right)^\circ\cdot\ZZ(\GGamma_1).\]
Evaluating the above inclusions at $\bkk$-points, by \autoref{propo:centralizer-of-reg-over-bkk}, we deduce the equality.
\end{proof}

\subsubsection{Returning to the $\oo$-rational setting} In this section we prove \autoref{theo:inverse-lim}. An initial step towards this goal is to show that the third assertion of \autoref{theo:properties} remains true when replacing the groups $\GG(\OO_r)$ and Lie-rings $\g(\OO_r)$ with the group and Lie-rings of $\oo_r$-rational points, i.e. $G_r=\GG(\oo_r)$ and $\mfr{g}_r=\g(\oo_r)$. Given $1\le m\le r$, we write $G^m_r$ and $\mfr{g}^m_r$ to denote the congruence subgroup $\ker(G_r\xrightarrow{\eta_{r,m}}G_m)=G_r\cap\eta_{r,m}^{-1}(1)$ and congruence subring $\ker(\mfr{g}_r\xrightarrow{\eta_{r,m}}\mfr{g}_m)=\mfr{g}_r\cap\eta_{r,m}^{-1}(0)$, respectively.

 Recall that $\sigma:\OO\to \OO$ was defined in \autoref{subsubsection:artinian-rings} to be the local Frobenius automorphism of $\OO$ over $\oo$, given on its quotient $\bkk$ by $\sigma(\xi)=\xi^{\abs{\kk}}$. This automorphism gives rise to an automorphism of $\GG(\OO)$, and of its quotients $\GG(\OO_r)$ and their Lie-algebras. By definition, an element $x\in\mfr{g}_r$ is regular if and only if it is a regular $\sigma$-fixed element of $\g(\OO_r)=\ggamma_r(\bkk)$. We require the following variant of Lang's Theorem.

\begin{lem}\label{lem:lang} Let $r\in\dbN$ and let $x_r\in \mfr{g}_r$ be a regular element and $x_1=\eta_{r,1}(x_r)$. Given $g\in\CC_{G_1}(x_1)=\CC_{\GG(\bkk)}(x_1)\cap G_1$, let $F_g=\eta_{r,1}^{-1}(g)\cap \CC_{\GG(\OO_r)}(x_r)$, and let $\mcal{L}_g$ be the map defined by
\[h\mapsto h\cdot \sigma(h)^{-1}.\]
Then $\mcal{L}_g:F_g\to F_1$ is a well-defined surjective map.
\begin{proof}
The sets $F_{g'}$ ($g'\in\CC_{G_1}(x_1)$) are simply cosets of the subgroup $F_1=\CC_{\GGamma^1_r(\bkk)}(x_r)$. In particular, by \eqref{cayley-property-1} and \eqref{cayley-property-2}, the $F_{g'}$'s are the $\bkk$-points of algebraic varieties, isomorphic to $\CC_{\ggamma^1_r(\bkk)}(x_r)$ and hence affine ${(r-1)n}$-dimensional spaces over~$\bkk$.

Since the reduction map $\eta_{r,1}$ commutes with the Frobenius maps, and since $g$ is assumed fixed by $\sigma$, we have that $\mcal{L}_g$ is well-defined. The surjectivity of $\mcal{L}_g$ now follows as in the proof of the classical Lang Theorem \cite{Lang}, using the fact that $F_1$ is a connected linear algebraic group over $\bkk$ (see also \cite[I,\:2.2]{SteinbergSpringer} and \cite[\S~3]{GreenbergII}).
\end{proof}
\end{lem}

\begin{corol}\label{corol:surjective-Gm} Let $x_r\in\mfr{g}_r$ be regular and $x_1=\eta_{r,1}(x_r)$. The restriction of $\eta_{r,1}$ to $\CC_{G_r}(x_r)$ is onto $\CC_{G_1}(x_1)$.
\end{corol}
\begin{proof}
\autoref{lem:lang} and \autoref{theo:properties}.(3) imply that for any $g\in\CC_{G_1}(x_1)$, there exists an element $h\in \CC_{\GGamma_r}(x_r)$ such that $\eta_{r,1}(h)=g$ and such that $\mcal{L}_g(h)=h\sigma (h)^{-1}=1$. In particular, $h$ is fixed under $\sigma$ and hence $h\in \CC_{G_r}(x_r)\cap \eta_{r,1}^{-1}(g)$.
\end{proof}


Another necessary ingredient in the proof of \autoref{theo:inverse-lim} is the connection between the groups $\CC_{G_r}(x_r)$ and $\CC_{G_m}(x_m)$, where $m\le r$ and $x\in\mfr{g}$ is such that $x_r$ is regular.
\begin{lem}\label{lem:surjective-r-to-m} Let $r\in\dbN$ and $x_r\in\mfr{g}_r$ be regular. For any $1\le m\le r$ write $x_m=\eta_{r,m}(x_r)$.
\begin{enumerate}
\item The map $\eta_{r,m}:\CC_{\mfr{g}_r}(x_r)\to \CC_{\mfr{g}_m}(x_m)$ is surjective.
\item The map $\eta_{r,m}:\CC_{G_r}(x_r)\to \CC_{G_m}(x_m)$ is surjective.
\end{enumerate}
\end{lem}
\begin{proof}
We prove both assertions by induction on $m$.
\begin{list}{\arabic{list}.}{\usecounter{list}\setlength{\leftmargin}{5pt}
   \setlength{\itemsep}{2pt} \setlength{\parsep}{0pt}\setlength{\labelwidth}{-5pt}}
\item The case $m=1$ follows in~\autoref{corol:reduction-surjective-lie-alg} and Lang's Theorem, as $\CC_{\ggamma_1}(x_1)$ and $\eta^*_{r,1}(\CC_{\ggamma_r}(x_r))$ are both affine $n$-spaces over $\bkk$. Consider the commutative diagram in~\eqref{figure:reduction-lie-algebra}, in which both rows are exact by induction hypothesis.
\begin{equation}
\begin{tikzpicture}
 \matrix (m) [matrix of math nodes,row sep=1.5em,column sep=3em,minimum width=2em,
 nodes=
        {minimum height=2em, 
        text height=1.5ex,
        text depth=.25ex,
        anchor=center}]
{\CC_{\mfr{g}^{m-1}_{r}}(x_r)&\CC_{\mfr{g}_r}(x_r)&\CC_{\mfr{g}_{m-1}}(x_{m-1})&0\\
\CC_{\mfr{g}^{m-1}_{m}}(x_m)&\CC_{\mfr{g}_m}(x_m)&\CC_{\mfr{g}_{m-1}}(x_{m-1})&0\\
};
\path[->,font=\scriptsize]
(m-1-1)	edge[thin]	node[auto]	{} 	(m-1-2)
(m-1-2)	edge[thin]	node[auto]	{$\eta_{r,m-1}$} 	(m-1-3)
(m-1-3)	edge[thin]	node[auto]	{} 	(m-1-4)
(m-2-1)	edge[thin]	node[auto]	{} 	(m-2-2)
(m-2-2)	edge[thin]	node[below]	{$\eta_{m,m-1}$} 	(m-2-3)
(m-2-3)	edge[thin]	node[auto]	{} 	(m-2-4)
(m-1-2) edge[thin]	node[auto]	{$\eta_{r,m}$}	(m-2-2)
(m-1-1) edge[thin]	node[auto]	{}	(m-2-1)
;
\path[-]
(m-1-3)	edge[double equal sign distance, thin]	node	{} 	(m-2-3)
(m-1-4) edge[double equal sign distance, thin]	node[auto]	{}	(m-2-4);
\end{tikzpicture}\label{figure:reduction-lie-algebra}\end{equation}
By the Four Lemma (on epimorphisms), in order to prove the surjectivity of the map $\eta_{r,m}:\CC_{\mfr{g}_r}(x_r)\to \CC_{\mfr{g}_m}(x_m)$, it suffices to show that the restricted map $\eta_{r,m}:\CC_{\mfr{g}^{m-1}_r}(x_r)\to \CC_{\mfr{g}_m^{m-1}}(x_m)$ is surjective. This follows from the commutativity of the square in \eqref{figure:square-lie}, in which the maps on the top and bottom rows are given the $\oo$-module isomorphism $y\mapsto \pi^{m-1}y$ (cf. \autoref{lem:verschiebung}), and the map on the left column is surjective by the base of induction.
\begin{equation}
\begin{tikzpicture}
\matrix (m) [matrix of math nodes, row sep=1.5em, column sep=3em, minimum width=2em,nodes=
        {minimum height=2em, 
        text height=1.5ex,
        text depth=.25ex,
        anchor=center}]
{\CC_{\mfr{g}_{r-m+1}}(x_{r-m+1})&\CC_{\mfr{g}^{m-1}_r}(x_r)\\
\CC_{\mfr{g}_1}(x_1)&\CC_{\mfr{g}_m^{m-1}}(x_m)\\
}
;
\path[->,font=\scriptsize]
(m-1-1)	edge[thin]	node[auto]	{$\sim$}	(m-1-2)
(m-2-1)	edge[thin]	node[auto]	{$\sim$}	(m-2-2)
(m-1-2)	edge[thin]	node[right]	{$\eta_{r,m}$}	(m-2-2)
;

\path[->>,font=\scriptsize]
(m-1-1)	edge[thin]	node[left]	{$\eta_{r-m+1,1}$}	(m-2-1)
;
\end{tikzpicture}\label{figure:square-lie}\end{equation}
\item  In the current setting, one invokes \autoref{lem:lang} in order to prove the induction base $m=1$. The case $m>1$ is handled in a manner completely analogous to the first case, applying the Four Lemma for a suitable diagram of groups. The main difference from the previous case is that in proving the surjectivity of the map $\eta_{r,m}:\CC_{G_r^{m-1}}(x_r)\to \CC_{G_m^{m-1}}(x_m)$, one considers the commutative square in \eqref{figure:square-groups}
in which the leftmost vertical arrow is shown to be surjective in the previous case, and the horizontal arrows are given by the suitable Cayley maps. Note that the fact that the top horizontal arrow in \eqref{figure:square-groups} is not necessarily a group homomorphism does not affect the proof of the assertion.
\end{list}
\begin{equation}
\begin{tikzpicture}
\matrix (m) [matrix of math nodes, row sep=1.5em, column sep=3em, minimum width=2em,nodes=
        {minimum height=2em, 
        text height=1.5ex,
        text depth=.25ex,
        anchor=center}]
{\CC_{\mfr{g}^{m-1}_r}(x_{r})&\CC_{G^{m-1}_r}(x_r)\\
\CC_{\mfr{g}^{m-1}_m}(x_m)&\CC_{G_m^{m-1}}(x_m)\\
}
;
\path[->,font=\scriptsize]
(m-1-1)	edge[thin]	node[auto]	{$\cay_{r}$}	(m-1-2)
(m-2-1)	edge[thin]	node[auto]	{$\cay_{m}$}	(m-2-2)
(m-1-2)	edge[thin]	node[right]	{$\eta_{r,m}$}	(m-2-2)
;

\path[->>,font=\scriptsize]
(m-1-1)	edge[thin]	node[left]	{$\eta_{r,m}$}	(m-2-1)
;
\end{tikzpicture}\label{figure:square-groups}\end{equation}
\end{proof}

\begin{proof}[Proof of \autoref{theo:inverse-lim}]\begin{list}{\arabic{list}.}{\usecounter{list}\setlength{\leftmargin}{0pt}\setlength{\itemsep}{2pt}\setlength{\labelwidth}{-5pt}}
\item 
Given $g_r\in \CC_{G_r}(x_r)$ one inductively invokes~\autoref{lem:surjective-r-to-m} to construct a converging sequence $(g_m)_{m\ge r}$ such that $g_m\in \CC_{G_m}(\eta_m(x))$ and such that $\eta_{m',m}(g_{m'})=g_m$ for all $m'\ge m\ge r$. The limit $g=\lim_m g_m$ is easily verified to be an element of $\CC_{G}(x)$, which is mapped by $\eta_r$ to $g_r$. 

\item By \autoref{theo:properties}, it suffices to consider the case where $x_1=\eta_{1}(x)\in\g(\kk)$ is regular. By \cite[(2.5.2)]{McNinch}, under this assumption, we have that \[\dim\CC_{\GG\times K^{\alg}}(x)=\dim\left(\CC_{\GG\times\OO}(x)\times K^{\mrm{unr}}\right)\le\dim\left(\CC_{\GG\times\OO}(x)\times \bkk\right)=\dim\CC_{\GGamma_1}(x_1)=n,\]
as $\CC_{\GG\times\OO}(x)\times K^{\rm unr}$ and $\CC_{\GG\times\OO}(x)\times\bkk$ are, respectively, the generic and special fiber of $\CC_{\GG\times\OO}(x)$.  On the other hand, by \cite[3.5, Proposition~1]{SteinbergConjugacy}, the minimum value of centralizer dimension of an element of $\mfr{g}$ is $n=\rk(\GG)$. Hence, $x$ is regular.
\end{list}
\end{proof}

Finally, we deduce \autoref{corol:fin-index-abelian}.
\begin{proof}[Proof of \autoref{corol:fin-index-abelian}]

The regularity of $x$ in $\mfr{g}$, and \autoref{propo:centralizer-of-reg-over-bkk} (applied for $L=K^{\alg}$), imply that the centralizer of $x$ in $\GG(K^\alg)$ is an abelian group. In particular, it follows from this that the group $\CC_{G}(x)$ is abelian as well, and consequently, by \autoref{theo:inverse-lim}.(1), so are its quotient groups $\CC_{G_r}(x_r)$ for all $r\in\dbN$. 
\end{proof}

\subsection{Regular characters}\label{subsection:reg-characters} At this point, our description of the regular elements of the Lie-algebras $\mfr{g}_r$ is sufficient in order to initiate the description of regular characters of $G$ and to prove \autoref{mainthm:enumeration+dimension} and \autoref{corol:reg-zeta-function}. To do so, we prove the following variant of \cite[Theorem~3.1]{KOS}.

\begin{theo}\label{theo:charcaters-of-Grm} Let $\Omega\subseteq\mfr{g}_1$ be a regular orbit and let $r\in\dbN$ and $m=\lfloor\frac{r}{2}\rfloor$.
\begin{enumerate}
\item The set $\irr(G^m_r\mid\Omega)$ of characters of $G^m_r=\ker(G_r\to G_m)$ which lie above the regular orbit $\Omega$ consists of exactly $q^{n(r-m-1)}$ orbits for the coadjoint action of $G_r$.
\item Given a character $\sigma\in \irr(G^m_r\mid \Omega)$, the set of irreducible characters of $G_r$ whose restriction to $G^m_r$ has $\sigma$ as a constituent is in bijection with the Pontryagin dual of $\CC_{G_m}(x_m)$, for $x_m\in\mfr{g}_m$ any element such that $\eta_{m,1}(x_m)\in \Omega$.
\item Any such character $\sigma\in\irr(G^m_r\mid\Omega)$ extends to its inertia group $I_{G_r}(\sigma)$. In particular, each such extension induces to a regular character of $G_r$. 

\end{enumerate}

\end{theo}

Note that the first assertion of \autoref{mainthm:enumeration+dimension} follows from Assertions (1) and (2) of \autoref{theo:charcaters-of-Grm} and \autoref{corol:surjective-Gm}. The second assertion of \autoref{mainthm:enumeration+dimension} follows from the Assertion (3) of \autoref{theo:charcaters-of-Grm} and \eqref{equation:inertia} below.

The proof of \autoref{theo:charcaters-of-Grm} follows the same path as \cite[\S~3]{KOS}. For the sake of brevity, rather then rehashing the proof appearing in great detail in \textit{loc. cit.}, our focus for the remainder of this section would be on setting up the necessary preliminaries and state the necessary modification required in order to adapt the construction of \cite{KOS} to the current setting.

Recall that the group $G=\GG(\oo)$ and $\mfr{g}=\g(\oo)$ are naturally embedded in the matrix algebra $\matr_N(\oo)$ (see~\autoref{subsection:group-notation}). Similarly, the congruence quotients $G_r$ and $\mfr{g}_r$ are embedded in $\matr_{N}(\oo_r)$. From here on, all computation are to be understood in the framework of the embedding of the given groups and Lie-rings in their respective matrix algebras. 

\subsubsection{Duality for Lie-rings} 

The Lie-algebra $\mfr{g}=\g(\oo)\subseteq \matr_N(\oo)$ is endowed with a symmetric bilinear $G(\oo)$-invariant form 
\[\kappa:\mfr{g}\times\mfr{g}\to\oo,\quad (x,y)\mapsto \Tr(xy).\]
Note that, by the assumption $p=\Char(\kk)$ is odd, the form $\kappa$ reduces to a non-degenerate form on $\mfr{g}_1$ (see \cite[Lemma~5.3]{SteinbergSpringer}), and hence $\set{x\in \mfr{g}\mid\kappa(x,y)\in \pp\text{ for all }y\in\mfr{g}}=\pi\mfr{g}$. Fixing a non-trivial character $\psi:K\to\dbC^\times$ with conductor $\oo$ (see e.g. \cite[\S~5.3]{AKOVA2}), for any $r\in\dbN$, we have a well-defined map \begin{equation}
\label{equation:Lie-ring-duality}
\mfr{g}_r\to \pd{\mfr{g}_r},\quad y\mapsto\varphi_y\text{ where }\varphi_y(x)=\psi(\pi^{-r}\kappa(x,y)).
\end{equation}
Furthermore, by the assumption $\pi^{-1}\oo\not\subseteq\ker(\psi)$, the map above induces a $G_r$-equivariant bijection of $\mfr{g}_r$ with its Pontryagin dual $\pd{\mfr{g}_r}$.

\subsubsection{Exponential and logarithm}

Let $m,r\in\dbN$ with $\frac{r}{3}\le m\le r$. The truncated exponential map, defined by
\begin{equation}\notag
\label{equation:exp-defi}
\exp(x)=1+x+\frac{1}{2}x^2\quad(x\in\mfr{g}^m_r),
\end{equation}
is a well-defined bijection of $\mfr{g}^m_r$ onto the group $G^m_r$, and is equivariant with respect to the adjoint action of $G_r$, with an inverse map given by 
\begin{equation}\label{equation:log-defi}\notag
\log(1+x)=x-\frac{1}{2}x^2\quad(1+x\in G^m_r).
\end{equation}

In the case where $\frac{r}{2}\le m$, the exponential map is simply given by $\exp(x)=1+x$ and defines an isomorphism of abelian groups $\mfr{g}^m_r\xrightarrow{\sim}G^m_r$.
In the more general setting we have the following. 
\begin{lem}\label{lem:ad-relations} Let $m,r\in\dbN$ be such that $\frac{r}{3}\le m\le r$. For any $x,y\in\mfr{g}^m_r$,
\begin{equation}\label{equation:ad-log-exp}\notag
\log\left(\left(\exp(x),\exp(y)\right)\right)=\left[x,y\right],
\end{equation}
where $(\exp(x),\exp(y))$ denotes the group commutator of $\exp(x)$ and $\exp(y)$ in $G^m_r$. 
Furthermore, the following truncated version of the Baker-Campbell-Hausdorff formula holds
\begin{equation}\notag
\log\left(\exp(x)\cdot\exp(y)\right)=x+y+\frac{1}{2}\left[x,y\right].
\end{equation}
\end{lem}
The formulae in \autoref{lem:ad-relations} may be verified by direct computation; their proof is omitted.

\subsubsection{Characters of $G^{\lfloor{r/2}\rfloor}_r$}\label{subsubsection:character-Grm}

Fix $r\in\dbN$ and put ${m'}=\lfloor\frac{r}{2}\rfloor$ and $m=\lceil\frac{r}{2}\rceil=r-{m'}$. As mentioned above, the exponential map on $\mfr{g}^{m}_r$ is given by $x\mapsto 1+x:\mfr{g}^{m}_r\to G^{m}_r$ and defines a $G_r$-equivariant isomorphism of abelian groups. Taking into account the module isomorphism $x\mapsto \pi^{m} x:\mfr{g}_{{m'}}\to \mfr{g}_{r}^{m}$ and \eqref{equation:Lie-ring-duality} we obtain a $G_r$-equivariant bijection
\begin{equation}\label{equation:Phi-definition}
 \Phi\quad:\quad\mfr{g}_{m'}\to \pd{\mfr{g}_{m'}}\to \pd{\mfr{g}_r^{m}}\to\irr(G^{m}_r),
\end{equation}
 given explicitly by $\Phi(y)(1+x)=\varphi_y(\pi^{-m}x)$, for $y\in\mfr{g}_{m'}$ and $x\in\mfr{g}^{m}_r$. In the case where $r=2m'$ deduce the following.

\begin{lem}\label{lem:half-level-characters-even-case} Assume $r=2m$ is even. The map $\Phi$ defined in \eqref{equation:Phi-definition} is a $G_r$-equivariant bijection of $\irr(G^{m'}_r)$ and $\mfr{g}_{m'}$. 
\end{lem}

In the case where $r=2m'+1$, the irreducible characters of $G^{m'}_r$ are classified in terms of their restriction to $G^{m}_r$, using the method of Heisenberg lifts, which we briefly recall here. For a more elaborate survey we refer to \cite[\S~3.2]{KOS} and \cite[Ch. 8]{BushnellFroelich}. 

Let $\vartheta\in\irr(G^{m}_r)$ be given, and let $y\in\mfr{g}_{m'}$ be such that $\vartheta=\Phi(y)$. Note that, as the group $G^{m}_r$ is central in $G^{m'}_r$ and $(G^{m'}_r,G^{m'}_r)\subseteq G^{m}_r$, the following map is a well defined alternating $\dbC^\times$-valued bilinear form
\[B_{\vartheta}:G^{m'}_{r}/G^{m}_r\times G^{m'}_r/G^{m}_r\to \dbC^\times ,\quad B_{\vartheta}(x_1 G^{m}_r,x_2G^{m}_r)=\vartheta\left(\left(x_1 ,x_2\right)\right).\]
Using the definition of $\Phi(y)=\vartheta$ and the explicit isomorphism $x\mapsto \exp(\pi^rx):\mfr{g}_1\to G^{m'}_{m}=G^{m'}_r/G^{m}_r$, we obtain an alternating bilinear form $\beta_y:\mfr{g}_1\times \mfr{g}_1\to \kk$ given by $\beta_y(x_1,x_2)=\Tr(\eta_{{m'},1}(y)\cdot [x_1,x_2])$, such that the diagram in \eqref{equation:g1-Grr'-square} commutes.
\begin{equation}\label{equation:g1-Grr'-square}\begin{tikzpicture}
\matrix (m) [matrix of math nodes, row sep=1.5em, column sep=.25em, minimum width=2em,nodes=
        {minimum height=2em, 
        text height=1.5ex,
        text depth=.25ex,
        anchor=center}]
{G^{m'}_{m}&\times&G^{m'}_{m}&~&~&\dbC^\times\\
\mfr{g}_1&\times&\mfr{g}_1&~&~&\kk
\\}
;
\path[->,font=\scriptsize]
(m-2-1)	edge[thin]	node[right]	{\rotatebox{90}{$\sim$}}	(m-1-1)
(m-2-3)	edge[thin]	node[right]	{\rotatebox{90}{$\sim$}}	(m-1-3)
(m-1-3)	edge[thin]	node[above]	{$B_\vartheta$}				(m-1-6)
(m-2-3)	edge[thin]	node[above]	{$\beta_y$}					(m-2-6)
(m-2-6)	edge[thin]	node[right]	{$\psi(\pi^{-1}(\cdot))$}					(m-1-6)
;
\end{tikzpicture}\end{equation}

A short computation, using the non-degeneracy of the trace and the definition of $\beta_y$, shows that the radical of this form coincides with the centralizer sub-algebra $\CC_{\mfr{g}_1}(\eta_{{m'},1}(y))$ of $\mfr{g}_1$ (see \cite[p.~125]{KOS}). Let $\mfr{R}_y$ and $R_y$ denote the preimages of $\mfr{m}_y$ in $\mfr{g}_r^{m'}$ and in $G^{m'}_r$ under the associated quotient maps. Let $\mfr{m}_y\subseteq \mfr{j}\subseteq\mfr{g}_1$ be a maximal subspace such that $\beta_y(\mfr{j},\mfr{j})=\set{0}$ (i.e. such that $\mfr{j}/\mfr{m}_y$ is a maximal isotropic subspace of $\mfr{g}_1/\mfr{m}_y$), and let $\mfr{J}\subseteq \mfr{g}^{m'}_r$ and $J\subseteq G^{m'}_r$ be the corresponding preimages of $\mfr{j}$; see \eqref{figure:heisenberg-diagram}.

\begin{figure}[H]
\begin{equation}
\begin{tikzpicture}[description/.style={fill=white,inner sep=2pt}]

\matrix (m) [matrix of math nodes, row sep=1.5em,
column sep=2em, text height=1.5ex, text depth=0.25ex, nodes=
        {minimum height=1em, 
        text height=1.5ex,
        text depth=.25ex,
        anchor=center}]
{G^{{m'}}_r&\mfr{g}_{r}^{{m'}}&\mfr{g}_1\\
J&\mfr{J}&\mfr{j}\\
R_y&\mfr{R}_y&\mfr{m}_y\\
G_{r}^{m}&\mfr{g}_{r}^{m}&0
\\};

\path[-,font=\scriptsize]
(m-1-1)	edge[thin]	node[auto]	{}	(m-2-1)
(m-2-1)	edge[thin]	node[auto]	{}	(m-3-1)
(m-3-1)	edge[thin]	node[auto]	{}	(m-4-1)
(m-1-2)	edge[thin]	node[auto]	{}	(m-2-2)
(m-2-2)	edge[thin]	node[auto]	{}	(m-3-2)
(m-3-2)	edge[thin]	node[auto]	{}	(m-4-2)
(m-1-3)	edge[thin]	node[auto]	{}	(m-2-3)
(m-2-3)	edge[thin]	node[auto]	{}	(m-3-3)
(m-3-3)	edge[thin]	node[auto]	{}	(m-4-3)
;
\path[->,font=\scriptsize]
(m-1-1)	edge[dashed]	node[below]	{$\llog$}	(m-1-2)
(m-2-1)	edge[dashed]	node[below]	{$\llog$}	(m-2-2)
(m-3-1)	edge[dashed]	node[below]	{$\llog$}	(m-3-2)

(m-4-1)	edge[thin]	node[below]	{$\llog$} node[above]	{$\sim$}	(m-4-2)
(m-1-2)	edge[thin]	node[auto]	{}	(m-1-3)
(m-2-2)	edge[thin]	node[auto]	{}	(m-2-3)
(m-3-2)	edge[thin]	node[auto]	{}	(m-3-3)
(m-4-2)	edge[thin]	node[auto]	{}	(m-4-3)
;
\end{tikzpicture}
\label{figure:heisenberg-diagram}\end{equation}
\end{figure}

Let $\theta=\vartheta\circ \exp$ be the pull-back of $\vartheta$ to $\mfr{g}^{m}_r$. By virtue of the commutativity of $\mfr{R}_y$, the character $\theta'$ extends to a character of $\mfr{R}_y$ in $\abs{\mfr{R}_y:G^{m}_r}=\abs{\mfr{m}_y}$ many ways. By \autoref{lem:ad-relations}, given such an extension $\theta'\in\pd{\mfr{R}_y}$, the map $\vartheta':R_y\to\dbC^\times$ is a character of $R_y$. Thus, the character $\vartheta$ admits $\abs{\mfr{m}_y}$ many extensions to $R_y$. 

\begin{lem}\label{lem:heisenberg}
\begin{enumerate}
\item Any extension $\vartheta'\in\irr(R_y)$ of $\vartheta$ extends further to a character $\vartheta''\in\irr(J)$.
\item The induced character $\sigma=(\vartheta'')^{G^{m'}_r}$ is irreducible and is independent of the choice of extension $\vartheta''$ and of $\mfr{j}$.
\item The character $\sigma$ is the unique character of $G^{m'}_r$ whose restriction to $R_y$ contains $\vartheta'$. Furthermore, all irreducible characters of $G^{m'}_r$ which lie above $\vartheta$ are obtained in this manner.
\end{enumerate}
\end{lem}

\begin{proof}
The triple $(G^{m'}_r,R_y,\vartheta')$ satisfies the hypothesis of \cite[\S~8.3]{BushnellFroelich}, and the alternating bilinear form $\beta_y$ (which corresponds to $h_\chi$ in the notation of \textit{loc. cit.}) reduces to a non-degenerate form on the elementary abelian group $G^{m'}_r/R_y\simeq \mfr{g}_1/\mfr{r}_y$. The subgroup $J\subseteq G^{m'}_r$ may be identified with the group denoted in \cite[Proposition~8.3.3]{BushnellFroelich} by $G_1$, and the extension of $\vartheta'$ to an irreducible character of $J$ exists by virtue of $J/\ker(\vartheta')$ being finite and abelian. The irreducibility and independence of the choice of $J$ are shown within the proof of \cite[Proposition~8.3.3]{BushnellFroelich}, as well as uniqueness of $\sigma$ as the only irreducible character of $G^{m'}_r$ whose restriction to $R_y$ contains $\vartheta'$. The final assertion, that all characters of $G^{m'}_r$ lying above $\vartheta$ are obtained in this manner is obvious, as the restriction of any such character of $G^{m'}_r$ to $R_y$ necessarily contains an extension $\vartheta'$ of $\vartheta$.
\end{proof}

\subsubsection{Inertia subgroups in $\GG(\oo_r)$ of regular characters}

The final ingredient required in order to implement the construction of \cite{KOS} to the current setting is a  structural description of the inertia subgroup of a character of $G^{\lceil r/2\rceil}_r$ lying below a regular character of level $\ell=r+1$. As in the previous section, put ${m'}=\lfloor\frac{r}{2}\rfloor$ and $m=\lceil\frac{r}{2}\rceil$, and let $\vartheta\in\irr(G^{m}_r)$. Recall that the inertia subgroup of $\vartheta$ in $G_r$ is defined by
\begin{equation}
\label{equation:inertia-subgroup-defi}
I_{G_r}(\vartheta)=\set{g \in G_r\mid \vartheta(g^{-1}xg)=\vartheta(x)\text{ for all }x\in G^{m}_r}.
\end{equation}

By \autoref{subsubsection:character-Grm}, there exists a unique $y\in\mfr{g}_{{m'}}$ such that $\vartheta=\Phi(y)$. Moreover, letting $\hat{y}_r\in\mfr{g}_r$ be an arbitrary lift of $y_{m'}$ to $\mfr{g}_r$, we have that

%
\begin{equation}
\label{equation:inertia}I_{G_r}(\vartheta)=G^{m}_r\cdot \CC_{G_r}(\hat{y}_r).\end{equation}
Indeed, the only non-trivial step to proving \eqref{equation:inertia} is the inclusion $\subseteq$, which from follows \autoref{lem:surjective-r-to-m}, as both hands of the equation are mapped by $\eta
_{r,m}$ onto the group $\CC_{G_{m}}(\eta_{r,m}(\hat{y}_r))$. 

\begin{proof}[Proof of \autoref{theo:charcaters-of-Grm}] 
A short computation, proves that the set $\tilde{\Omega}=\eta_{{m'},1}^{-1}(\Omega)$ consists of $q^{n({m'}-1)}$ distinct adjoint orbits for the action of $G_{m'}$, and hence for the action of $G_r$ as well. Indeed, $\tilde{\Omega}$ is a $G_{m'}$-stable set of order $\abs{\Omega}\cdot \abs{G^1_{m'}}=\abs{\Omega}q^{d(m'-1)}$, invoking the bijection $\mfr{g}^1_{m'}\to G^1_{m'}$ induced by the Cayley map, and each of the orbits $G_{m'}$-orbits in $\tilde{\Omega}$ has cardinality \[\abs{G_{m'}:\CC_{G_{m'}}(x)}=\abs{G_1:\CC_{G_1}\left(\eta_{m',1}\left(x\right)\right)}\cdot \abs{G^1_{m'}:\CC_{G^1_{m'}}(x)}=\abs{\Omega}\cdot q^{(d-n)(m'-1)},\] by \autoref{corol:surjective-Gm}, for any $x\in \tilde{\Omega}$. By the $G_r$-equivariance of the map $\Phi$, defined in \autoref{subsubsection:character-Grm}, it follows that the set $\irr(G^{m}_r\mid \Omega)$ consists of $q^{n({m'}-1)}$ coadjoint orbits of $G_r$. In the case where $r$ is even, the first assertion of \autoref{theo:charcaters-of-Grm} follows from \autoref{lem:half-level-characters-even-case}, since $m=r-m$. In the case of $r$ odd, by \autoref{lem:heisenberg}, and by regularity of the elements of $\Omega$, any character in $\irr(G_r^{m}\mid \Omega)$ extends to $G^{m'}_r$ in exactly $q^n$-many ways. Thus, the number of coadjoint $G_r$-orbits in $\irr(G^{m'}_r)$ is $q^{n({m'}-1)+n}=q^{n(r-{m'}-1)}$, whence the first assertion.

The second assertion of \autoref{theo:charcaters-of-Grm} follows from the third assertion, \eqref{equation:inertia} and \cite[Corollary~6.17]{Isaacs}. 

Lastly, for the proof of the third assertion of \autoref{theo:charcaters-of-Grm}, we refer to \cite[\S~3.5]{KOS} for the explicit construction, in the analogous case of $\GL_n(\oo)$ and $\UU_n(\oo)$, of an extension of a character $\sigma\in\irr(G^{m'}_r)$ to its inertia subgroup $I_{G_r}(\sigma)$. Note that the construction of \textit{loc. cit.} can be applied verbatim to the present setting, invoking the fact the $I_{G_r}(\sigma)$ is generated by two abelian subgroups, one of which is normal in $G_r$ (\eqref{equation:inertia} and \autoref{corol:fin-index-abelian}) in the generality of classical groups.
\end{proof}

\section{The symplectic and orthogonal groups}\label{section:classical-groups}

\subsection{Summary of section}\label{subsection:summary-classical}
In this section we compute the regular representation zeta function of classical groups of types $\msf{B}_n,\msf{C}_n$ and $\msf{D}_n$. Following \autoref{corol:reg-zeta-function}, to do so, we classify the regular orbits in the space of orbits $\Ad(G_1)\backslash \mfr{g}_1$ and compute their cardinalities, in order to obtain a formula for the Dirichlet polynomial 
\begin{equation*}
\diri_{\mfr{g}}(s)=\sum_{\Omega\in X}{\abs{G_1}}\cdot\abs{\Omega}^{-(s+1)}.
\end{equation*}

As it turns out, the cases where $\GG$ is of type $\msf{B}_n$ or $\msf{C}_n$, i.e. $\GG=\SO_{2n+1}$ or $\GG=\Sp_{2n}$, can be handled simultaneously, and are analyzed in \autoref{subsection:sp2n-so2n+1}. The case of the groups of the form $\msf{D}_n$, i.e. even-dimensional orthogonal groups, is slightly more intricate. The analysis of this case is carried out in \autoref{subsection:so2n}. The main difference between the two cases lies in the fact that regularity of an element of the Lie-algebras $\mfr{sp}_{2n}(\kk)$ and $\mfr{so}_{2n+1}(\kk)$ is equivalent to it being give by a regular matrix in $\matr_N(\kk)$; see \autoref{propo:reg-equiavlent} (also, cf. \cite[\S~5]{TakaseGeneral}). This equivalence fails to hold for even-orthogonal groups; see \autoref{lem:nilpotent-reg-glN-not-reg-soN} below. Nevertheless, in both cases, we obtain a classification of the regular orbits in the Lie-algebra $\mfr{g}_1$ in terms of the minimal polynomial of the elements within the orbit.

Recall that two matrices $x,y\in\matr_N(\kk)$ are said to be \textbf{similar} if there exists a matrix $g\in\GL_N(\kk)$ such that $y=gxg^{-1}$. Our description of regular orbits of $\mfr{g}_1$ follows the following steps.
\begin{enumerate}
\item Classification of all similarity classes in $\mfr{gl}_N(\kk)$ which intersect the set of regular elements in $\mfr{g}_1$ non-trivially;
\item Description of the intersection of such a similarity class with $\mfr{g}_1$ as a union of $\Ad(G_1)$-orbits;
\item Computation of the cardinality of the $\Ad(G_1)$-orbit of each regular element.
\end{enumerate}

A rich theory of centralizers and conjugacy classes in classical groups over finite fields already exists, most notably Wall's extensive analysis in \cite[\S~2.6]{WallIsometries}. The enumeration of elements of a finite classical group $G_1$ whose representing matrix is cyclic (i.e. regular when considered as an element of $\GL_N(\bkk)$) was addressed in \cite{NeumannPraeger-Classical} and \cite{FNP-Generating} where the proportion of such elements in $G_1$, for all classical groups, was estimated and its limit as $\rk(\GG)$ tends to infinit was computed using generating functions. The precise number of regular semisimple conjugacy classes was computed, again using generating functions, in \cite{FulmanGuralnick}, where the discrepancy between regularity of semisimple elements of the even dimensional orthogonal groups and of regularity of their representing matrices in $\GL_N(\bkk)$ is determined (see \cite[Lemma~5.1]{FulmanGuralnick}). In the case of the symplectic group, the equivalence of regularity of an element of $\Sp_{2n}(\kk)$ and of its representing matrix in $\GL_{2n}(\bkk)$ was noted in \cite[\S~1.1]{FNP-Generating}. Examples of regular elements of $\SO_{2n}(\kk)$ which do not satisfy this equivalence appear in \cite[Note~8.1]{NeumannPraeger-Cyclic}. 

The setting considered in the present manuscript, while akin to, is rather simpler than the one dealt with in \cite{WallIsometries}. Namely, the relatively simpler theory of centralizers for the adjoint action of $\GG(\bkk)$ on $\g(\bkk)$, in comparison with that of $\GG(\bkk)$ on itself by conjugation (compare, for example, \autoref{propo:centralizer-semisimple-element} and \cite[\S~2.14]{HumphreysConj}), allows one to retrace much of Wall's analysis in the Lie-algebra setting, without having to invoke the notion of \textit{multipliers} (see \cite[p.~11]{WallIsometries}). Furthermore, the focus on \textit{regular} adjoint classes results in a fairly ``well-behaved'' elementary divisor decomposition of the elements in the orbits under inspection. We also remark that steps (1),(2) and (3) above are in direct parallel with items (i), (ii) and (iv), respectively, of \cite[\S~2.6.(B) and (C)]{WallIsometries}, and may be derived from \textit{loc. cit.} by the following procedure. Given $x\in \matr_N(\kk)$ let $\lambda\in \kk$ be such that $x-\lambda\id$ is a non-singular matrix, and consider the dilated Cayley transform $g_x=(x-\lambda\id )^{-1}(x+\lambda\id)$. Then $x$ is similar to an element of $\mfr{g}_1$ if and only if $g_x$ is similar to an element of $G_1$, and the map $x'\mapsto (x'-\lambda\id)(x'+\lambda\id)$ is a bijection between the adjoint orbit of $x$ under $\GL_N(\kk)$ (resp. under $G_1$), and the similarity (resp. adjoint) class of $g_x$. However, applying such an argument necessitates imposing additional restrictions on the characteristic of $\kk$, and is somewhat less suitable for the purpose of enumeration of regular classes. Given these complications, and the relative simplicity of the adjoint classes in question, we have opted to present a self-contained and independent analysis of the regular adjoint classes in $\mfr{g}_1$, which is presented in Sections \ref{subsection:preliminaries}-\ref{subsection:so2n} below.

\subsubsection{Enumerative set-up}\label{subsubsection:enumerative}

\begin{defi}[Type of a polynomial]\label{defi:poly-type}
Let $f(t)\in\kk[t]$ be a polynomial of degree $N$ and $n=\lfloor\frac{N}{2}\rfloor$. For any $1\le d,e\le n$, let $\A_{d,e}(f)$ denote the number of distinct monic irreducible even polynomials of degree $2d$ which occur in $f$ with multiplicity $e$, and let $\B_{d,e}(f)$ denote the number of pairs $\set{\tau(t),\tau(-t)}$, with $\tau(t)$ monic, irreducible and coprime to $\tau(-t)$, such that $\tau$ is of degree $d$ and occurs in $f$ with multiplicity $e$. Let $r(f)$ be the maximal integer such that $t^{2r(f)}$ divides $f$. The \textbf{type} of $f$ is defined to be the triplet $\ttau(f)=(r(f),\A(f),\B(f))$, where $\A(f)$ and $\B(f)$ are the matrices $(\A_{d,e}(f))_{d,e}$ and $(\B_{d,e}(f))_{d,e}$ respectively.
\end{defi}
Recall that $\mcal{X}_n$ denotes the set of triplets $\ttau=(r,\A,\B)\in\dbN_0\times\matr_n(\dbN_0)\times \matr_{n}(\dbN_0)$, with $\A=(\A_{d,e})$ and $\B=(\B_{d,e})$ which satisfy 
\begin{equation*}
r+\sum_{d,e=1}^n de\cdot(\A_{d,e}+\B_{d,e})=n.
\end{equation*}
Note that, for $n=\lfloor\frac{N}{2}\rfloor$, it holds that $\ttau(f)\in \mcal{X}_n$ whenever $f$ is monic and satisfies $f(-t)=(-1)^{N}f(t)$.

The number of monic irreducible polynomials of degree $d$ over $\kk$ is given by evaluation at $t=q$ of the function $w_d(t)=\frac{1}{d}\sum_{r\mid d}\mu\left(\frac{d}{r}\right)t^d$,
where $\mu(\cdot)$ is the M\"obius function (see, e.g., \cite[Ch.~14]{DummitFoote}). A polynomial $f\in\kk[t]$ is said to be \textbf{even} (resp. \textbf{odd}) if it satisfies the condition $f(-t)=f(t)$ (resp. $f(-t)=-f(t)$). Note that, by assumption the $\kk$ is of odd characteristic, the only monic irreducible odd polynomial over $\kk$ is $f(t)=t$. The number of monic irreducible even polynomials of degree $d$ over $\kk$ is given by evaluation at $t=q$ of the function
\begin{equation}\label{equation:even-polys}
E_d(t)=\begin{cases}\frac{1}{d}\displaystyle\sum_{\substack{m\mid d\:,\:m\text{ odd}}
}\mu\left(m\right)(t^{d/2m}-1)&\text{if $d$ is even}\\0&\text{otherwise};\end{cases}
\end{equation}
cf.~\cite[Lemma~3.2]{Britnell-Conformal}, noting that the set of monic irreducible even polynomials of degree $d$ is in bijection with the set $N^{*1}(d,q)\subseteq\kk[t]$, defined in \textit{loc. cit.}, via the map $f(t)\mapsto \frac{(1+t)^{\deg f}}{f(-1)}f(\frac{1-t}{1+t})$. 

Put
\begin{equation}\label{equation:corrected-moebius}
 P_d(t)=\begin{cases}w_d(t)-E_d(t)&\text{if }d>1\\
 t-1&\text{if }d=1.
\end{cases}
\end{equation}
Note that, for $q$ odd, $P_d(q)$ is the number of irreducible polynomials of degree $d$ which are neither odd nor even over a field of cardinality $q$.

Given $N\in\dbN$, $n=\lfloor\frac{N}{2}\rfloor$, and $\ttau\in\mcal{X}_n$, the number of polynomials $f\in\kk[t]$ of type $\ttau(f)$ such that $f(-t)=(-1)^Nf(t)$ is given by evaluation at $t=q$ of the polynomial
\begin{equation}\label{equation:M_tau}M_{\ttau}(t)=\left(\frac{1}{2}\right)^{\sum_{d,e} \B_{d,e}}\prod_{d=1}^n{\sum_e \A_{d,e}\choose \A_{d,1},\A_{d,2},\ldots,\A_{d,n}}\cdot {E_{2d}(t)\choose \sum_e \A_{d,e}}\cdot {\sum_e \B_{d,e}\choose \B_{d,1},\B_{d,2},\ldots,\B_{d,n}}\cdot{P_{d}(t)\choose \sum_e \B_{d,e}}.\end{equation}

The combinatorial data described above is utilized in \autoref{mainthm:dirichlet-polnomials-sp-odd-orth} and \autoref{mainthm:dirichlet-polnomials-even-orth}, where it allows to enumerate the similarity classes in $\matr_N(\kk)$ which meet the Lie-algebra $\mfr{g}_1$ non-trivially in terms of the minimal polynomial of the class elements. The classification of such similarity classes and their decomposition into $\Ad(G_1)$ is described \autoref{theo:orbits-sp2nso2n+1} and \autoref{theo:orbits-so2n} below.

Once Theorems \ref{theo:orbits-sp2nso2n+1} and \ref{theo:orbits-so2n} are proved, the proof of \autoref{mainthm:dirichlet-polnomials-sp-odd-orth} and of \autoref{mainthm:dirichlet-polnomials-even-orth} may be completed by direct computation.

\subsubsection{Statement of results- symplectic and odd-dimensional special orthogonal groups}


\begin{theo}\label{theo:orbits-sp2nso2n+1}
Assume $\Char(\kk)\ne 2$. Let $V=\kk^N$ and let $B$ be a non-degenerate bilinear form which is anti-symmetric if $N=2n$ is even, and symmetric if $N=2n+1$. Let $\GG\in\set{\Sp_{2n},\SO_{2n+1}}$  be the algebraic group of isometries of $V$ with respect to $B$ and put $G_1=\GG(\kk)$ and $\mfr{g}_1=\g(\kk)$ where $\g=\Lie(\GG)$.

Let $x\in\matr_N(\kk)$ have minimal polynomial $m_x\in\kk[t]$.
\begin{enumerate}
\item The element $x$ is similar to a regular element of $\mfr{g}_1$ if and only if $m_x$ has degree $N$ and satisfies $m_x(-t)=(-1)^Nm_x(t)$. 
\end{enumerate}
Furthermore, assume $x\in\mfr{g}_1$ is a regular element and let $\Omega=\Ad(G_1)x$ denote its orbit under~$G_1$. 
\begin{enumerate}\stepcounter{enumi}
\item If $N$ is even and $m_x(0)=0$, then the intersection $\Ad(\GL_N(\kk))x\cap \mfr{g}_1$ is the union of two distinct $\Ad(G_1)$-orbits. Otherwise, $\Ad(\GL_N(\kk))x\cap \mfr{g}_1=\Omega$.
\item Let $\ttau=\ttau(m_x)=\left(r(m_x),\A(m_x),\B(m_x)\right)
$ as in \autoref{defi:poly-type}. Then 
\[\abs{\Omega}=q^{2n^2}\cdot\left(\frac{1}{2}\right)^\nu\frac{\prod_{i=1}^n(1-q^{-2i})}{\prod_{1\le d,e\le n}(1+q^{-d})^{\A_{d,e}(m_x)}\cdot (1-q^{-d})^{\B_{d,e}(m_x)}},\]
where $\nu=1$ if $N=2n$ is even and $m_x(0)=0$, and $\nu=0$ otherwise.
\end{enumerate}
\end{theo}

The proofs of Assertions (1), (2) and (3) of the theorem are carried out in sections \ref{subsubsection:simclasses-sp2nso2n+1},\:\ref{subsubsection:simtoconj-sp2nso2n+1} and \ref{subsubsection:cetralizers-sp2nso2n+1} respectively.

\subsubsection{Statement of results- even-dimensional special orthogonal groups}
\begin{theo}\label{theo:orbits-so2n}
Assume $\abs{\kk}>3$ and $\Char(\kk)\ne 2$. Let $N=2n$ with $n\ge2$. Let $V=\kk^N$ and let $B^+$ and $B^-$ be non-degenerate symmetric forms on $V$ of Witt index $n$ and $n-1$, respectively. Given $\epsilon\in\set{\pm 1}$, let $\GG^\epsilon=\SO_{2n}^\epsilon$ be the $\kk$-algebraic group of isometries of $V$ with respect to $B^\epsilon$ and put $G^\epsilon_1=\GG_1^\epsilon(\kk)$ and let $\mfr{g}^\epsilon_1=\g^\epsilon(\kk)$, where $\g^\epsilon=\Lie(\GG)$.

Let $x\in\matr_N(\kk)$ have minimal polynomial $m_x(t)$.

\begin{enumerate}
\item If $m_x(0)=0$ (i.e. $x$ is a singular matrix) then the following are equivalent.
\begin{enumerate}
\item The polynomial $m_x$ has degree $N-1$ and satisfies $m_x(-t)=-m_x(t)$.
\item The element $x$ is similar to a regular element of $\mfr{g}_1^+$.
\item The element $x$ is similar to a regular element of $\mfr{g}_1^-$.
\end{enumerate}
Otherwise, if $m_x(0)\ne 0$, let $\epsilon=\epsilon(x)=(-1)^{\sum_{e}e\A_{d,e}(m_x)}$ where $\A=(\A_{d,e}(m_x))$ is as in \autoref{defi:poly-type}. Then $x$ is similar to a regular element of $\mfr{g}_1^\epsilon$ if and only if $m_x$ has degree $N$ and satisfies $m_x(-t)=m_x(t)$. Moreover, in this case $x$ is \textit{not} similar to an element of $\mfr{g}_1^{-\epsilon}$. 
\end{enumerate}
Furthermore, assume $x\in\mfr{g}^\epsilon_1$ is a regular element and let $\Omega^\epsilon=\Ad(G_1^\epsilon)x$ denote its orbit under $G_1^\epsilon$, for $\epsilon\in\set{\pm 1}$ fixed.
\begin{enumerate}
\setcounter{enumi}{1}
\item In the case where $m_x(0)=0$, the intersection $\Ad(\GL_N(\kk))x\cap\mfr{g}_1^\epsilon$ is the disjoint union of two distinct $\Ad(G_1^\epsilon)$-orbits. Otherwise, $\Ad(\GL_N(\kk))x\cap\mfr{g}_1^\epsilon=\Omega^\epsilon$.
\item 
\begin{enumerate}
\item Assume $m_x(0)=0$ and let $\ttau=\ttau(t\cdot m_x)$. Then
\[\abs{\Omega^\epsilon}=q^{2n^2}\cdot\frac{1}{2}\cdot\frac{(1+\epsilon q^{-n})\prod_{i=1}^{n-1}(1-q^{-2i})}{\prod_{1\le d,e\le n}(1+q^{-d})^{\A_{d,e}(m_x)}\cdot(1-q^{-d})^{\B_{d,e}(m_x)}}.\]
\item Otherwise, let $\ttau=\ttau(m_x)$. Then
 \[\abs{\Omega^\epsilon}=q^{2n^2}\cdot \frac{(1+\epsilon q^{-n})\prod_{i=1}^{n-1}(1-q^{-2i})}{ \prod_{1\le d,e\le n}(1+q^{-d})^{\A_{d,e}(m_x)}\cdot(1-q^{-d})^{\B_{d,e}(m_x)}}.\]
\end{enumerate}
\end{enumerate}
\end{theo}

 The proofs of Assertions (1),(2) and (3) of the theorem appear in sections \ref{subsubsection:simclasses-so2n},\:\ref{subsubsection:simtoconj-so2n} and \ref{subsubsection:cetralizers-so2n}. The exclusion of the specific case of $\kk=\dbF_3$ is done for technical reasons, and may possibly be undone by replacement of the argument in \autoref{lem:difference-of-squares} below.

\subsection{Preliminaries to the proofs \autoref{theo:orbits-sp2nso2n+1} and \autoref{theo:orbits-so2n}}
\label{subsection:preliminaries}
\subsubsection{Regularity for non-singular elements}\label{subsection:centralizers-res-field}

\begin{lem}
\label{corol:reg-g1-reg-glN-nonsing}Let $x\in \g(\bkk)\subseteq\mfr{gl}_N(\bkk)$ be non-singular. Then $x$ is regular in $\g(\bkk)$ if and only if $x$ is a regular element of $\mfr{gl}_N(\bkk)$.
\end{lem}
\begin{proof}
Let $W=(\bkk)^N$ , so that $\g(\bkk)$ is given as the Lie-algebra of anti-symmetric operators with respect to a non-degenerate bilinear form $B=B_{\bkk}$ on $W$ (see \autoref{subsection:group-notation}). Note that the existence of non-singular elements in $\g(\bkk)$ implies that $N=2n$ is even. Indeed, $x\in\g(\bkk)$ if and only if $x^\inv=-x$ (notation of \autoref{subsubsection:adjoint-operators}), and $\det(x)=\det(x^\inv)=(-1)^N\det(x)$ is possible if and only if $N$ is even, since $\Char(\bkk)\ne 2$. 

Let $x=s+h$ be the Jordan decomposition of $x$, with $s,h\in \g(\bkk)$, $s$ semisimple, $h$ nilpotent and $[s,h]=0$. Let $\lambda_1,\ldots,\lambda_t\in\bkk$ be non-zero and such that $\set{\pm\lambda_1,\ldots,\pm\lambda_t}$ is the set of all eigenvalues of $s$ with $\lambda_i\ne \pm\lambda_j$ whenever $i\ne j$. As in \autoref{propo:centralizer-semisimple-element}, the space $W$ decomposes as a direct sum $W=\bigoplus_{i=1}^t (W_{\lambda_i}\oplus W_{-\lambda_i})$, where, for any $i=1,\ldots,t$, the subspace $W_{[\lambda_i]}=W_{\lambda_i}\oplus W_{-\lambda_i}$ is non-degenerate, and its subspaces $W_{\lambda_i}$ and $W_{-\lambda_i}$ are maximal isotropic. Comparing centralizer dimension, and invoking \cite[\S~3.5\:,Proposition~1]{SteinbergConjugacy}, we have that $x$ is regular if and only if the restriction of $x$ to each of the subspaces $W_{[\lambda_i]}$ ($i=1,\ldots,t$) is regular in $\gl_N(W_{[\lambda_i]})$. Likewise, $x$ is regular in $\g(\bkk)$ if and only if the restriction of $x$ to each subspace $W_{[\lambda_i]}$ is regular within the Lie-algebra of anti-symmetric operators with respect to the restriction of $B_\bkk$ to $W_{[\lambda_i]}$. Thus, it is sufficient to prove the lemma in the case where $s$ has precisely two eigenvalues $\lambda,-\lambda$. 

Representing $s$ in a suitable eigenbasis, it be identified with the block-diagonal matrix $\diag(\lambda\id_n,-\lambda\id_n)$. Under this identification, the centralizer of $s$ in $\gl_N(\bkk)$ is identified with the subgroup of block diagonal matrices consisting of two $n\times n$ blocks. Moreover, the involution $\inv$ maps an element $\diag(y_1,y_2)\in\CC_{\gl_N(\bkk)}(s)$, with $y_1,y_2\in\gl_n(\bkk)$ to the matrix $\diag(y_2^t,y_1^t)$. In particular, it follows that $h\in \CC_{\gl_N}(\bkk)\cap \g(\bkk)$ is of the form $h=\diag(h_1,-h_1^t)$, where $h_1\in\gl_n(\bkk)$ is nilpotent. Arguing as in \cite[III,\:\S~1]{SteinbergSpringer}, we have that
\[\CC_{\GL_N}(x)=\CC_{\CC_{\GL_N}(s)}(h)\simeq\CC_{\GL_n}(h_1)\times \CC_{\GL_n}(-h_1^t)\simeq\CC_{\GL_n}(h_1)\times \CC_{\GL_n}(h_1)\]
where the final isomorphism utilizes the isomorphism $y\mapsto (y^t)^{-1}:\CC_{\GL_n}(-h_1^t)\to \CC_{\GL_n}(h_1)$.

Finally, since the group $\GG$ is embedded in $\GL_N$ as the group of unitary elements with respect to $\inv$, we have $\diag(y_1,y_2)\in\CC_{\GL_N(\bkk)}(s)\cap \GG(\bkk)$ if and only if $y_2=(y_1^{t})^{-1}$, and hence the map $y\mapsto\diag(y,(y^t)^{-1})$ is an isomorphism of $\CC_{\GL_n}(h_1)$ onto $\CC_{\CC_\GG(s)}(h)$ and hence 
\[\CC_\GG(x)=\CC_{\CC_\GG(s)}(h)\simeq\CC_{\GL_n}(h_1).\]
Thus 
\[\dim\CC_{\GL_N}(x)=2\dim\CC_{\GL_n}(h_1)=2\dim\CC_{\GG}(x),\]
and the lemma follows.
\end{proof}
\begin{rem} The assumption that $x$ is non-singular in \autoref{corol:reg-g1-reg-glN-nonsing} is crucial, as the proof  relies heavily on the fact that the centralizer of a non-singular semisimple element of $\ggamma_1(\bkk)$ in $\GGamma_1(\bkk)$ is a direct product of groups of the form $\GL_{m_j}(\bkk)$ (see \autoref{propo:centralizer-semisimple-element}). The same argumentation would not apply in the case where $x$ is singular, and in fact fails in certain cases; see \autoref{lem:nilpotent-reg-glN-not-reg-soN} below.
\end{rem}


\subsubsection{From similarity classes to adjoint orbits}\label{subsubsection:wall}
In this section develop some the tools required in order to analyze the decomposition of the set $\Pi_x=\Ad(\GL_N(\kk))x\cap\mfr{g}_1$, for $x\in \mfr{g}_1$ regular, in to $\Ad(G_1)$-orbits. The results appearing below can also be derived from \cite[\S~2.6]{WallIsometries}. However, as the case of regular elements of the Lie-algebra $\mfr{g}_1$ allows for a much more transparent argument, we present it here for completeness.


Let $\sym(\inv;x)$ be the set of elements $Q\in\CC_{\GL_N(\kk)}(x)$ such that $Q^\inv=Q$ and define an equivalence relation on $\sym(\inv;x)$ by
\begin{equation}\label{equation:sim-defi}
Q_1\sim Q_2\quad\text{if there exists }a\in\CC_{\GL_N(\kk)}(x)\text{ such that }Q_1=a^\inv Q_2a.
\end{equation}

Let $\Theta_x$ to be the set of equivalence classes of $\sim$ in $\sym(\inv;x)$. In the case where $\CC_{\GL_N(\kk)}(x)$ is abelian (e.g., when $x$ is a regular element of $\mfr{gl}_N(\kk)$), the set $\sym(\inv;x)$ is a subgroup and the set $\Theta_x$ is simply its quotient by the image of restriction of $w\mapsto w^\inv w$ to $\CC_{\GL_N(\kk)}(x)$.

\begin{propo}\label{propo:orbit-decomposition} Let $x\in\mfr{g}_1$ and let $\Pi_x$ denote the intersection $\Ad(\GL_N(\kk))x\cap \mfr{g}_1$. There exists a map $\Lambda:\Pi_x\to \Theta_x$ such that $y_1,y_2\in \Pi_x$ are $\Ad(G_1)$-conjugate if and only if $\Lambda(y_1)=\Lambda(y_2)$. 

\end{propo}

\begin{proof}
\begin{list}{\arabic{list}.}{\usecounter{list}\setlength{\leftmargin}{0pt}\setlength{\itemsep}{2pt}\setlength{\labelwidth}{-5pt}}
\item \textit{Construction of $\Lambda$.} Let $y\in \Pi_x$ and let $w\in \GL_N(\kk)$ be such that $y=wxw^{-1}$. Put $Q=w^\inv w$. Note that, as $x,y\in \mfr{g}_1$, by applying the anti-involution $\inv$ to the equation $y=wxw^{-1}$, we deduce that $(w^\inv )^{-1}xw^\inv=y$ as well and consequently, that $Q=w^\inv w$ commutes with $x$. Since $Q^\inv =Q$, we get that  $Q\in\sym(\inv;x)$.

Define $\Lambda(y)$ to be the equivalence class of $Q$ in $\Theta_x$. To show that $\Lambda$ is well-defined, let $w'\in \GL_N(\kk)$ be another element such that $y=w'xw'^{-1}$ and $Q'=w'^\inv w'$. Put $a=w^{-1}w'$. Then $a$ commutes with $x$, and 
\[a^\inv Q a=w'^\inv (w^\inv)^{-1}Qw^{-1}w'=w'^\inv w'=Q',\]
hence $Q\sim Q'$.
\item \textit{Proof that $y_1,y_2\in \Pi_x$ are $\Ad(G_1)$-conjugate if $\Lambda(y_1)=\Lambda(y_2)$.} Let $w_1,w_2\in \GL_N(\kk)$ be such that $y_i=w_ixw_i^{-1}$, and let $Q_i=w_i^\inv w_i$  ($i=1,2$). Then, by assumption, there exists $a\in \CC_{\GL_N(\kk)}(x)$ such that $Q_2=a^\inv Q_1a$. Put $z=w_1aw_2^{-1}$. Note that $zy_2z^{-1}=y_1$. We claim that $z\in G_1$. This holds since for any $u,v\in V$ 
\begin{align*}
B(zu,zv)&=B(w_1aw_2^{-1}u,w_1aw_2^{-1}v)=B(a^\inv(w_1^\inv w_1)aw_2^{-1}u,w_2^{-1}v)\\
&=B(a^\inv Q_1 aw_2^{-1}u,w_2^{-1}v)=B(Q_2 w_2^{-1} u,w_2^{-1}v)&\text{(since $Q_2=a^\inv Q_1 a$)}\\
&=B(w_2^\inv u,w_2^{-1}v)=B(u,v).
\end{align*}
\item \textit{Proof that $y_1,y_2\in \Pi_x$ are $\Ad(G_1)$-conjugate only if $\Lambda(y_1)=\Lambda(y_2)$.} Assume now that $z\in G_1$ is such that $y_1=zy_2z^{-1}$, and let $w_1,w_2\in \GL_N(\kk)$ be such that $y_i=w_ixw_i^{-1}$ ($i=1,2$). Then $w_1$ and $zw_2$ both conjugate $x$ to $y_1$, and hence, by the unambiguity of the definition of $\Lambda$ and fact that $z\in G_1$, we have that
\[\Lambda(y_1)=[w_1^\inv w_1]=[w_2^\inv( z^\inv z) w_2]=[w_2^\inv w_2]=\Lambda(y_2).\]
\end{list}
\end{proof}

A crucial property of the set $\Theta_x$ in the case $x$ is regular, which makes the analysis of adjoint orbits feasible, is that it may be realized within the quotient of an \'etale algebra over $\kk$ by the image of the algebra under an involution. As a consequence, the set $\Pi_x$ decomposes into $\abs{\im\Lambda}$ many $\Ad(G_1)$-orbits, a quantity which does not exceed the value four in the regular case.

Let us state another general lemma, which will be required in the description of $\Theta_x$.
\begin{lem}\label{lem:CmodN-suff-cond} Let $\mcal{C}\subseteq\matr_N(\kk)$ be the ring of matrices commuting with a matrix $x$, with $x^\inv=-x$ (or $x^\inv =x$), and let $\mcal{N}\triangleleft\mcal{C}$ be a nilpotent ideal, invariant under $\inv$. The following are equivalent, for any $Q_1,Q_2\in\sym(\inv;x)$.
\begin{enumerate}
\item There exists $a\in \mcal{C}$ such that $a^\inv Q_1 a=Q_2$;
\item There exists $a\in\mcal{C}$ such that $a^\inv Q_1 a\equiv Q_2\pmod{\mcal{N}}$. 
\end{enumerate}
\end{lem}
\begin{proof}
The argument of \cite[Theorem~2.2.1]{WallIsometries} applies to the case where $\mcal{N}$ is any nilpotent ideal which is invariant under $\inv$, provided that the required trace condition holds. In the present case the condition holds since $\Char(\kk)\ne 2$. 
\end{proof}
\subsubsection{Similarity classes via bilinear forms}\label{subsection:bilinear-forms-similarity}
We recall a basic lemma which would allow us to determine when an element of $\mfr{gl}_N(\kk)$ is similar to an element of $\mfr{g}_1$. Here and in the sequel, given a non-degenerate bilinear form $C$ on a finite dimensional vector space $V$ over $\kk$, we call an operator $x\in \End (V)$ $C$-\textbf{anti-symmetric}, if $C(xu,v)+C(u,xv)=0$ holds for all $u,v\in V$.
\begin{lem}\label{lem:conj-by-forms} Let $C_1,C_2$ be two non-degenerate bilinear forms on a vector space $V=\kk^N$, and assume there exists $g\in \End(V)$ and $\delta\in\kk$ such that $C_1(gu,gv)=\delta C_2(u,v)$ for all $u,v\in V$. Let $x\in\mfr{gl}_N(\kk)$ be $C_2$-anti-symmetric. Then $gxg^{-1}$ is $C_1$-anti-symmetric.
\end{lem}
The proof of \autoref{lem:conj-by-forms} is by direct computation, and is omitted. 
\subsection{Symplectic and odd-dimensional special orthogonal groups}\label{subsection:sp2n-so2n+1}

Throughout \autoref{subsection:sp2n-so2n+1}, we assume $\GG=\Sp_{2n}$ or $\GG=\SO_{2n+1}$. The following well-known fact is very useful in the classification of regular adjoint classes in the Lie-algebra $\mfr{g}_1$.

\begin{lem}\label{lem:eseential-uniqueness} Let $\epsilon =-1$ and $N=2n$ in the symplectic case, or $\epsilon=1$ and $N=2n+1$ in the special orthogonal case. Let $C_1,C_2$ be two non-degenerate forms on $V=\kk^N$ such that $C_i(u,v)=\epsilon C_i(v,u)$ for all $u,v\in V$ and $i=1,2$. There exists $\delta\in\kk$ and $g\in \End(V)$ such that $C_1(gu,gv)=\delta C_2(u,v)$ for all $u,v\in V$. Additionally, if $\epsilon=-1$ then $\delta$ can be taken to be $1$.
\end{lem}
\begin{proof}
See, e.g., \cite[\S 3.4.4]{WilsonFinite} in the symplectic case and \cite[\S~3.4.6 and \S~3.7]{WilsonFinite} in the special orthogonal case.
\end{proof}

\subsubsection{Similarity classes of regular elements}\label{subsubsection:simclasses-sp2nso2n+1}
The following lemma gives a criterion for a regular matrix to be similar to an element of $\mfr{g}_1$.

\begin{lem}\label{lem:conj-to-g1}
Let $x\in\gl_N(\kk)$ with minimal polynomial $m_x(t)\in\kk[t]$.
\begin{enumerate}
\item If $x$ is similar to an element of $\mfr{g}_1$ then $m_x(t)$ satisfies $m_x(-t)=(-1)^{\deg m_x}m_x(t)$. 
\item If $x$ is a regular element of $\mfr{gl}_N(\kk)$ (and hence $\deg m_x=N$) such that $m_x(t)=(-1)^Nm_x(t)$, then $x$ is similar to an element of $\mfr{g}_1$.
\end{enumerate}
\end{lem}
\begin{proof}
For the first assertion, we may assume $x\in\mfr{g}_1$. Note that for any $r\in\dbN$ we have that $B(x^ru,v)=B(u,(-1)^r x^rv)$ for all $u,v\in V=\kk^N$. Invoking the non-degeneracy of $B$, we deduce that $(-1)^{\deg m_x}m_x(-t)$ is a monic polynomial of degree $\deg m_x$ which vanishes at $x$, and hence equal to $m_x(t)$.

By \autoref{lem:conj-by-forms}, to prove the second assertion it would suffice to construct a non-degenerate bilinear form $C$ on $V$ such that $B$ and $C$ satisfy the hypothesis of \autoref{lem:conj-by-forms}. In view of \autoref{lem:eseential-uniqueness}, in the present case it suffices to construct \textit{some }non-degenerate bilinear form $C$ on $V$ such that $C(u,v)=\epsilon C(v,u)$, where $\epsilon=(-1)^N$, and such that $x$ is $C$-anti-symmetric.

\newcommand{\proj}{\mathrm{Prj}}
By \cite[Ch.~III,\:3.5, Proposition~2]{SteinbergConjugacy}, the assumption that $x$ is a regular matrix is equivalent to $V$ being a cyclic module over the ring $\kk[x]$ (which, in turn, is equivalent to $\deg m_x=N$). In particular, there exists $v_0\in V$ such that $(v_0,xv_0,\ldots,x^{N-1}v_0)$ is a $\kk$-basis for $V$. Let $\proj_{N-1}:V\to \kk$ denote the projection onto $\kk\cdot x^{N-1}v_0$. Given $u_1,u_2\in V$ let $p_1,p_2\in\kk[t]$ be polynomials such that $u_i=p_i(x)v_0$ and define \begin{equation}
\label{equation:C-definition} C(u_1,u_2)=\proj_{N-1}\left(p_1(x)p_2(-x)v_0\right).
\end{equation}
The fact that $C$ is well-defined, bilinear and satisfies $C(u,v)=\epsilon C(c,u)$ follows by direct computation. Let us verify that $C$ is non-degenerate.

Let $u\in V$ be non-zero, and let $p(t)$ be such that $p(x)v_0=u$. By unambiguity of the definition of $C$, we may assume that $\deg p(t)<N$. Let $v=x^{N-1-\deg p} v_0\in V$. Then 
\[C(u,v)=\proj_{N-1}((-1)^{N-1-\deg p}x^{N-1-\deg p}p(x)v_0)\]
is non-zero, since $t^{N-1-\deg p} p(t)$ is a polynomial of degree $N-1$. 

Finally, for $u_i=p_i(x)v_0$ as above, we have that
\[C(xu,v)+X(u,xv)=\proj_{N-1}(xp_1(x)p_2(x)v_0)+\proj_{N-1}(p_1(x)\cdot(-xp_2(-x))v_0)=\proj_{N-1}(0)=0,\]
and hence $x$ is $C$-anti-symmetric.
\end{proof}

Note that \autoref{lem:conj-to-g1} gives a criterion for a regular element of $\mfr{gl}_N(\kk)$ to be similar to an element of $\mfr{g}_1$, but a-priori, not necessarily to \textit{a regular} element of $\mfr{g}_1$. We will shortly see that it is indeed the case that the similarity class of such $x$ meets $\mfr{g}_1$ at a regular orbit. Before proving this, let us consider an important example.

\begin{exam}[Regular nilpotent elements]\label{exam:reg-nil-elt-so2n+1-sp2n}
Let $x\in\mfr{gl}_N(\kk)$ be a regular nilpotent element, i.e. $m_x(t)=t^N$. Picking a generator $v_0$ for $V$ over $\kk[x]$ and putting $\mcal{E}=(v_0,xv_0,\ldots,x^{N-1}v_0)$, the element $x$ is represented in the basis $\mcal{E}$ by the matrix $\nil$, given by an $N\times N$ nilpotent Jordan block. The bilinear form $C$ of \autoref{lem:conj-to-g1} is represented in this basis by the matrix \begin{equation}\label{eqution:C-nil-elt}
\mbf{c}=\mat{&&&1\\&&-1\\
&\reflectbox{$\ddots$}\\(-1)^{N-1}}.
\end{equation}
\end{exam}
To show that $\nil$ is similar to a \textit{regular} element of $\mfr{g}_1$, by \cite[3.5,\:Proposition~1]{SteinbergConjugacy} and \cite[\S~1.10,\:Proposition]{HumphreysConj}, it suffices to show that the centralizer of $\nil$ within the Lie-algebra $\mfr{h}\subseteq \matr_N(\bkk)$, of matrices $y$ satisfying the condition $y^t\mbf{c}+\mbf{c} y$ (i.e. the Lie-algebra of the linear algebraic $\bkk$-group of isometries of $C(\cdot,\cdot)$), is of dimension $n$ over $\bkk$. By direct computation, one shows that
\[\CC_{\mfr{h}}(\nil)=\set{\mat{a_1&a_2&\ldots&a_{N}\\&\ddots&\ddots&\vdots\\&&a_1&a_2\\&&&a_1}\in \matr_N(\bkk)\mid 2a_{2i-1}=0\text{ for all }i=1,\ldots,\lceil{N}/{2}\rceil}.\]
Recalling that $\Char(\bkk)\ne 2$, it follows that $a_{2i+1}=0$ for all $i=0,\ldots,\lceil{N/2}\rceil$ and hence $\dim_{\bkk}\CC_{\mfr{h}}(\nil)=\lfloor N/2\rfloor= n.$



\begin{propo}\label{propo:reg-equiavlent} Let $x\in\mfr{g}_1$. Then $x$ is a regular element of $\mfr{g}_1$ if and only if $x$ is regular in $\mfr{gl}_N(\kk)$.
\end{propo}
\begin{proof}
By \cite[\S~3.5,\:Proposition~1]{SteinbergConjugacy}, we need to show $\dim\CC_{\GGamma_1}(x)=n$ if and only if $\dim\CC_{\GL_N\times\bkk}(x) =N$. Let $x=s+h$ be the Jordan decomposition of $x$ over $\bkk$, with $s,h\in \g(\bkk)$, $s$ semisimple, $h$ nilpotent, and $[s,h]=0$. As seen in the proof of  \autoref{propo:centralizer-semisimple-element}, the space $W=(\bkk)^N$ decomposes as an orthogonal direct sum $W_1\oplus W_0$ with respect to the ambient bilinear form $B_{\bkk}$, where $W_0=\ker(s)$ and $s\mid_{W_1}$ is non-singular. Let $\Sigma\subseteq \GG(\bkk)$ be the subgroup of elements acting trivially on $W_0$ and preserving $W_1$, and let $\Delta$ be as in \autoref{propo:centralizer-semisimple-element}. Then 
\begin{align*}\CC_{\GGamma_1}(x)&=\CC_{\Sigma}(x)\times \CC_{\set{1_{W_1}}\times\Delta}(x)\intertext{and} \CC_{\GL_N(\bkk)}(x)&=\CC_{\GL(W_1)\times\set{1_{W_0}}}(x)\times \CC_{\set{1_{W_1}}\times \GL(W_0)}(x)\end{align*}
and therefore the proof reduces to the cases where $x$ is non-singular and where $x$ is a nilpotent element acting on $W_0$. The first case follows from \autoref{corol:reg-g1-reg-glN-nonsing}, whereas the second case follows from \autoref{exam:reg-nil-elt-so2n+1-sp2n} and from the uniqueness of a regular nilpotent orbit over algebraically closed fields \cite[III, Theorem~1.8]{SteinbergConjugacy}.
\end{proof}
\begin{proof}[Proof of \autoref{theo:orbits-sp2nso2n+1}.(1)]
\autoref{propo:reg-equiavlent} implies that any element $x\in \matr_N(\kk)$ which is similar to a regular element of $\mfr{g}_1$ is regular as an element of $\mfr{gl}_N(\kk)$. It follows easily that $\deg m_x=N$ and $m_x(-t)=(-1)^N m_x(t)$(see \autoref{lem:conj-to-g1}.(1)). The converse implication is given by \autoref{lem:conj-to-g1}.(2).
\end{proof}

\subsubsection{From similarity classes to adjoint orbits}\label{subsubsection:simtoconj-sp2nso2n+1}
In this section is to we analyze decomposition of the set $\Ad(\GL_N(\kk))x\cap\mfr{g}_1$, for $x\in\mfr{g}_1$ regular, into $\Ad(G_1)$-orbits, and prove \autoref{theo:orbits-sp2nso2n+1}.(2). 
\begin{nota}\label{nota:even-polys-splitting-rings} Given a polynomial $f(t)\in \kk[t]$ we write $\spring{f}$ for the quotient ring $\kk[t]/(f)$. For example, if $f$ is an irreducible polynomial over $\kk$ then $\spring{f}$ stands for the splitting field of $f$. We write $\GL_1(\spring{f})$ for the group of units of $\spring{f}$.

Assuming further that $f(t)=\pm f(-t)$, let $\sigma_f$ denote the $\kk$-involution of $\spring{f}$, induced from the $\kk[t]$-involution $t\mapsto -t$, and let $\UU_1(\spring{f})$ be the group of elements $\xi\in \spring{f}$ such that $\sigma_f(\xi)\cdot \xi=~1$. 
\end{nota}

\begin{propo} \label{propo:decomposition-of-sim-class1} Let $x\in\mfr{g}_1$ be a regular element, and put $\Pi_x=\Ad(\GL_N(\kk))x\cap \mfr{g}_1$. If $x$ is singular and $N$ is even, then the intersection $\Pi_x$ is the disjoint union of two distinct $\Ad(G_1)$-orbits.  Otherwise, $\Pi_x=\Ad(G_1)x$.
\end{propo}
\begin{proof}The notation of \autoref{propo:orbit-decomposition} is used freely throughout the proof. We proceed in the following steps.
\begin{enumerate}
\item Computation of the cardinality of $\Theta_x$. Namely, we show that $\abs{\Theta_x}=2$ if $x$ is singular and equals $1$ otherwise.
\item Description of the image of the map $\Lambda$ in $\Theta_x$.
\end{enumerate}

By Lemma~\ref{lem:conj-to-g1}, the minimal polynomial $m_x$ of $x$ is of degree $N$ and satisfies $m_x(-t)=(-1)^{N} m_x(t)$. Thus, it can be expressed uniquely as the product of pairwise coprime factors
\begin{equation}
\label{equation:m_x-decompose}
m_x(t)=t^{d_1}\cdot\prod_{i=1}^{d_2}\varphi_i(t)^{l_i}\cdot\prod_{i=1}^{d_3}\theta_i(t)^{r_i},
\end{equation}
where the polynomials $\varphi_1,\ldots,\varphi_{d_2}$ are irreducible, monic and even, and $\theta_1,\ldots,\theta_{d_3}$ are of the form $\theta_i(t)=\tau_i(t)\cdot\tau_i(-t)$ with $\tau_i(t)$ monic, irreducible and coprime to $\tau(-t)$. The centralizer $\mcal{C}=\CC_{\matr_N(\kk)}(x)$ is isomorphic to the ring $\spring{m_x}$ and the restriction of the involution $\inv$ to $\mcal{C}$ is transferred via this isomorphism to the map $\sigma_{m_x}$, defined in \autoref{nota:even-polys-splitting-rings}. By the Chinese remainder theorem, we get
\begin{equation}\label{equation:C-decomposition}
\mcal{C}\simeq\spring{t^{d_1}}\times \prod_{i=1}^{d_2}\spring{\varphi_i(t)^{m_1}}\times\prod_{i=1}^{d_3}\spring{\theta_i(t)^{r_i}}.
\end{equation}
Furthermore, the restriction of the involution $\sigma_{m_x}$ to each of the factors $\spring{f}$, for $f\in\set{ t^{d_1},\varphi_i^{l_i},\theta_j^{r_j}}$ coincides with the respective involution $\sigma_f$, induced from $t\mapsto -t$. A short computation shows that the nilpotent radical of $\mcal{C}$ is isomorphic to the direct product of the nilpotent radicals of all factors on the right hand side of \eqref{equation:C-decomposition}, and that the quotient $\mcal{C}/\mcal{N}$ is isomorphic to the \'etale algebra
\begin{equation}\label{equation:C-mod-nil-radical}\mcal{K}=\kk^r\times \prod_{i=1}^{d_2}\spring{\varphi_i}\times\prod_{i=1}^{d_3}\spring{\theta_i},
\end{equation}
where $r=1$ if $d_1>0$ (i.e. if $x$ is singular) and equals $0$ otherwise\footnote{Here it is understood that the ring $\kk^0$ is the trivial algebra $\set{0}$.}. Let $\dagger$ denote the involution induced on the $\kk$-algebra $\mcal{K}$ in \eqref{equation:C-mod-nil-radical} from the restriction of $\inv$ to $\mcal{C}$. From the observation regarding the action of $\inv$ on $\mcal{C}$ above, we deduce the following properties of the involution $\dagger$ on $\mcal{K}$.
\begin{enumerate}\renewcommand{\theenumi}{D\arabic{enumi}}
\item\label{dag-property1} The involution $\dagger$ preserves the factor $\kk^r$ and acts trivially on it.
\item\label{dag-property2} The involution $\dagger$ preserves the factors $\spring{\varphi_i}$ and coincides with the non-trivial field involution $\sigma_{\varphi_i}$.
\item\label{dag-property3} The involution $\dagger$ preserves the factors $\spring{\theta_i}\simeq\spring{\tau_i(t)}\times\spring{\tau_i(-t)}$ and maps a pair $(\xi,\nu)\in \spring{\tau_i(t)}\times\spring{\tau_i(-t)}$ to the pair $(\iota^{-1}(\nu),\iota(\xi))$, where $\iota:\spring{\tau_i(t)}\to\spring{\tau_i(-t)}$ is the isomorphism induced from $t\mapsto -t$.
\end{enumerate} 

Let $\sym(\dagger)$ be subgroup of $\mcal{K}^\times$ of elements fixed by $\dagger$. Note that, as $\mcal{K}\simeq \mcal{C}/\mcal{N}$ is a commutative ring, by \autoref{lem:CmodN-suff-cond}, the set $\Theta_x$ can be identified with the quotient of $\sym(\dagger)$ by the image of the map $z\mapsto z^\dagger z:\mcal{K}\to\sym(\dagger)$. 

By \eqref{dag-property2} and the theory of finite fields, the restriction of the map $z\mapsto z^\dagger z$ to the factors $\spring{\varphi_i}$ coincides with the field norm onto the subfield of element fixed by $\dagger$, and is surjective onto this subfield. Furthermore, by \eqref{dag-property3}, it is evident that an element $(\xi,\nu)\in\spring{\tau_i(t)}\times\spring{\tau_i(-t)}$ is fixed by $\dagger$ if an only if $\nu=\iota(\xi)$, in which case $(\xi,\nu)=(\xi,1)^\dagger\cdot (\xi,1)$. Lastly, by \eqref{dag-property1} it holds that the image of the restriction of $z\mapsto z^\dagger z$ to the multiplicative group of $\kk^r$ is either trivial, if $r=0$, or the group of squares in $\kk^\times$, otherwise. It follows from this that the set $\Theta_x$ is either in bijection with the quotient $\left(\kk^\times/(\kk^\times)^2\right)$, and hence of cardinality $2$, if $x$ singular, or otherwise trivial. This completes the first step of the proof.

For the second step, we divide the analysis according to the parity of $N$, in order to describe the image of $\Lambda$.

\begin{list}{}{\setlength{\leftmargin}{0pt}\setlength{\itemsep}{2pt}}
\item\textit{$N$ even.} In this case we show that $\Lambda$ is surjective. To do so, let $Q\in \sym(\inv; x)$. Note that, by assumption, $Q^\inv= Q$ and $Q\in \GL_N(\kk)$, and hence the form $(u,v)\mapsto B(u,Qv)$ is alternating and non-degenerate. By \autoref{lem:eseential-uniqueness}, there exists $w\in \GL_N(\kk)$ such that $Q=w^\inv w$. To show that $Q\in\im\Lambda$ we only need to verify that $y=wxw^{-1}\in \mfr{g}_1$. This holds, as 
\[y^\inv=(w^\inv)^{-1 }x^\inv w^\inv =-(w^{\inv})^{-1}(Q x Q^{-1})w^\inv=-wxw^{-1}=-y,\]
since $Q$ is assumed to commute with $x$.
\item\textit{$N$ odd.} Note that in this case, all elements of $\mfr{g}_1$ are non-singular and hence $\abs{\Theta_x}=2$ for all $x\in \mfr{g}_1$. In this case we prove that the map $\Lambda$ is not surjective. Note that by definition of the equivalence class $\sim$, if $Q_1,Q_2\in\sym(\inv;x)$ are such that $Q_1\sim Q_2$, then $\det(Q_1)^{-1}\det(Q_2)$ is a square in $\kk^\times$. This holds since $\det(a^\inv)=\det(a)$ for all $a\in \matr_N(\kk)$. By the same token, it follows that the $\det(w^\inv w)$ is a square in $\kk^\times$ for all $w\in \GL_N(\kk)$. 

Therefore, to show that $\Lambda$ is not surjective, it suffices to show that $\sym(\inv;x)$ contains elements whose determinant is not a square in $\kk$. One may take, for example, the element $Q=\delta \cdot 1_N$, for $\delta\in \kk^\times$ non-square.
\end{list}
\end{proof}

\subsubsection{Centralizers of regular elements}\label{subsubsection:cetralizers-sp2nso2n+1}

Finally, we compute the order of the centralizer of a regular element of $\mfr{g}_1$. The analysis we propose is analogous to \cite[Proposition~4.4]{KOS}. 

\begin{lem}\label{lem:short-exact-sequence} Let $x\in\mfr{g}_1$ be regular with minimal polynomial \[m_x(t)=t^{d_1}\prod_{i=1}^{d_2}\varphi_i(t)^{l_i}\prod_{i=1}^{d_3}\theta_i(t)^{r_i},\] where the product on the right hand side is as in \eqref{equation:m_x-decompose}, with $\theta_i(t)=\tau_i(t)\tau_i(-t)$. The determinant map induces a short exact sequence
\begin{equation}
\label{ses}
1\to \CC_{G_1}(x)\to \UU_1(\spring{t^{d_1}})\times \prod_{i=1}^{d_2}\UU_1(\spring{\varphi_i^{l_i}})\times \prod_{i=1}^{d_3}\GL_1(\spring{\tau_i^{r_i}})\xrightarrow{\det} Z\to 1
\end{equation}
where $Z\subseteq\kk^\times$ is the subgroup of order $2$ if $N$ is odd and trivial otherwise.
\end{lem}
\begin{proof}
As shown in the proof on \autoref{propo:decomposition-of-sim-class1}, the centralizer of $x$ in $\GL_N(\kk)$ is isomorphic to the group of units of the ring $\mcal{C}$, i.e. the direct product 
\[\CC_{\GL_N(\kk)}(x)\simeq \GL_1(\spring{t^{d_1}})\times\prod_{i=1}^{d_2}\GL_1(\spring{\varphi_i^{l_i}})\times\prod_{i=1}^{d_3} \GL_1(\spring{\theta_i^{r_i}}).\]
Furthermore, the involution $\inv$ of $\GL_N(\kk)$ restricts to an involution of $\CC_{\GL_N(\kk)}(x)$ which is transferred via this isomorphism to the involution $\sigma_{m_x}$, induced by $t\mapsto -t$, and restricts to the involution $\sigma_f$ on each of the factors $\GL_1(\spring{f})$ for $f\in\set{t^{d_1},\varphi_i^{l_i},\theta_i^{r_i}}$. 

The additional condition $z^\inv z=1$, and the fact that $\inv$ preserves all factors in the decomposition~\eqref{equation:C-decomposition}, imply that the centralizer of $x$ in $G_1$ is embedded in the group \[ \UU_1(\spring{t^{d_1}})\times \prod_{i=1}^{d_2}\UU_1(\spring{\varphi_i(t)^{l_i}}\times\prod_{i=1}^{d_3}\UU_1(\spring{\theta_i(t)^{r_i}}).\]

Similarly to \autoref{propo:decomposition-of-sim-class1}, the map $\sigma_{\theta_i^{r_i}}$ acts on the factors $ \GL_1(\spring{\theta_i(t)^{r_i}})\simeq \GL_1(\spring{\tau_i(t)^{r_i}})\times \GL_1(\spring{\tau_i(-t)^{r_i}})$ as $(\xi,\nu)\mapsto (\iota^{-1}(\nu),\iota(\xi))$, where $\iota:\spring{\tau_i(t)^{r_i}}\to\spring{\tau_i(-t)^{r_i}}$ is the isomorphism induced from $t\mapsto -t$. It follows from this that $(\xi,\nu)\in \UU_1(\spring{\theta_i^{r_i}})$ if and only if $\iota(\xi)=\nu^{-1}$, and hence that $\UU_1(\spring{\theta_i^{r_i}})\simeq \GL_1(\spring{\tau_i^{r_i}})$.

Lastly, we compute order of the group $Z$. Since for any $w\in \GL_N(\kk)$ we have that $\det(w^\inv)=\det(w)$, it follows that the condition $w^\inv w=1$ implies that $\det(w)\in\set{\pm 1}$. Thus, to complete the lemma, we need to show that both values occur in the case of $N$ odd, and that only $1$ is possible for $N$ even. Both statements are well-known. The former can be proved simply by considering the elements $\pm 1\in \GL_N(\kk)$, while the latter can be deduced by considering the Pfaffian of the matrix $w^t \BB w=\BB$.
\end{proof}

\begin{lem}\label{lem:unitary-of-kphi}
Let $f\in\kk[t]$ be a monic irreducible polynomial with $f(-t)=\pm f(t)$ and let $r\in\dbN$. Let $E_{f^r}$ denote the image of the map $z\mapsto \sigma_{f^r}(z)\cdot z:\GL_1(\spring{f^r})\to \GL_1(\spring{f^r})$. Given $y\in \GL_1(\spring{f^r})$ it holds that $y\in E_{f^r}$ if and only if 
\begin{enumerate}
\item $\sigma_{f^r}(y)=y$, and
\item there exists $z\in \GL_1(\spring{f(t)^r})$ such that $y\equiv z\sigma_{f^r}(z)\pmod{f}$. 
\end{enumerate} 
In particular, we have
\[\abs{E_{f^r}}=\begin{cases}q^{\frac{1}{2}r\deg f}(1-q^{-\frac{1}{2}\deg f})&\text{if }f(t)\ne t\\\frac{q-1}{2}q^{\lceil\frac{r}{2}\rceil-1}&\text{if }f(t)=t.
\end{cases}\]
\end{lem}
\begin{proof}
Let $W$ denote the vector space underlying the ring $\spring{f^r}$ and let $C$ be the bilinear form defined on $W$ as  in \autoref{lem:conj-to-g1}. Let $x$ be the linear operator defined on $W$ by multiplication by $t$. The map $t\mapsto x$ sets up a ring isomorphism of $\spring{f^r}$ with the ring $\mcal{C}\subseteq\matr_{r\cdot \deg f}(\kk)$ of matrices commuting with $x$, and the involution $\inv$ on $\mcal{C}$ is identified with the ring involution $\sigma_{f^r}$. Note that, in the current setting, if $y\in \spring{f^r}$ is the image modulo $(f^r)$ of a polynomial $\tilde{y}(t)$, then the assumption $\sigma_{f^r}(y)=y$ is equivalent to $\tilde{y}(x)\in \mcal{C}$ satisfying $\tilde{y}(x)^\inv=\tilde{y}(x)$ or, in the notation of \autoref{subsubsection:wall}, to $\tilde{y}(x)\in \sym(\inv;x)$. Also, the nilpotent radical of $\mcal{C}$ is given as the image of the ideal $(f)\subseteq \spring{f^r}$. The equivalence stated in the lemma now follows from \autoref{lem:CmodN-suff-cond}, by taking $Q_1=1$ and $Q_2=\tilde{y}(x)\in\sym(\inv;x)$.

We now compute the cardinality of $E_{f^r}$. In the case $f(t)=t$, the equivalence proved above implies that $E_{f^r}$ can be identified with the subgroup of the ring $\kk[t]/(t^{r})$ of truncated polynomials of degree no greater than $r-1$, which consists of even polynomials whose constant term is an invertible square of $\kk$. Hence $\abs{E_{f^r}}=\frac{q-1}{2}q^{\lceil\frac{r}{2}\rceil-1}$. 

In the complementary case, by irreducibility, necessarily $f(t)=f(-t)$ and has even degree. In this case, by the Jordan-Chevalley Decomposition Theorem, there exist polynomials $S,H\in~ \kk[t]$ such that the endomorphism $S(x)$ (resp. $H(x)$) acts semisimply (resp. nilpotently) on the vector space $W=\spring{f^r}$, on which $x$ acts by multiplication by $t$, and such that $H(t)+S(t)\equiv t\pmod{f(t)^r}$ (see \cite[\S~4.2]{HumphreysLie}; note that $S,H\in\kk[t]$ is possible since $\kk$ is perfect). It follows that $\spring{f^r}\simeq \kk[x]=\kk[S(x)][H(x)]$. A quick computation shows that the minimal polynomials of $S(x)$ and $H(x)$ are $f(t)$ and $t^r$ respectively, and thus $\spring{f}\simeq \kk[S(x)][H(x)]\simeq\spring{f}\otimes_\kk(\kk[h]/(h^r))$. Moreover, by the properties of the Jordan-Chevalley decomposition, both $S(t)$ and $H(t)$ satisfy $S(-x)=-S(x)$ and $H(-x)=-H(x)$ \cite[\S~3, Proposition~3]{BourbakiLieCh1}. Thus, under this identification, the involution $\sigma_{f^r}$ is transferred to an involution of $ \spring{f}\otimes_\kk(\kk[h]/(h^r))$, mapping $h$ to $-h$ and acting as $\sigma_f$ on the field $\spring{f}$. 

By the equivalence in the lemma, and the theory of finite fields, the group $E_{f^r}$ is identified with the subgroup of  $(\spring{f^r})^\times$ of elements fixed by $\sigma_{f^r}$. Using the identification above, this subgroup consists of elements of the form $\sum_{i=0}^{r-1}a_i\otimes h^i$, with $a_0,\ldots,a_{r-1}\in \spring{f}$, $a_0\ne 0$, and 
\[\sigma_f(a_i)=\begin{cases}a_i&\text{if $i$ is even}\\-a_i&\text{if $i$ is even}.\end{cases}\]
The equality $\abs{E_{f^r}}=q^{\frac{1}{2}r \deg f}(1-q^{-\frac{1}{2}\deg f})$ now follows by direct computation.
\end{proof}

\begin{propo}\label{corol:centralizer-size-sp2nso2n+1} Let $x\in \mfr{g}_1$ be a regular element with minimal polynomial $m_x\in\kk[t]$. Let $\ttau(m_x)=(r(m_x),\A(m_x),\B(m_x))\in\mcal{X}_n$ be the type of $m_x$ (see \autoref{defi:poly-type}). Then 
\[\abs{\CC_{G_1}(x)}=2^{\nu}q^n\prod_{d,e}\left(1+q^{-d}\right)^{\A_{d,e}(m_x)}\cdot\left(1-q^{-d}\right)^{\B_{d,e}(m_x)},\]
where $\nu=1$ in the case where $N=2n$ is even and $r(m_x)>0$, and $\nu=0$ otherwise.
\end{propo}
\begin{proof}
Let $m_x=t^{d_1}\prod_{i=1}^{d_2}\varphi_i^{l_i}\prod_{i=1}^{d_3}\theta_i^{r_i}$ be a decomposition of $m_x$ as in \eqref{equation:m_x-decompose}, with $\varphi_i$ even and irreducible, and $\theta_i(t)=\tau_i(t)\tau_i(-t)$ with $\tau_i(t),\tau_i(-t)$ irreducible and coprime.
Note that by definition of $\ttau(m_x)$ we have that $r(m_x)=\lfloor\frac{d_1}{2}\rfloor$.

In view of \autoref{lem:short-exact-sequence} it suffices to show the following three assertions.
\begin{enumerate}
\item $\abs{\UU_1(\spring{t^{d_1}})}=2q^{r(m_x)}$;
\item $\abs{\UU_1(\spring{\varphi_i^{l_i}})}=q^{\frac{1}{2}l_i\cdot \deg\varphi_i}(1+q^{-\frac{1}{2}\deg\varphi_i})$;
\item $\abs{\GL_1(\spring{\tau_i^{r_i}})}=q^{r_i\cdot\deg\tau_i}(1-q^{-\deg\tau_i})$.
\end{enumerate}

Note that for any irreducible  polynomial $f(t)\in\kk[t]$ and $r\in\dbN$, invoking the Jordan-Chevalley Decomposition as in \autoref{lem:unitary-of-kphi}, the group $\GL_1(\spring{f^r})$ is isomorphic to the group of units of the ring $\spring{f}[u]/(u^r)$, and hence $\abs{\GL_1(\spring{f^r})}=q^{r\cdot \deg f}\left(1-q^{-\deg f}\right)$. Assertion (3) follows by taking $f(t)=\tau_i(t)$ and $r=r_i$. 

Assertions (1) and (2) follow from the exactness of the sequence \[1\to \UU_1(\spring{f^r})\to \GL_1(\spring{f^r})\xrightarrow{x\:\mapsto\sigma_{f^r}(x)\cdot x} E_{f^r}\to 1,\]
which holds for any irreducible $f\in\kk[t]$ with $f(-t)=\pm f(t)$ and $r\in\dbN$, and from the computation of $\abs{E_{f^r}}$ in \autoref{lem:unitary-of-kphi} and $\abs{\GL_1(\spring{f^r})}$ for the case where $f(t)=t$ and $r=d_1$, and the cases $f(t)=\varphi_i(t)$ and $r=l_i$.

\end{proof}
The final assertion of \autoref{theo:orbits-sp2nso2n+1} follows directly from \autoref{corol:centralizer-size-sp2nso2n+1}.

\subsection{Even dimensional special orthogonal groups}\label{subsection:so2n}

The following lemma demonstrates the failure of the first assertion of \autoref{theo:orbits-sp2nso2n+1} in the even orthogonal case.
\begin{lem}\label{lem:nilpotent-reg-glN-not-reg-soN}Let $N=2n$ be even and let $x\in\mfr{gl}_N(\bkk)$ be a regular nilpotent element. Then $x$ is not anti-symmetric with respect to any non-degenerate symmetric bilinear form on $V=(\bkk)^N$.  
\end{lem}
\begin{proof}
Note that, as $x$ is conjugate to an $N\times N$ nilpotent Jordan block, the kernel of $x$ is one dimensional. Assume towards a contradiction that $C$ is a symmetric non-degenerate bilinear form on $V$ such that $x$ is $C$-anti-symmetric. Consider the form $F(u,v)=C(u,xv)$ on $V$. By assumption the $C(xu,v)+C(u,xv)=0$, we have that $F$ is anti-symmetric. Additionally, the radical of $F$ coincides with the kernel of $x$, by non-degeneracy of $C$. By properties of antisymmetric forms, it follows that the kernel of $x$ is even-dimensional. A contradiction.
\end{proof}

Nonetheless, regular nilpotent elements in the case of even-dimensional special orthogonal groups are well-known to exist \cite[III,\:1.19]{SteinbergSpringer}. In \autoref{lem:nilpotent-reg-element-so2n} below we shall construct such an element and compute its centralizer.

Recall that non-degenerate symmetric bilinear forms on $V=\kk^N$ are classified by the dimension of a maximal totally isotropic subspace of $V$ with respect to the given form (i.e. its Witt index), and that over a finite field of odd characteristic there are exactly two such forms, upto isometry. We fix $B^+$ and $B^-$ to be bilinear forms on $V$ of Witt index $n$ and $n-1$, respectively. In suitable bases, the forms $B^+$ and $B^-$ are represented by the matrices 
\begin{equation}\label{equation:J+J-}\BB^+=\mat{0&1\\1&0\\&&\ddots\\&&&0&1\\&&&1&0\\&&&&&0&1\\&&&&&1&0}\quad\text{or}\quad\BB^-=\mat{0&1\\1&0\\&&\ddots\\&&&0&1\\&&&1&0\\&&&&&1&0\\&&&&&0&\delta},\end{equation}
where $\delta\in\oo^\times$ is a fixed non-square.

Given $\epsilon\in\set{\pm 1}$, let $G_1^\epsilon=\SO_N^\epsilon(\kk)$ and $\mfr{g}_1^\epsilon=\mfr{so}_{N}^\epsilon(\kk)$ be the group of isometries of determinant $1$ and the Lie-algebra of anti-symmetric operators with respect to the form $B^\epsilon$. We will also occasionally use the colloquial notation $G_1^{\pm}=G_1^{+} \cup G_1^-$ and $\mfr{g}_1^{\pm}=\mfr{g}_1^+\cup \mfr{g}_1^-$. For example, the phrase \textit{$x$ is a regular element of $\mfr{g}_1^\pm$} indicates that $x$ is either a regular element of $\mfr{g}_1^+$ or of~$\mfr{g}_1^-$.
\subsubsection{Similarity classes of regular elements}\label{subsubsection:simclasses-so2n}
In this section we prove the first assertion of \autoref{theo:orbits-so2n}, which classifies the similarity classes of $\mfr{gl}_N(\oo)$ whose intersection with $\mfr{g}_1^\pm$ consists of regular elements. Following this, we differentiate whether such a similarity class intersects $\mfr{g}_1^+$ or $\mfr{g}_1^-$. 

Note that if $x\in\mfr{gl}_N(\kk)$ is a non-singular element whose minimal polynomial $m_x$ is even and has degree $N$ then, by applying the argument of \autoref{lem:conj-to-g1}.(2) verbatim, we have that $x$ is anti-symmetric with respect to a non-degenerate symmetric bilinear form and hence similar to an element of $\mfr{g}_1^\pm$. By \autoref{corol:reg-g1-reg-glN-nonsing}, all non-singular regular elements of $\mfr{g}_1$ are obtained in this manner. Thus, for $x\in\mfr{gl}_N(\kk)$ non-singular, it holds that $x$ is similar to a regular element of $\mfr{g}^{\pm}$ if and only if the minimal polynomial of $x$ is even and of degree $N$.

As explained below (see \autoref{propo:orbit-decomposition=so2n}), the case of singular regular elements of $\mfr{g}_1^\pm$ is essentially reduced to the study of nilpotent regular elements. These elements are considered in the following lemma.
\begin{lem}\label{lem:nilpotent-reg-element-so2n}
Let $x\in \mfr{gl}_N(\kk)$ have minimal polynomial $m_x(t)=t^{N-1}$. Then $x$ is similar to a regular nilpotent element of $\mfr{g}_1^+$, as well as to a regular nilpotent element of $\mfr{g}_1^-$.
\end{lem}

\begin{proof}
By considering the Jordan normal form of such an element $x$, there exist elements $v_0,u_0\in V$ with $u_0\in\ker(x)$ and such that $\mcal{E}=\set{v_0,xv_0,\ldots,x^{N-2}v_0,u_0}$ is a $\kk$-basis for $V$. 

Let $\mcal{E}'=\set{v_0,\ldots x^{N-2}v_0}$ and $V'=\Span_\kk \mcal{E}'$. Since the element $x\mid_{V'}\in\mfr{gl}(V')$ has minimal polynomial $t^{N-1}=t^{\dim V'}$, it is regular in $\mfr{gl}(V')$. By the proof of \autoref{lem:conj-to-g1}, there exists a non-degenerate symmetric bilinear form $C'$ on $V'$, with respect to which $x\mid_{V'}$ is anti-symmetric. We wish to extend $C'$ to a non-degenerate symmetric bilinear form on $V$, with respect to which $x$ is anti-symmetric. This is equivalent to finding an invertible matrix $\mbf{d}\in \matr_N(\kk)$, whose top-left $(N-1)\times (N-1)$ submatrix coincides with the matrix $\mbf{c}$ of \autoref{exam:reg-nil-elt-so2n+1-sp2n} (see \eqref{eqution:C-nil-elt}), and such that 
\begin{equation}\label{equation:nil-defiso2n}
\mbf{d}^t\nil+\nil\mbf{d}=0\quad\text{where}\quad\nil=[x]_\mcal{E}=\mat{0&1\\&\ddots&\ddots\\&&0&1\\&&&0&0\\&&&&0}.
\end{equation}

A short computation shows that the matrix
\begin{equation}\label{equation:d-eta-defin}
\mbf{d}=\mbf{d}_\eta=\mat{&&&1\\&&-1
\\&\reflectbox{$\ddots$}
\\1\\&&&&\eta},
\end{equation}
where $\eta\in \kk^\times$ satisfies the required equality. Furthermore, by applying a signed permutation to $\mcal{E}$, one may verify easily that $\mbf{d}_\eta$ is congruent to the matrix $\BB^+$ of \eqref{equation:J+J-} if $\eta$ is a square, and to $\BB^-$ otherwise. Thus, $x$ is similar in this case to elements of both $\mfr{g}_1^+$ and of $\mfr{g}_1^-$. 

Lastly, we need to verify that $x$ is similar to a \textit{regular} element of $\mfr{g}_1^\pm$. To do so, we pass to the algebraic closure $\bkk$ of $\kk$ and compute the centralizer in $\GG(\bkk)$ of an element $zxz^{-1}\in \mfr{g}_1$. Working in the basis $\mcal{E}$, by direct computation, one sees that the centralizer of $x$ in $\matr_N(\bkk)$ can be identified with the set of matrices $\mbf{y}=\smat{\mbf{A}&\mbf{v}\\\mbf{u}^t&r}$, where
\begin{enumerate}
\item $\mbf{A}\in\matr_{N-1}(\bkk)$ and commutes with the restriction of $\nil$ to $V'=\Span_{\bkk}{\mcal{E}'}$;
\item $\mbf{u},\mbf{v}\in(\bkk)^{N-1}$ are elements of the kernel of $\nil$ and $\nil^t$, respectively, and hence of the form $\mbf{v}=\mat{v_1&0&\ldots&0}^t$ and $\mbf{u}=\mat{0&\ldots&0&u_{N-1}}^t$; and
\item $r\in\bkk$ is arbitrary.
\end{enumerate}
As in \autoref{exam:reg-nil-elt-so2n+1-sp2n}, the centralizer of $z x z^{-1}\in\mfr{g}_1$ is conjugated in $\GL_N(\bkk)$ to the group \[\set{\mbf{y}\in\CC_{\GL_N(\bkk)}(\nil)\mid \mbf{y}^t\mbf{d}\mbf{y}=\mbf{d}}.\]
Computing its Lie-algebra, which consists of matrices $\mbf{y}\in\CC_{\matr_N(\bkk)}(\nil)$ satisfying $\mbf{y}^t\mbf{d}+\mbf{d}\mbf{y}=0$, we get the additional three conditions
\begin{enumerate}
\item $\mbf{A}^t\mbf{c}+\mbf{c}\mbf{A}=0$, where $\mbf{c}$ 
 is as in \autoref{exam:reg-nil-elt-so2n+1-sp2n};
\item $\eta\mbf{u}+\mbf{c}\mbf{v}=0$, i.e. $v_1=-\eta u_{N-1}$;  and
\item $2\eta r=0$, and hence $r=0$. 
\end{enumerate}
It follows that $\CC_{\GGamma_1}(zx z^{-1})$ is at most $n$-dimensional, and hence $x$ is regular.
\end{proof}

To streamline the analysis of nilpotent regular orbits, let us fix some notation.
\begin{nota}\label{nota:regular-nilpotent-centralizer}
Given a matrix $\mbf{A}\in\matr_{N-1}(\kk)$, column vectors $\mbf{v},\mbf{u}\in\kk^{N-1}$ and $r\in\kk$, let $\cma(\mbf{A},\mbf{v},\mbf{u},r)$ denote the $N\times N$ matrix
\[\cma(\mbf{A},\mbf{v},\mbf{u},r)=\mat{\mbf{A}&\mbf{v}\\\mbf{u}^t&r}.\]
We also write $\mbf{A}^\invv$ for the matrix $\mbf{c}\mbf{A}^t\mbf{c}$, where $\mbf{c}$ is as in \autoref{exam:reg-nil-elt-so2n+1-sp2n}. Note that, in the case where $\mbf{d}=\mbf{d}_\eta$ is the representing matrix for the symmetric bilinear form given on $V$, we have that 
\begin{equation}
\label{equation:cma}
\cma(\mbf{A},\mbf{v},\mbf{u},r)^\inv=\mat{\mbf{A}^\invv&\eta\mbf{cu}\\\eta^{-1}\mbf{v}^t\mbf{c}&r}=\cma(\mbf{A}^\invv,\eta\mbf{cu},\eta^{-1}\mbf{cv},r).\end{equation}
\end{nota}

The next step of the computation is to differentiate whether a given element $x\in\mfr{gl}_N(\kk)$, which is similar to a regular element of $\mfr{g}_1^\pm$, is similar to either $\mfr{g}_1^+$ or $\mfr{g}_1^-$. We first consider two specific cases, depending on the minimal polynomial of $x$.
\begin{lem}[{cf. \cite[\S~2.6.(B).(i) and (i')]{WallIsometries}}]\label{lem:basic-cases}
Let $x\in\mfr{gl}_N(\kk)$ have minimal polynomial $m_x$. Assume $x$ is similar to a regular element of $\mfr{g}_1^\pm$. 
\begin{enumerate}
\item If $m_x(t)=f(t)f(-t)$ for some polynomial $f\in\kk[t]$ with $f(0)\ne 0$, then $x$ is similar to an element of $\mfr{g}_1^+$, and not to an element of $\mfr{g}_1^-$.
\item If $m_x=\varphi^r$ for $\varphi\in\kk[t]$ an even irreducible polynomial and $r\in \dbN$ odd, then $x$ is similar to a regular element of $\mfr{g}_1^-$ and not to an element of $\mfr{g}_1^+$.
\end{enumerate}
\end{lem}
\begin{proof} Let $C$ be a non-degenerate symmetric bilinear, with respect to which $x$ is $C$-anti-symmetric. We will show that $C$ necessarily has Witt index $n$ in the first case and $n-1$ in the second case. 

\begin{list}{\arabic{list}.}{\usecounter{list}\setlength{\leftmargin}{5pt}\setlength{\itemsep}{2pt}\setlength{\labelwidth}{-5pt}}
\item By the assumption $m_x(0)\ne 0$ and \autoref{corol:reg-g1-reg-glN-nonsing}, it follows that $x$ is also a regular element of $\mfr{gl}_N(\kk)$, and hence the space $V$ is cyclic as a $\kk[x]$ module. Put $W=f(x) V$. Then $W$ is isomorphic, as a $\kk[x]$-module, to $V/f(-x)V$, and hence is of dimension $n=\frac{N}{2}$ over $\kk$. Additionally, for any $u,v\in V$ we have $C(f(x)u,f(x)v)=C(f(x)f(-x)u,v)=0$, and hence $W$ is totally isotropic.
\item Let us first consider the case where $r=1$, and hence $V$ is isomorphic to the field extension~$\spring{\varphi}$ of $\kk$. Furthermore, the map $\sigma_{\varphi}\in\aut_\kk(\spring{\varphi})$, induced from $t\mapsto -t$ is a field involution of $\spring{\varphi}$ over~$\kk$, with fixed field $\msf{K}$, such that $\abs{\spring{\varphi}:\msf{K}}=2$. Note that in this setting, without loss of generality, we may assume that $C(u,v)= \Tr_{\spring{\varphi}/\kk}(\sigma_\varphi(u)v)$ for all $u,v\in V$. Indeed, invoking the separability of the extension~$\spring{\varphi}/\kk$, there exists an element $c\in \spring{\varphi}$ such that $C(u,1)=\Tr_{\spring{\varphi}/\kk}(c\cdot u)$ for all $u\in\spring{\varphi}$. From the symmetry of $C$ and the invariance of $\Tr_{\spring{\varphi}/\kk}$ under $\sigma_\varphi$, it can be deduced that in fact $c\in\msf{K}$. By the theory of finite fields, there exists an element $d\in \spring{\varphi}$ such that $c=\sigma_{\varphi}(d)d$. It follows that multiplication by $d$ is an isometry of $C$ with the trace pairing $(u,v)\mapsto \Tr_{\spring{\varphi}/\kk}(\sigma_\varphi(u)v)$.

Note that an element $u\in \spring{\varphi}$ is isotropic if and only if $\sigma_\varphi(u) u$ is a traceless element of~$\msf{K}$. Since the number of non-zero traceless elements in the extension $\msf{K}/\kk$ is $q^{n-1}-1$, and by the surjectivity of the norm map $\Nr_{\spring{\varphi}/\msf{K}}$,  it follows that the number of non-zero isotropic element of $\spring{\varphi}$ is $(q^n+1)(q^{n-1}-1)$. The fact that $C$ is of Witt index $n-1$ now follows as in \cite[\S~3.7.2]{WilsonFinite}.

For the case $r>1$, put $l=\lfloor\frac{r}{2}\rfloor$ and $U=\varphi(x)^{l+1}V$. Then, similarly to (1), $U$ is an isotropic subspace of $V$, with perpendicular space $U^\perp=\varphi(x)^{l}V$. Moreover, the form $C$ reduces to a non-degenerate symmetric bilinear form on the quotient space $U^\perp/U$, on which $x$ acts as an anti-symmetric operator with minimal polynomial $\varphi$. By the case $r=1$, we find a two-dimensional anisotropic subspace $\bar{L}\subseteq U^\perp/U$, whose pull-back to $U^\perp$ is contains a two-dimensional anisotropic subspace of $V$. It follows that the Witt index of $C$ is necessarily $n-1$.
\end{list}
\end{proof}

Having \autoref{lem:basic-cases} at hand, we need one more basic tool in order to complete the classification of similarity classes containing regular elements of $\mfr{g}_1^\pm$. 
\begin{nota} Given a finite, even-dimensional vector space $U$ over $\kk$ with a non-degenerate symmetric bilinear form $C$, put $\delta_U=1$ if $U$ is of Witt index $\frac{1}{2}\dim_\kk U$ and $\delta_U=-1$ otherwise. 
\end{nota}
\begin{lem}\label{lem:direct-sum-of-forms}
Let $U,W$ be finite, even dimensional vector spaces over $\kk$ with non-degenerate symmetric bilinear forms $C_U$ and $C_V$ respectively. Endow the space $U\oplus W$ with the non-degenerate symmetric bilinear form $C_{U\oplus W}(u+w,u'+w')=C_U(u,u')+C_W(w,w')$ where $u,u'\in U$ and $w,w'\in W$. Then
\[\delta_{U\oplus W}=\delta_U\cdot\delta_W.\]
\end{lem} 
\begin{proof}
The lemma follows, e.g., from \cite[\S~3.7.4,\:p.~68]{WilsonFinite}, noting that the direct product of the groups of isometries of $C_U$ and $C_W$ is embedded in the group of isometries of $C_{U\oplus W}$.
\end{proof}

We are now ready to complete the proof of the first and second assertions of \autoref{theo:orbits-so2n}.

\begin{propo}\label{propo:conj-to-g1-so2n} Let $x\in\mfr{gl}_N$ have minimal polynomial $m_x$. Assume $m_x(-t)=(-1)^{\deg m_x}m_x(t)$ and let \[m_x(t)=t^{d_1}\prod_{i=1}^{d_2}\varphi_i^{l_i}\prod_{i=1}^{d_3}\theta_i^{r_i}\] a decomposition as in \eqref{equation:m_x-decompose}, with $\varphi_i(t)$ even and irreducible, and $\theta_i(t)=\tau_i(t)\tau_i(-t)$ with $\tau_i(t)$ monic, irreducible and coprime to $\tau_i(-t)$.
\begin{enumerate}
\item If $d_1>0$ then $x$ is similar to a regular element of $\mfr{g}_1^\pm$ if and only if $\deg m_x=N-1$. Moreover, in this case $x$ is similar to an element of $\mfr{g}_1^+$ as well as to an element of $\mfr{g}_1^-$.
\item Otherwise, if $d_1=0$ then $x$ is similar to a regular element of $\mfr{g}_1^\pm$ if and only if $\deg m_x=N$. In this case, put $\omega(m_x)=\sum_{i=1}^d l_i$.
\begin{enumerate}
\item If $\omega(m_x)$ is even, then $x$ is similar to an element of $\mfr{g}_1^+$ and not to an element of $\mfr{g}_1^-$.
\item Otherwise, if $\omega(m_x)$ is odd, then $x$ is similar to an element of $\mfr{g}_1^-$ and not to an element of $\mfr{g}_1^-$. 
\end{enumerate}  
\end{enumerate}
\end{propo}
\begin{proof}
Considering the primary canonical form of $x$, the space $V$ decomposes as a $\kk[x]$-invariant direct sum $V=W_{t^{d_1}}\oplus\bigoplus_{i=1}^{d_2} W_{\varphi_i^{l_i}}\oplus\bigoplus_{i=1}^{d_3} W_{\theta_i^{r_i}}$, where the restriction of $x$ to the spaces $W_{f}$ has minimal polynomial $f(t)$, with $f\in\set{t^{d_1},\varphi_i^{l_i},\theta_i^{r_i}}$. 

For any $f(t)\ne t^{d_1}$, the restriction of $x$ to $W_{f}$ is a regular element of $\mfr{gl}(W_f)$. By \autoref{corol:reg-g1-reg-glN-nonsing}, the space $W_f$ is endowed with a non-degenerate symmetric bilinear form on which $x\mid_{W_f}$ acts as an anti-symmetric operator. Furthermore, by \autoref{lem:basic-cases}, in the case where $f=\theta_i^{r_i}$ for $i=1,\ldots,d_3$ or $f=\varphi_i^{l_i}$ with $l_i$ even, then $\delta_{W_f}=+1$. Otherwise, if $f=\varphi_i^{l_i}$ with $l_i$ odd, $\delta_{W_f}=-1$. Assertion (2), in which $d_1=0$ is assumed, now follows from \autoref{lem:direct-sum-of-forms}.

In the case where $d_1>0$, the assumption $\deg m_x=N-1$  implies that $t\cdot m_x(t)=c_x$, where $c_x(t)$ is the characteristic polynomial of $x$. It follows that the restriction of $x$ to $W_{t^{d_1}}$ has minimal polynomial $t^{d-1}$, and hence, by \autoref{lem:nilpotent-reg-element-so2n}, is antisymmetric with respect to non-degenerate symmetric forms of Witt index $\frac{d_1}{2}$ as well as $\frac{d_1}{2}-1$. Thus $\delta_{W_{t^{d_1}}}$ can be taken to be either $+1$ or $-1$. By the case where $x$ is non-singular, and by \autoref{lem:direct-sum-of-forms}, $x$ is similar to an element of $\mfr{g}_1^+$ as well as to an element of $\mfr{g}_1^-$.
\end{proof}
\subsubsection{From Similarity classes to adjoint orbits}\label{subsubsection:simtoconj-so2n}
Our next goal, once the similarity classes containing regular elements of $\mfr{g}_1^{\pm}$ have been classified, is to describe the set $\Pi_x=\Ad(\GL_N(\kk))x\cap \mfr{g}_1^\epsilon$  into $\Ad(G_1^\epsilon)$-orbits, for $\epsilon\in\set{\pm 1}$ fixed. 
In order to complete the description, we require the following lemma, whose proof is appears after \autoref{propo:orbit-decomposition=so2n}.

\begin{lem}\label{lem:difference-of-squares}Assume $\abs{\kk}>3$ and $\Char(\kk)\ne 2$. For any element $\gamma\in\kk^\times$ there exist $\nu,\delta\in\kk^\times$ such that $\nu\in(\kk^\times)^2,\:\delta\in \kk^\times\setminus(\kk^\times)^2$ and such that $\gamma=\nu-\delta$.
\end{lem}

\begin{propo}\label{propo:orbit-decomposition=so2n} Assume $\abs{\kk}>3$. Fix $\epsilon\in\set{\pm 1}$ and let $x\in\mfr{g}^\epsilon_1$ be regular. If $x$ is singular, then the intersection $\Ad(\GL_N(\kk))x\cap \mfr{g}_1^\epsilon$ is the disjoint union of two distinct $\Ad(G_1^\epsilon)$-orbits. Otherwise, $\Ad(\GL_N(\kk))x\cap \mfr{g}_1^\epsilon=\Ad(G_1^\epsilon)x$.
\end{propo}

\begin{proof} 
In the notation of \autoref{propo:orbit-decomposition}, let $\Pi_x=\Ad(\GL_N(\kk)) x\cap \mfr{g}_1$ and $\Theta_x$ the set of equivalence classes in $\sym(\inv;x)=\set{Q\in\CC_{\GL_N(\kk)}(x)\mid Q^\inv= Q}$ under the equivalence relation~$\sim$, defined in \eqref{equation:sim-defi}. Let $\Lambda:\Pi_x\to \Theta_x$ be the map $wxw^{-1}\mapsto [w^\inv w]\in\Theta_x$, for $y=wxw^{-1}\in\Pi_x$.

In the case where $x$ is non-singular, by applying the argument of \autoref{propo:decomposition-of-sim-class1} for non-singular elements verbatim, we have that $\Theta_x$ consists of a single element and therefore that $\Pi_x=\Ad(G_1^\epsilon) x$.

Furthermore, in the case where $x$ is singular, by considering the decomposition of $x$ into primary rational canonical forms, one may restrict $x$ to a maximal subspace of $\kk^N$ on which $x$ acts as a regular nilpotent element. This subspace is even-dimensional and admits an orthogonal complement, on which $x$ acts as a non-singular regular element. Additionally, any operator commuting with $x$ must preserve this subspace as well as its orthogonal complement. It follows that to prove the proposition in the case where $x$ is singular it is sufficient to consider the case where $x$ is a nilpotent regular element of $\mfr{g}_1^\epsilon$. 

In this case, by the uniqueness of a nilpotent regular element in $\g(\bkk)$ \cite[III, Theorem~1.8]{SteinbergSpringer}, we may invoke \autoref{lem:nilpotent-reg-element-so2n} and fix a basis $\mcal{E}$, with respect to which $x$ is represented by the matrix $\nil$, defined in \eqref{equation:nil-defiso2n}, and that the ambient non-degenerate symmetric bilinear form is represented in $\mcal{E}$ by the matrix $\mbf{d}=\mbf{d}_\eta$ of \eqref{equation:d-eta-defin}, 
where $\eta\in\kk^\times$ is a square if $\epsilon=1$ and non-square otherwise.

The centralizer $\mcal{C}$ of $\nil$ in $\matr_N(\kk)$ is isomorphic to the ring of $\kk[x]$-endomorphisms of $\kk[x]\times \kk$, and can be realized as the set of matrices $\cma(\mbf{A},\mbf{v},\mbf{u},r)$ (see \autoref{nota:regular-nilpotent-centralizer}) with $\mbf{v}$ and $\mbf{u}$ elements of the kernel of $\nil$ and $\nil^t$ respectively, $\mbf{A}\in\matr_{N-1}(\kk)$ is an upper triangular T\"oplitz matrix, and $r\in\kk$. Note that the ideal generated by elements of the form $\cma(0_{N-1},\mbf{v},\mbf{u},0)\in\mcal{C}$ is nilpotent and in particular is contained in the nilpotent radical $\mcal{N}$ of $\mcal{C}$. It follows that the quotient ring $\mcal{C}/\mcal{N}$ is isomorphic to the \'etale algebra $\kk\times \kk$. Additionally, by \autoref{lem:CmodN-suff-cond}, we have that $\cma(\mbf{A},\mbf{v},\mbf{u},r)\sim \cma(\mbf{A}',\mbf{v}',\mbf{u}',r')$ if and only if there exists a block matrix $\cma(\mbf{q},0,0,s)$ such that \[\mat{\mbf{q}&\\&s}^\inv \mat{\mbf{A}&\mbf{v}\\\mbf{u}^t&r}\mat{\mbf{q}&\\&s}\equiv\mat{\mbf{A}'&\mbf{v}'\\\mbf{u}'^t&r'}\pmod{\mcal{N}}.\]
Applying a similar argument as in the nilpotent case of \autoref{propo:decomposition-of-sim-class1}, we have that the involution~$\inv$ restricts to the identity map on $\mcal{C}/\mcal{N}$ and hence that the quotient $\Theta_x$ of $\sym(\inv;x)$ by the relation $\sim$, defined in \autoref{subsubsection:wall}, is isomorphic to the quotient group $\kk^\times/(\kk^\times)^2\times \kk^\times/(\kk^\times)^2$ and is of order $4$.

The final step of the proof is to compute the image of the map $\Lambda$. Recall that $\Lambda$ maps an element $wxw^{-1}\in \Pi_x=\Ad(\GL_N(\kk))x\cap \mfr{g}_1^\epsilon$ to the equivalence class of $w^\inv w$ in $\Theta_x$. As in the odd orthogonal case, two elements which are equivalent with respect to $\sim$ must have determinant in the same coset of $\kk^{\times}/(\kk^\times)^2$. In particular, as $w^\inv w$ has square determinant, the image of $\Lambda$ in $\Theta_x$ is contained in the subset of equivalence classes in $\Theta_x$, containing block matrices $\cma(\mbf{A},0,0,r)$ with $\det{\mbf{A}}\equiv r\pmod{(\kk^\times)^2}$.

To complete the proof that $\abs{\im(\Lambda)}=2$ it suffices to find an element $w\in\GL_N(\kk)$ such that $wxw^{-1}\in\mfr{g}_1$ and such that $w^\inv w$ is a block matrix of the form $\cma(\mbf{A},0,0,r)$ with $\det\mbf{A},r~\notin~(\kk^\times)^2$.

Let $\eta\in\kk^\times$ be as above put $\alpha=(-1)^{(N-2)/2}$. Let $\nu\in(\kk^\times)^2$ and $\delta\in\kk^\times\setminus(\kk^\times)^2$ be such that $\alpha\eta=\nu-\delta$; see \autoref{lem:difference-of-squares}. Let $\nu_1\in\kk^\times$ be such that $\nu_1^2=\nu$, and put $z=\eta\cdot\nu_1^{-1}$. Let $w\in\GL_N(\kk)$ be represented in $\mcal{E}$ by the matrix $\mbf{w}$ of \eqref{equation:w-definition}, in which the upper-left scalar block with $\delta$ on the diagonal is $\left(\frac{N-2}{2}\right)\times\left(\frac{N-2}{2}\right)$.

Recalling that $w^\inv$ is represented by the matrix $\mbf{d}^{-1}\mbf{w}^t\mbf{d}$, one verifies by direct computation that $w^\inv w$ is given by the diagonal matrix $\cma(\delta 1_{N-1},0,0,\nu^{-1}\delta)$, and consequently, that $w^\inv w\in \sym(\inv;x)$ and $wxw^{-1}\in\mfr{g}^\epsilon_1$, and that $w^\inv w$ is not equivalent to $1_N$ under the relation~ $\sim$.
\begin{equation}
\label{equation:w-definition}
\mbf{w}=\mat{\delta\\
&\ddots\\&&\delta\\&&&\nu_1&&&&\alpha z\\&&&&1\\&&&&&\ddots\\&&&&&&1\\
&&&-\eta^{-1}\nu_1z&&&&1}.\end{equation}
\end{proof}
\begin{proof}[Proof of \autoref{lem:difference-of-squares}]
Let $\xi\in\kk^\times$ be a non-square, and let $\msf{K}=\spring{t^2-\xi}$ be the splitting field of $t^2-\xi$, with $\xi_1\in\msf{K}^\times$ a square root of $\xi$. The norm map $\Nr_{\msf{K}\mid \kk}:\msf{K}^\times\to \kk^\times$ is surjective and has fibers of order $q+1$. In particular, there exist $\nu_1,\delta_1\in\kk$ such that \[\Nr_{\msf{K}\mid\kk}(\nu_1+\xi_1\delta_1)=\nu_1^2-\xi\delta_1^2=\gamma.\]

We claim that $\nu_1$ and $\delta_1$ can be taken to be both non-zero. 
\begin{list}{}{\setlength{\leftmargin}{0pt}\setlength{\itemsep}{2pt}}
\item \textit{Case 1, $\gamma\in\kk^\times\setminus(\kk^\times)^2$}. Note that in this case we must have that $\delta_1\ne 0$, as otherwise $\gamma=\nu_1^2\in(\kk^\times)^2$. Furthermore, if $\nu_1=0$ for any pair $(\nu_1,\delta_1)$ such that $\nu_1^2-\xi\delta_1^2=\gamma$ then $\Nr_{\msf{K}\mid \kk}^{-1}(\gamma)\subseteq \xi_1\kk^\times$, and in particular has order smaller than $q$. A contradiction.
\item \textit{Case 2, $\gamma\in(\kk^\times)^2$}. Consider the set $\Nr_{\msf{K}\mid \kk}^{-1}(\gamma)\setminus\kk^\times$. Note that, as  $\abs{\Nr_{\msf{K}\mid\kk}^{-1}(\gamma)\cap \kk^\times}=2$ (namely, it consists of the two roots of $\gamma$ in $\kk$), the order of $\Nr_{\msf{K}\mid\kk}^{-1}(\gamma)\setminus\kk^\times$ is exactly $q-1$. Assume towards a contradiction that there is no solution $(\nu_1,\delta_1)\in\kk^\times\times\kk^\times$ for the equation \[\nu_1^2-\xi\delta_1^2=\Nr_{\msf{K}\mid \kk}(\nu_1-\xi_1\delta_1)=\gamma.\]
This implies that any solution not in $\kk^\times\times\set{0}$ is an element of $\set{0}\times\kk^\times$, or in other words, that $\Nr_{\msf{K}\mid \kk}^{-1}(\gamma)\setminus\kk^\times\subseteq\xi_1\kk^\times$. By considering the cardinality of the two sets, we deduce that this inclusion is in fact an equality. In particular, this implies that for any $\delta_1\in\kk^\times$, \[\Nr_{\msf{K}\mid\kk}(\xi_1\delta_1)=-\xi\delta_1^2=\gamma.\]
Thus, the set of squares in $\kk^\times$ equals the singleton set $\set{-\xi^{-1}\gamma}$. This contradicts the assumption $\abs{\kk}>3$.
\end{list}

The lemma follows by taking $\nu=\nu_1^2$ and $\delta=\xi\delta_1^2$.

\end{proof}

\subsubsection{Centralizers of regular elements}\label{subsubsection:cetralizers-so2n}
\begin{lem}\label{lem:short-exact-sequence-so2n} Let $\epsilon\in\set{\pm 1}$. 
  Let $x\in\mfr{g}_1^\epsilon$ be regular, with minimal polynomial 
\[m_x(t)=t^{d_1}\prod_{i=1}^{d_2}\varphi_i^{l_i}\prod_{i=1}^{d_3}\theta_i^{r_i},\]
a decomposition as in \eqref{equation:m_x-decompose}, with $\theta_i=\tau_i(t)\tau_i(-t)$ and $\tau_i(t)$ irreducible and coprime to $\tau_i(-t)$. 
\begin{enumerate}
\item If $d_1>0$, then there exists a short exact sequence 
\begin{equation}\label{equation:ses2}1\to \CC_{G_1^\epsilon}(x)\to \mcal{A}^\epsilon\times\prod_{i=1}^{d_2}\UU_1(\spring{\varphi_i^{l_i}})\times\prod_{i=1}^{d_3}\GL_1(\spring{\tau_i^{r_i}})\xrightarrow{\det} \set{\pm 1}\to 1. \end{equation}
where \[\mcal{A}^\epsilon=\set{\mbf{w}\in\CC_{\GL_{d_1+1}(\kk)}{(\nil)}\mid \mbf{w}^t\mbf{d}_\eta\mbf{w}=\mbf{d}_\eta},\]
with $\nil$ and $\mbf{d}_\eta$ the $(d_1+1)\times(d_1+1)$ matrices defined as in \eqref{equation:nil-defiso2n} and \eqref{equation:d-eta-defin}.
\item Otherwise, the group $\CC_{G_1^\epsilon}(x)$ is isomorphic to $\prod_{i=1}^{d_2}\UU_1(\spring{\varphi_i^{l_i}})\times\prod_{i=1}^{d_3}\GL_1(\spring{\tau_i^{r_i}})$.
\end{enumerate}
\end{lem}
\begin{proof}
Similarly to \autoref{lem:short-exact-sequence}, in order to prove the lemma, it is sufficient to compute the possible determinants of the middle term of \eqref{equation:ses2}. For the first assertion it is sufficient to verify that both $+1$ and $-1$ are obtained as determinant of elements from $\mcal{A}^\epsilon$, for which it is enough to consider block diagonal matrices of the form $\smat{1_{d_1}&0\\0&\pm 1}\in \mcal{A}^\epsilon$.

For the second assertion, we need to verify that any element $w\in \CC_{\GL_N(\kk)}(x)$ such that $w^\inv w=1$ has determinant $1$. Since any element of $\CC_{\GL_N(\kk)}(x)$ preserves the invariant factors of the decomposition of $V$ as a $\kk[x]$-module, it is sufficient to consider the following cases of the minimal polynomial of $x$.
\begin{list}{}{\setlength{\leftmargin}{0pt}\setlength{\itemsep}{2pt}}
\item \textit{Case 1}. Assume $m_x(t)=\varphi_i(t)^m$, with $\varphi_i\in\kk[t]$ irreducible and even and $m\in\dbN$. Let $x=s+h$ be the Jordan decomposition of $x$, with $s,h\in \mfr{g}_1^\epsilon$, $s$ semisimple, $h$ nilpotent and $[s,h]=0$. As $m_x(0)\ne 0$, by \autoref{propo:orbit-decomposition=so2n}.(2),  the space $V$ is cyclic as a $\kk[x]$-module and hence $\CC_{\matr_N(\kk)}(x)\simeq \kk[x]=\kk[s][h]\simeq \spring{\varphi_i}[u]/({u}^{m})$ (see \autoref{lem:unitary-of-kphi}). Let $\rho: \spring{\varphi_i}[u]/(u^m)\to \CC_{\matr_N(\kk)}(x)$ be a $\kk$-linear isomorphism. The $\kk$-linearity of $\rho$ and the nilpotency of $u$ imply that
 \[\det(\rho(\alpha_0+\alpha_1 u+\ldots+\alpha_{m-1} u^{m-1}))=\Nr_{\spring{\varphi_i}/\kk}(\alpha_0)^m.\]

Furthermore, the restriction of the involution $\inv$ to the image of $\rho$ induces a $\kk$-automorphism $\sigma_{\varphi_i^m}$ of $\spring
{\varphi_i^m}$ which acts on $\spring{\varphi_i}$ as the involution $\sigma_{\varphi_i}$, and maps $u$ to $-u$. Consequently, if $z\in \CC_{\GL_N(\kk)}(x)$ is given by $z=\rho(\alpha_0+\alpha_1 u+\ldots+\alpha_{m-1}u^{m-1})$ and satisfies $z^\inv z=1$ then necessarily $\Nr_{\spring{\varphi_i}/\msf{K}}(\alpha_0)=\sigma_{\varphi_i}(\alpha_0)\alpha_0=\rho^{-1}(z^\inv z)\mid_{u=0}=1$ and
\begin{align*}
\det(z)&=\det(\rho(\alpha_0+\alpha_1 u+\ldots+ \alpha_{m-1}u^{m-1}))\\&=\Nr_{\spring{\varphi_i}/\kk}(\alpha_0)^m=\left(\Nr_{\msf{K}/\kk}\circ\Nr_{\spring{\varphi_i}/\msf{K}}(\alpha_0)\right)^m=1.
\end{align*}
\item \textit{Case 2}. Assume $m_x(t)=\left(\tau_i(t)\cdot\tau_i(-t)\right)^{r}$, for $\tau_i(t)$ irreducible and coprime to $\tau(-t)$. In this case, by the cyclicity of the $\kk[x]$ module $V$, we have that $\CC_{\GL_N(\kk)}(x)\simeq\GL_1(\spring{\tau(t)^r})\times\GL_1(\spring{\tau(-t)^r})$. Moreover, the map $\inv$ restricts to the map $(\xi,\nu)\mapsto(\iota^{-1}(\xi),\iota(\nu))$, where $\iota:\spring{\tau(t)^r}\to\spring{\tau(-t)^r}$ is the isomorphism induced from $t\mapsto -t$. Furthermore, since $\iota$ is a ring-isomorphism which preserves $\kk$, we have that $\det(\iota(\xi))=\det(\xi)$ for all $\xi\in\spring{\tau(t)^r}$. In particular, if $(\xi,\nu)^\inv (\xi,\nu)=1$ then $\nu=\iota(\xi)^{-1}$ and hence, $\det((\xi,\nu))=\det(\xi)\cdot\det(\xi)^{-1}=1$.
\end{list}
\end{proof}

\begin{propo}\label{corol:size-of-centralizer-so2n}
Let $x\in\mfr{g}_1^\pm$ be regular with minimal polynomial $m_x(t)$. Let $c_x$ denote the characteristic polynomial of $x$, i.e. $c_x=m_x$ if $x$ is non-singular, and $c_x(t)=t\cdot m_x(t)$ otherwise. Let $\ttau(c_x)=(r(c_x),\A(c_x),\B(c_x))\in\mcal{X}_n$ be the type of $c_x$ (see \autoref{defi:poly-type}). Then 
\[\abs{\CC_{G_1^\epsilon}(x)}=2^{\nu}q^n\prod_{d,e}\left(1+q^{-d}\right)^{\A_{d,e}(m_x)}\cdot\left(1-q^{-d}\right)^{\B_{d,e}(m_x)},\]
where $\epsilon\in\set{\pm }$ and $\nu=1$ if $r(m_x)>0$ and $0$ otherwise.
\end{propo}

\begin{proof}
In the case where $x$ is non-singular the assertion follows verbatim as in \autoref{corol:centralizer-size-sp2nso2n+1}. Otherwise, if $x$ is singular, by decomposing $x$ into its primary rational canonical forms, it is sufficient to consider the case where $x$ is a regular nilpotent element, with minimal polynomial $m_x(t)=t^{2n-1}$, and show that $\abs{\CC_{G_1}(x)}=2q^n$.

Without loss of generality, we fix the basis $\mcal{E}$ of \autoref{lem:nilpotent-reg-element-so2n}, with respect to which the ambient symmetric form $B^\epsilon$ is represented by the matrix $\mbf{d}=\mbf{d}_\eta$, for some $\eta\in\kk^\times$, and $x$ is represented by the matrix $\nil$. Let $\mcal{A}^\epsilon=\set{z\in\CC_{\GL_{N}(\kk)}(\nil)\mid z^t\mbf{d}z=\mbf{d}}$, as in \autoref{lem:short-exact-sequence-so2n}. Let $\mcal{N}\subseteq \mcal{A}^\epsilon$ be the subgroup consisting of elements of the form \[\mfr{X}(\xi)=\mat{1&&&2\eta \xi^2&2\xi\\&1\\&&\ddots\\&&&1\\&&&2\eta \xi&1}\quad (\xi\in \kk).\]
Note that $\mfr{X}$ defines a one-parameter subgroup of $\mcal{A}^\epsilon$ of order $\abs{\kk}=q$. Additionally, $\mcal{N}=\im(\mfr{X})$ is the image under the Cayley map of the Lie-ideal generated by elements of the form $\cma(0_{N-1},\mbf{u},\mbf{v},0)\in \mfr{g}_1$, and hence is normal in $\mcal{A}^\epsilon$. 

Let $\mcal{H}\subseteq\mcal{A}^\epsilon$ be the subgroup of block diagonal matrices $\cma{(\mbf{A},0,0,r)}$. Note that, by~\eqref{equation:cma} and the assumption $\cma(\mbf{A},0,0,r)^\inv\cma(\mbf{A},0,0,r)=1_N$, we have that $\mbf{A}^\invv\mbf{A}=1_{N-1}$ and $r^2=1$. Additionally, since $\mbf{A}$ commutes with the restriction of $\nil$ to the subspace spanned by the first $N-1$ elements of $\mcal{E}$, we have that $\abs{\mcal{H}}=\abs{\UU_1(\spring{t^{2n-1}})\times\set{\pm 1}}=4q^{n-1}$ (by the first assertion in the proof of \autoref{corol:centralizer-size-sp2nso2n+1}).

Given an arbitrary element $\cma(\mbf{A},\mbf{v},\mbf{u},r)\in \mcal{A}^\epsilon$, it holds that $\mbf{A}$ must be invertible, and that $\mbf{v}=\gamma\mbf{d}\mbf{u}$ for some $\gamma\in\kk$. In particular, $\mbf{v}=0$ if and only if $\mbf{u}=0$. It follows from this, and by direct computation, that 
\[\mfr{X}\left(-\frac{v_1}{a_{1,1}\eta}\right)\mat{\mbf{A}&\mbf{v}\\\mbf{u}^t&r}\in\mcal{H},\]
where $v_1$ is the first entry of $\mbf{v}$, and $a_{1,1}$ is the $(1,1)$-th entry of $\mbf{A}$. Therefore, we have that $\mcal{A}^\epsilon=\mcal{H}\cdot \mcal{N}$ and hence, as $\mcal{H}\cap\mcal{N}=\set{1}$, that \[\abs{\mcal{A}^\epsilon/\mcal{N}}=\abs{\mcal{H}}=4q^{n-1}.\]
To conclude, we have that $\abs{\mcal{A}^\epsilon}=4q^{n}$, and the result follows from \autoref{lem:short-exact-sequence-so2n}.
\end{proof}
The final assertion of \autoref{theo:orbits-so2n} follows from \autoref{corol:size-of-centralizer-so2n}.

\section*{Acknowledgements} This paper is part of the author's doctoral thesis. I wish to thank Uri Onn for guiding and advising this research. I also wish to thank Alexander Stasinski for carefully reading through a preliminary version of this article and offering some essential remarks. Finally, I wish to acknowledge the valuable input offered by the anonymous referee.

The present research was supported by the Israel Science Foundation (ISF) grant 1862.

\bibliographystyle{acm}
\bibliography{REGbib}

\begin{thebibliography}{10}

\bibitem{SGA3}
{\sc Artin, M., Bertin, J.-E., Demazure, M., Grothendieck, A., Gabriel, P.,
  Raynaud, M., and Serre, J.-P.}
\newblock {\em Sch\'emas en groupes (SGA 3)}.
\newblock S\'eminaire de G\'eom\'etrie Alg\'ebrique de l'Institut des Hautes
  \'Etudes Scientifiques. Institut des Hautes \'Etudes Scientifiques, Paris,
  1963/1966.

\bibitem{AKOVA2}
{\sc Avni, N., Klopsch, B., Onn, U., and Voll, C.}
\newblock Similarity classes of integral p-adic matrices and representation
  zeta functions of groups of type {$\mathsf{A}_2$}.
\newblock {\em Proceedings of the London Mathematical Society 112}, 2 (2016),
  267.

\bibitem{Begueri-Dualite}
{\sc B\'{e}gueri, L.}
\newblock Dualit\'{e} sur un corps local \`a corps r\'{e}siduel
  alg\'{e}briquement clos.
\newblock {\em M\'{e}m. Soc. Math. France (N.S.)}, 4 (1980/81), 121.

\bibitem{BCM}
{\sc Bia\l~ynicki Birula, A., Carrell, J.~B., and McGovern, W.~M.}
\newblock {\em Algebraic quotients. {T}orus actions and cohomology. {T}he
  adjoint representation and the adjoint action}, vol.~131 of {\em
  Encyclopaedia of Mathematical Sciences}.
\newblock Springer-Verlag, Berlin, 2002.
\newblock Invariant Theory and Algebraic Transformation Groups, II.

\bibitem{Borel}
{\sc Borel, A.}
\newblock {\em Linear Algebraic Groups}.
\newblock Graduate Texts in Mathematics. Springer New York, 1991.

\bibitem{BLR-Neron}
{\sc Bosch, S., L\"utkebohmert, W., and Raynaud, M.}
\newblock {\em N\'eron models}, vol.~21 of {\em Ergebnisse der Mathematik und
  ihrer Grenzgebiete (3) [Results in Mathematics and Related Areas (3)]}.
\newblock Springer-Verlag, Berlin, 1990.

\bibitem{BourbakiLieCh1}
{\sc Bourbaki, N.}
\newblock {\em \'El\'ements de math\'ematique. {F}asc. {XXVI}. {G}roupes et
  alg\`ebres de {L}ie. {C}hapitre {I}: {A}lg\`ebres de {L}ie}.
\newblock Seconde \'edition. Actualit\'es Scientifiques et Industrielles, No.
  1285. Hermann, Paris, 1971.

\bibitem{Britnell-Conformal}
{\sc Britnell, J.~R.}
\newblock Cyclic, separable and semisimple transformations in the finite
  conformal groups.
\newblock {\em J. Group Theory 9}, 5 (2006), 571--601.

\bibitem{BushnellFroelich}
{\sc Bushnell, C., and Fr{\"o}hlich, A.}
\newblock {\em Gauss sums and p-adic division algebras}.
\newblock Lecture notes in mathematics. Springer, 1983.

\bibitem{Cayley}
{\sc Cayley, A.}
\newblock Sur quelques propri\'et\'es des d\'eterminants gauches.
\newblock {\em J. Reine Angew. Math. 32\/} (1846), 119--123.

\bibitem{DemazureGabriel}
{\sc Demazure, M., and Gabriel, P.}
\newblock {\em Introduction to algebraic geometry and algebraic groups},
  vol.~39 of {\em North-Holland Mathematics Studies}.
\newblock North-Holland Publishing Co., Amsterdam-New York, 1980.
\newblock Translated from the French by J. Bell.

\bibitem{EGA}
{\sc Dieudonn{\'e}, J., and Grothendieck, A.}
\newblock \'{E}l\'ements de g\'eom\'etrie alg\'ebrique.
\newblock {\em Inst. Hautes \'Etudes Sci. Publ. Math. 4, 8, 11, 17, 20, 24, 28,
  32\/} (1961--1967).

\bibitem{DummitFoote}
{\sc Dummit, D.~S., and Foote, R.~M.}
\newblock {\em Abstract Algebra - 3rd Edition}, 3~ed.
\newblock John Wiley and Sons, Inc., 2004.

\bibitem{FulmanGuralnick}
{\sc Fulman, J., and Guralnick, R.}
\newblock The number of regular semisimple conjugacy classes in the finite
  classical groups.
\newblock {\em Linear Algebra Appl. 439}, 2 (2013), 488--503.

\bibitem{FNP-Generating}
{\sc Fulman, J., Neumann, P.~M., and Praeger, C.~E.}
\newblock A generating function approach to the enumeration of matrices in
  classical groups over finite fields.
\newblock {\em Mem. Amer. Math. Soc. 176}, 830 (2005), vi+90.

\bibitem{GreenbergI}
{\sc Greenberg, M.~J.}
\newblock Schemata over local rings.
\newblock {\em Annals of Mathematics 73}, 3 (1961), 624--648.

\bibitem{GreenbergII}
{\sc Greenberg, M.~J.}
\newblock Schemata over local rings: {II}.
\newblock {\em Annals of Mathematics 78}, 2 (1963), 256--266.

\bibitem{Hill}
{\sc Hill, G.}
\newblock Regular elements and regular characters of ${\GL_n(\OO)}$.
\newblock {\em Journal of Algebra 174\/} (1995), 610--635.

\bibitem{HumphreysLie}
{\sc Humphreys, J.~E.}
\newblock {\em Introduction to {L}ie algebras and representation theory},
  vol.~9 of {\em Graduate Texts in Mathematics}.
\newblock Springer-Verlag, New York-Berlin, 1978.
\newblock Second printing, revised.

\bibitem{HumphreysConj}
{\sc Humphreys, J.~E.}
\newblock {\em Conjugacy classes in semisimple algebraic groups}, reprint~ed.
\newblock Mathematical Surveys and Monographs 043. American Mathematical
  Society, 1995.

\bibitem{Isaacs}
{\sc Isaacs, I.~M.}
\newblock {\em Character theory of finite groups}.
\newblock AMS Chelsea Publishing, Providence, RI, 2006.
\newblock Corrected reprint of the 1976 original [Academic Press, New York;
  MR0460423].

\bibitem{KOS}
{\sc Krakovski, R., Onn, U., and Singla, P.}
\newblock Regular characters of groups of type {$\mathsf{A}_n$} over discrete
  valuation rings.
\newblock {\em J. Algebra 496\/} (2018), 116--137.

\bibitem{Lang}
{\sc Lang, S.}
\newblock Algebraic groups over finite fields.
\newblock {\em Amer. J. Math. 78\/} (1956), 555--563.

\bibitem{LPR-Cayley}
{\sc Lemire, N., Popov, V.~L., and Reichstein, Z.}
\newblock Cayley groups.
\newblock {\em J. Amer. Math. Soc. 19}, 4 (2006), 921--967.

\bibitem{McNinch-Faithful}
{\sc McNinch, G.~J.}
\newblock Faithful representations of {$\rm SL_2$} over truncated {W}itt
  vectors.
\newblock {\em J. Algebra 265}, 2 (2003), 606--618.

\bibitem{McNinch}
{\sc McNinch, G.~J.}
\newblock The centralizer of a nilpotent section.
\newblock {\em Nagoya Math. J. 190\/} (2008), 129--181.

\bibitem{NeumannPraeger-Cyclic}
{\sc Neumann, P.~M., and Praeger, C.~E.}
\newblock Cyclic matrices over finite fields.
\newblock {\em J. London Math. Soc. (2) 52}, 2 (1995), 263--284.

\bibitem{NeumannPraeger-Classical}
{\sc Neumann, P.~M., and Praeger, C.~E.}
\newblock Cyclic matrices in classical groups over finite fields.
\newblock {\em J. Algebra 234}, 2 (2000), 367--418.
\newblock Special issue in honor of Helmut Wielandt.

\bibitem{SerreLocal}
{\sc Serre, J.}
\newblock {\em Local Fields}.
\newblock Graduate Texts in Mathematics. Springer New York, 1995.

\bibitem{Shechter}
{\sc Shechter, S.}
\newblock Characters of the norm-one units of local division algebras of prime
  degree.
\newblock {\em J. Algebra 474\/} (2017), 134--165.

\bibitem{Shintani}
{\sc Shintani, T.}
\newblock On certain square integrable irreducible unitary representations of
  some $\mfr{p}$-adic linear groups.
\newblock {\em Proceedings of the Japan Academy, Series A, Mathematical
  Sciences 44}, 1 (1968), 1--3.

\bibitem{SteinbergSpringer}
{\sc Springer, T., and Steinberg, R.}
\newblock Conjugacy classes.
\newblock In {\em Seminar on Algebraic Groups and Related Finite Groups},
  vol.~131 of {\em Lecture Notes in Mathematics}. Springer Berlin Heidelberg,
  1970, pp.~167--266.

\bibitem{Stasinski}
{\sc Stasinski, A.}
\newblock Reductive group schemes, the {G}reenberg functor, and associated
  algebraic groups.
\newblock {\em Journal of Pure and Applied Algebra 216}, 5 (2012), 1092 --
  1101.

\bibitem{StasinskiOverview}
{\sc Stasinski, A.}
\newblock Representations of {$\mathrm{GL}_N$} over finite local principal
  ideal rings - an overview.
\newblock In {\em Around Langlands Correspondences}, vol.~691 of {\em Contemp.
  math.} American Mathematical Society, January 2017, pp.~336--358.

\bibitem{StasinskiStevens}
{\sc Stasinski, A., and Stevens, S.}
\newblock The regular representations of {$\mathrm{GL}_{N}$} over finite local
  principal ideal rings, 2016.

\bibitem{SteinbergConjugacy}
{\sc Steinberg, R.}
\newblock {\em Conjugacy classes in algebraic groups}.
\newblock Lecture notes in mathematics. Springer, 1974.

\bibitem{TakaseGLn}
{\sc Takase, K.}
\newblock Regular characters of {$GL_n(O)$} and {W}eil representations over
  finite fields.
\newblock {\em J. Algebra 449\/} (2016), 184--213.

\bibitem{TakaseGeneral}
{\sc {Takase}, K.}
\newblock {Regular irreducible characters of a hyperspecial compact group}.
\newblock {\em ArXiv e-prints\/} (Jan. 2017).

\bibitem{WallIsometries}
{\sc Wall, G.~E.}
\newblock On the conjugacy classes in the unitary, symplectic and orthogonal
  groups.
\newblock {\em Journal of the Australian Mathematical Society 3\/} (2 1963),
  1--62.

\bibitem{waterhouse}
{\sc Waterhouse, W.~C.}
\newblock {\em Introduction to Affine Group Schemes}.
\newblock Graduate Texts in Mathematic. Springer, New York, Heidelberg, Berlin,
  1979.

\bibitem{WeilAlgebras}
{\sc Weil, A.}
\newblock Algebras with involutions and the classical groups.
\newblock {\em J. Indian Math. Soc. 24\/} (1960), 589--623.

\bibitem{WilsonFinite}
{\sc Wilson, R.}
\newblock {\em The Finite Simple Groups}.
\newblock Graduate Texts in Mathematics. Springer London, 2009.

\end{thebibliography}

\end{document}